\documentclass[11pt]{article}
\usepackage{cite}
\usepackage{enumitem}
 \usepackage{amsfonts}
\usepackage{amsmath}
\usepackage{amsthm}
\usepackage{amssymb}
\usepackage{amscd}

\DeclareMathOperator{\dist}{dist}
\DeclareMathOperator{\intr}{int}

\DeclareMathOperator{\cl}{cl}
\DeclareMathOperator{\bd}{bd}
\DeclareMathOperator{\co}{co}

\newtheorem{definition}{Definition}[section]
\newtheorem{theorem}{Theorem}[section]
\newtheorem{corollary}{Corollary}[section]
\newtheorem{lemma}{Lemma}[section]
\newtheorem{proposition}{Proposition}[section]
\newtheorem{remark}{Remark}[section]

\newtheorem{counterexample}{Counterexample}[section]

\begin{document}
\title{{ Directional Derivatives and Error Bounds of Merit Functions in Vector Optimization }}
\author{{Yu Han \footnote{Corresponding author,  E-mail: hanyumath@163.com}}\\
{\small\it School of Statistics and Data Science, Jiangxi University of Finance and Economics, }\\
{\small\it  Nanchang, Jiangxi 330013,  China} }
\date{}
\maketitle
\vspace*{-9mm}
\begin{center}
\begin{minipage}{6.2in}
{\bf Abstract.}    This paper presents a comprehensive analysis of directional derivatives and error bounds for the merit function $\theta(x)=\sup_{a\in A}\bigl(-\Delta_C(F(x)-F(a))\bigr)$ associated with the vector optimization problem $\operatorname{Min}_C\{F(x):x\in A\}$, where $\Delta_C$ is the  oriented distance function.   We first prove that $\theta$ is concave and Lipschitz continuous on the whole space and derive its dual representation via the weak$^*$ compact convex set $K=\overline{\operatorname{co}}^{w^*}(S(C^+))$.  
At a weakly efficient solution $\bar x$, we obtain the explicit formula
$\theta'(\bar x;d)=\min_{y^*\in W(\bar x)}\langle y^*,F(d)\rangle$ with \(W(\bar{x})=\{y^*\in K:F^*y^*\in -N_A(\bar{x})\}\), characterize the zero-directional-derivative cone, and prove that, under a local error bound condition, the tangent cone to the solution set is
$T_{E_w}(\bar x)=T_A(\bar x)\cap T_{\widehat A}(\bar x)=\{d\in T_A(\bar x):\theta'(\bar x;d)=0\}$.  
We establish the equivalence of thirteen distinct global error bound conditions, including characterizations via linear regularity, the global slope, an asymptotic   condition, and perturbation stability.
A central result shows that the global error bound property for \(\theta\) on the feasible set \(A\) is   characterized by a uniform negativity condition on the unit-sphere minimal directional derivative, namely \(\sup_{x \in A \setminus E_w} \varphi(x) < 0\).
We also determine the optimal local error bound constant precisely as $1/\varphi(\bar x)$ when $\varphi(\bar x)>0$, and provide a counterexample demonstrating that an additional directional condition is essential when $\varphi(\bar x)=0$.  
These results provide a complete bridge between the directional derivative of the merit function and the geometry of the solution set, offering fundamental tools for the convergence analysis of algorithms in vector optimization.

\end{minipage}
\end{center}
\section{Introduction}
Vector optimization, which deals with the problem of optimizing multiple conflicting objectives simultaneously, has emerged as a fundamental framework in modern optimization theory with wide-ranging applications in economics, engineering, and decision sciences \cite{Ehrgott, Jahn, Luc}. Unlike scalar optimization, where a total order provides a natural notion of optimality, vector optimization relies on partial orders induced by convex cones, leading to the concepts of Pareto efficiency and weak efficiency. The mathematical analysis of such problems requires sophisticated tools from convex analysis, variational analysis, and nonlinear functional analysis to handle the inherent nonsmoothness and set-valued nature of the solution maps.

A powerful approach to analyzing vector optimization problems is the construction of merit functions---scalar-valued functions whose global minimizers coincide with the (weakly) efficient solutions of the original vector problem. Merit functions serve as indispensable tools in optimization, providing a bridge between the multiobjective nature of the problem and classical scalar optimization theory. They play a central role in the design and convergence analysis of numerical algorithms, and are also instrumental in deriving optimality conditions \cite{Fukushima, Hearn, FacchineiPang}.  In the context of vector optimization, the construction of suitable merit functions has attracted considerable attention in recent years. 

The Hiriart-Urruty oriented distance function, introduced by Hiriart-Urruty \cite{Hiriart1, Hiriart2}, provides a remarkably elegant framework for this purpose. This function, defined as \(\Delta_D(y) := d_D(y) - d_{Y\setminus D}(y)\), encodes in a single real value the position of a point relative to a set \(D\), with its sign indicating the location of a point (interior, boundary, or exterior) relative to \(D\). Its fundamental properties---in particular its \(1\)-Lipschitz continuity and its convexity when \(D\) is convex---make it an ideal scalarization device for set-valued problems. Zaffaroni \cite{Zaffaroni2003} further developed the theory of oriented distance functions and their applications in efficiency theory, establishing connections to various notions of minimality. The dual representation of \(\Delta_D\) via support functions, as established by Liu, Ng, and Yang \cite{LiuNgYang2009}   reveals its deep connections to the dual cone and the geometry of the underlying space.

The seminal work of Liu, Ng, and Yang \cite{LiuNgYang2009} introduced the merit function
\[
\theta(x)=\sup_{a\in A}\bigl(-\Delta_C(F(x)-F(a))\bigr)
\]
for the vector optimization problem \(\operatorname{Min}_C\{F(x):x\in A\}\), where \(F:X\to Y\) is a linear continuous mapping between Banach spaces and \(C\subset Y\) is a closed convex cone with nonempty interior. They established the fundamental properties of this merit function: \(\theta\) is nonnegative on the feasible set \(A\), its zero set coincides with the weakly efficient solution set \(E_w\), and it satisfies a crucial two-sided estimate \(r\,d_{\widehat A}(x)\le \theta(x)\le \|F\|\,d_{\widehat A}(x)\) for all \(x\in A\), where \(\widehat A = X\setminus(A+\operatorname{int}C_X)\) and \(C_X=F^{-1}(C)\). These results already demonstrate the potential of \(\theta\) as a quantitative measure of the distance to the solution set. However, several critical aspects of the theory remained unexplored in their work: the directional differentiability and explicit directional derivative formulas, the   characterization of error bounds via directional derivative conditions, the precise identification of optimal error bound constants, and the   geometric description of the solution set's tangent structure through the behavior of \(\theta\).

Error bound theory, originating from the classical work of Hoffman \cite{Hoffman} on linear inequalities, has evolved into a rich and mature field with profound implications for the convergence analysis of optimization algorithms, sensitivity analysis, and the study of metric regularity.   A comprehensive theory of error bounds for lower semicontinuous functions was developed by Az\'e and Corvellec \cite{AC2004}, who characterized error bounds using the strong slope, and by Ng and Zheng \cite{NgZheng}, who established connections between error bounds and directional derivatives for convex functions. The general theory of error bounds and metric subregularity has been systematically developed by Kruger \cite{Kruger}, who provided characterizations in terms of primal and subdifferential slopes.  
Ng and Yang \cite{NgYang2002} studied error bounds for abstract linear inequality systems in Banach spaces, while  Klatte and Li   \cite{KL1999} characterized global error bounds for convex inequalities via asymptotic constraint qualifications.
More recently, Kruger, Ngai, and Théra \cite{KNT2010} and Kruger, López, and Théra \cite{KLT2018} developed stability theory for error bounds under  $\epsilon$-perturbations, characterizing the radius via boundary subdifferential slopes.

In the context of vector optimization, the study of error bounds for merit functions has been advanced through several distinct lines of research. Liu, Ng, and Yang \cite{LiuNgYang2009} established a foundational framework for linear vector optimization problems, demonstrating that global error bounds for their merit function are equivalent to the linear regularity of the feasible set and the solution set. For nonsmooth convex problems, Dutta, Kesarwani, and Gupta \cite{DKG2017} developed gap functions based on subdifferentials and provided error bound estimates under strong convexity assumptions. More recently, Tanabe, Fukuda, and Yamashita \cite{TanabeNew} proposed new merit functions
for multiobjective optimization with lower semicontinuous objectives, convex objectives, and composite objectives, and provided sufficient conditions for boundedness of the level set and error bounds. Despite these significant contributions, a comprehensive and unified theory of error bounds for the vector optimization merit function \(\theta\)---one that   characterizes the global error bound through directional derivative conditions, establishes optimal constants, and reveals the complete equivalence with various stability properties---has remained lacking.

The present paper aims to fill this gap by providing a complete theory of directional derivatives and error bounds for the merit function \(\theta\). Our contributions are fourfold.

First, we establish a systematic theory of directional derivatives for \(\theta\). We prove that \(\theta\) is concave and Lipschitz continuous on the entire space \(X\) (Theorem \ref{YL:thetaprop2}), and derive its dual representation
\[
\theta(x)=\min_{v\in K}\left(\langle v,F(x)\rangle-\inf_{a\in A}\langle v,F(a)\rangle\right),
\]
where \(K=\overline{\operatorname{co}}^{w^*}(S(C^+))\)  (Theorem \ref{YDOBS}). This representation is fundamental for all subsequent developments. For a weakly efficient solution \(\bar x\in E_w\), we obtain the explicit directional derivative formula
\[
\theta'(\bar x;d)=\min_{y^*\in W(\bar x)}\langle y^*,F(d)\rangle,
\quad
W(\bar x)=\{y^*\in K:F^*y^*\in -N_A(\bar x)\},
\]
(Theorem  \ref{TTXEQ}), and characterize the zero-directional-derivative cone
\[
D(\bar x)=F^{-1}\bigl(\operatorname{bd}(W(\bar x)^+)\bigr)
\]
as the preimage of the boundary of the polar cone of \(W(\bar x)\). This reveals the precise link between the directional behavior of \(\theta\) and the geometry of the normal cone to the feasible set and the dual cone structure. We further establish that \(\theta'(\bar x;d)\ge 0\) for all \(d\in T_A(\bar x)\) exactly characterizes weak efficiency (Theorem \ref{ZSQE}), and prove that under a local error bound condition, the tangent cone to the weakly efficient solution set admits the decomposition
\[
T_{E_w}(\bar x)=T_A(\bar x)\cap T_{\widehat A}(\bar x)=\{d\in T_A(\bar x):\theta'(\bar x;d)=0\}
\]
(Theorem  \ref{ZAQW}). This geometric characterization directly links the local geometry of the Pareto set to the first-order behavior of the merit function.

Second, we develop a comprehensive equivalence theory for the global error bound property. We prove that the following thirteen conditions are equivalent (Theorem \ref{XLWCJK}): (i) the global error bound for \(\theta\) on \(A\); (ii) linear regularity of the pair \(\{A,\widehat A\}\); (iii) the error bound for the distance function \(d_{\widehat A}\) on \(A\); (iv) a uniform descent condition; (v) a perturbation stability condition; (vi) an extended-domain error bound; (vii) an error bound via the infimal convolution of \(\theta\) and the distance to \(A\); (viii) a descent condition; (ix) a uniform lower bound on the global slope; (x) an error bound for the image under \(F\); (xi) an asymptotic   condition; (xii) Hausdorff stability of sublevel sets; and (xiii) an inverse-sublevel characterization. This unified framework consolidates and substantially extends previous results in the literature, including those of Liu, Ng, and Yang \cite{LiuNgYang2009}, Ng and Zheng \cite{NgZheng}, and the general error bound theory of Kruger \cite{Kruger}, Az\'e and Corvellec \cite{AC2004}, and others. Notably, conditions (v), (xii), and (xiii) are new in the vector optimization context and connect the error bound theory to the perturbation stability framework of Kruger, L\'opez, and Th\'era \cite{KLT2018}.

Third, we provide a sufficient condition for the global error bound in terms of directional derivatives. For \(x\in A\), define the unit-sphere minimal directional derivative
\[
\varphi(x):=\inf\{\theta'(x;d):d\in T_A(x),\,\|d\|=1\}.
\]
We prove (Theorem \ref{TXEvQ}) that if
\[
\sup_{x\in A\setminus E_w}\varphi(x)<0,
\]
then \(\theta\) has a global error bound on \(A\). This reveals that a uniform negative directional derivative away from the solution set is sufficient for a global estimate of the distance to \(E_w\).  Moreover, when \(\bar x\in E_w\) and \(\varphi(\bar x)>0\), we identify the optimal local error bound constant precisely as
\[
\tau^*(\bar x)=\frac{1}{\varphi(\bar x)}<+\infty
\]
(Theorem \ref{TYXWs}). When \(\varphi(\bar x)=0\), we provide a counterexample demonstrating that an additional directional condition is essential for the existence of a local error bound---the condition \(\varphi(\bar x)=0\) alone does not preclude a local error bound. This counterexample also reveals the subtlety of the local error bound theory and highlights the importance of the tangent cone \(T_{E_w}(\bar x)\) in determining whether a zero directional derivative direction can escape the solution set.

Fourth, we prove a quantitative lower bound on the directional derivative under local error bound and local convexity assumptions (Theorem \ref{TaqXEQ}):
\[
\theta'(\bar x;d)\ge \frac{1}{\tau}\operatorname{dist}\bigl(d,T_{E_w}(\bar x)\bigr),\qquad d\in T_A(\bar x),
\]
when \(\theta\) admits a local error bound at \(\bar x\) with constant \(\tau\). This inequality   provides a rigorous foundation for the convergence analysis of algorithms that exploit the directional derivative structure of the merit function. It also establishes a reciprocal relationship between the error bound constant and the rate at which the directional derivative grows with the distance to the tangent cone of the solution set.

The results of this paper have immediate implications for algorithm design in vector optimization. The explicit directional derivative formulas enable the development of first-order descent methods, while the error bound characterizations provide theoretical guarantees for convergence rates.  In particular, the observation that the uniform negativity of \(\varphi\) is sufficient for a global error bound suggests practical stopping criteria and a way to estimate error bound constants directly from first-order information. The tangent cone decomposition (Theorem \ref{ZAQW}) provides a geometric characterization of the solution set that can guide the design of feasible-direction methods and proximal-type algorithms. Recent work on multiobjective descent methods by Fliege, Gra\~na Drummond, and Svaiter \cite{Fliege}, and Tanabe, Fukuda, and Yamashita \cite{TanabeProx}  can potentially benefit from the error bound theory developed here, as the convergence rates of many algorithms depend crucially on the presence of a global error bound for the underlying merit function. Huang, Jiao, Kim, and Yin \cite{HJKY2026} have recently extended merit function approaches to vector polynomial optimization over LMI constraints, demonstrating the broad applicability of this framework.

 The remainder of the paper is organized as follows. Section 2 collects the necessary preliminaries, including basic notation, fundamental properties of the  oriented distance function, and the cone-theoretic framework underlying the vector optimization problem.
  Section 3 is devoted to the analytic properties of the merit function \(\theta\), where we establish its concavity and Lipschitz continuity, derive a dual representation in terms of the weak\(^*\)-compact convex set \(K = \overline{\operatorname{co}}^{w^*}(S(C^+))\), and characterize its superdifferential at weakly efficient points. Section~4 presents the main error bound results, including the equivalence of thirteen characterizations and the relationships among them. Section~5 develops the theory of directional derivatives of \(\theta\), including explicit formulas, tangent cone characterizations, and the relationship to distance to the tangent cone. Section~6 characterizes error bounds in terms of directional derivatives and establishes the optimal local error bound constant. A counterexample illustrates the sharpness of the theoretical results. Section~7 concludes the paper with a summary and discussion of future directions.

\section{Preliminaries}
\subsection{Basic notation}

Throughout this paper, unless otherwise specified,
let $X$ and $Y$ be real Banach spaces with topological duals $X^*$ and $Y^*$, respectively.  Let $A \subseteq X$ be a nonempty closed convex set, $C \subseteq Y$ be a closed convex cone with nonempty interior ($\intr C \neq \emptyset$), and $F: X \to Y$ be a continuous linear map. Consider the vector optimization problem
\begin{equation}\label{eq:VP}
	\operatorname{Min}_C \; F(x) \quad \text{subject to} \quad x \in A,
\end{equation}
where $\operatorname{Min}_C$ denotes minimization with respect to the partial order $\leq_C$ induced by $C$, i.e., $y_1 \leq_C y_2$ if and only if $y_2 - y_1 \in C$. The strict partial order $\ll_C$ is defined by $y_1 \ll_C y_2$ if and only if $y_2 - y_1 \in \intr C$.

A feasible point $\bar{x} \in A$ is called a \emph{weakly efficient solution} of \eqref{eq:VP} if there exists no $x \in A$ such that $F(x) \ll_C F(\bar{x})$, which is equivalent to
\[
F(\bar{x}) \notin F(A) + \intr C.
\]
The set of all weakly efficient solutions is denoted by $E_w(\text{VP})$ or simply $E_w$. Its image $F(E_w)$ consists of the \emph{weakly efficient points} of $F(A)$, denoted by $\operatorname{WMin}(F(A), C)$.

For a subset $A \subseteq X$ with respect to a closed convex cone $P \subseteq X$ ($\intr P \neq \emptyset$), the set of \emph{weakly efficient points} of $A$ is
\[
\operatorname{WMin}(A, P) := \{a \in A : a \notin A + \intr P\}.
\]

A set $A$ is said to have the  weak domination property  (WDP) with respect to $P$, if for every $z \in A \setminus \operatorname{WMin}(A, P)$, there exists $z' \in \operatorname{WMin}(A, P)$ such that $z' \leq_P z$.

\begin{remark} \label{RngXEQ}     It is easy to verify that  $A$ has  the  weak domination property with respect to $P$ if and only if for any $z \in A \setminus \operatorname{WMin}(A, P)$, there exists $z' \in \operatorname{WMin}(A, P)$ such that $z'  \ll_P  z$.
\end{remark}

The canonical bilinear pairing is denoted by $\langle \cdot, \cdot \rangle$: for $x^* \in X^*$ and $x \in X$, $\langle x^*, x \rangle := x^*(x)$. The norm on $X$ (and other spaces) is denoted by $\|\cdot\|$. The closed unit ball and unit sphere of $X$ are
\[
B_X := \{x \in X : \|x\| \leq 1\}, \quad S_X := \{x \in X : \|x\| = 1\}.
\]

$B(x_0,r)$ denotes the open ball with center $x_0$ and radius $r > 0$. For a subset $D \subseteq X$, its interior, boundary, closure, convex hull  and  closed convex hull    are denoted by $\intr D$, $\bd D$, $\cl D$, $\co D$ and $\overline{\co} D$,  respectively.

For a nonempty set \(D \subseteq X\), the distance from a point \(x \in X\) to the set \(D\) is defined as
\[
d_D(x) := \inf_{y \in D} \|x - y\|,
\]
with the convention that \(d_\emptyset(\cdot) = +\infty\). Occasionally, we use \(d(x,D)\) to denote the distance from a point \(x \in X\) to the set \(D\).
The distance between two nonempty sets \(Q \subseteq X\) and \(D \subseteq X\) is defined by
\[
d(Q,D) = \inf \{ \|x - y\| : x \in Q,\; y \in D \}.
\]
The Hausdorff distance between $Q$ and $D$ is defined by
$$d_H \left( {Q,D} \right): = \max \left\{ {g\left( {Q,D} \right),g\left( {D,Q} \right)} \right\},$$
where $g\left( {Q,D} \right): = \mathop {\sup }\limits_{x \in Q} d\left( {x,D} \right)$.

The  indicator function  of a  nonempty set   $D$ is
\[
\delta_D(x) := \begin{cases}
	0, & x \in D, \\
	+\infty, & x \notin D.
\end{cases}
\]

For  $ x^* \in X^*$,     the support function   of a  nonempty set   $D$ is
$$\sigma_D(x^*)=\sup_{y\in D}  \;  \langle x^*,y\rangle.$$

The \emph{projection} of a point $x$ onto a closed convex set $D$ is the (possibly empty) set
\[
P_D(x) := \{\bar{x} \in D : \|x - \bar{x}\| = d_D(x)\}.
\]
When $X$ is a Hilbert space, $P_D(x)$ is a singleton for every $x \in X$.

The \emph{positive polar cone} (or dual cone) of a set $Q \subseteq X$ is
\[
Q^+ := \{x^* \in X^* : \langle x^*, x \rangle \geq 0, \; \forall x \in Q\},
\]
and the \emph{negative polar cone} is $Q^\ominus := -Q^+ = \{x^* \in X^* : \langle x^*, x \rangle \leq 0, \; \forall x \in Q\}$.    The unit sphere of \(Q^+\) is denoted by \(S(Q^+)\), i.e., \(S(Q^+) = \{ x^* \in Q^+ : \|x^*\| = 1 \}\).

\begin{definition}  
 For a nonempty set $D \subseteq X$ and $x \in D$, the \emph{tangent cone} and \emph{normal cone} of  $D$ at  $x$ are defined, respectively, by  
 \begin{align*}
 	T_D(x) &:= \bigl\{ d \in X : \exists\, t_k \downarrow 0,\; \exists\, d_k \to d \text{ with } x + t_k d_k \in D \bigr\},\\
 	N_D(x) &:= \{x^* \in X^* : \langle x^*, h \rangle \leq 0, \; \forall h \in T_D(x)\}.
 \end{align*}
  \end{definition}

\begin{remark} \label{RtDFxl}    \cite{Zalinescu2002, AE1984}    Let  $D \subseteq X$ be a closed convex set  and $x \in D$.  Then
	\begin{align*}
		T_D(x) &= \cl \left( \bigcup_{t > 0} t(D - x) \right),\\
		N_D(x) &= \{ x^* \in X^* : \langle x^*, y - x \rangle \le 0,\ \forall y \in D \}.  
	\end{align*}
	Moreover, we have  $N_D(x) = T_D(x)^\ominus$ and $T_D(x) = N_D(x)^\ominus$.
\end{remark}

\subsection{The  oriented distance function}

The oriented distance function (see \cite{Hiriart1,Hiriart2}) of a set $D \subseteq Y$ is defined as
$$
\Delta_D(y) := d_D(y) - d_{Y \setminus D}(y) = \begin{cases}
	d_D(y), & y \in Y \setminus D, \\
	-d_{Y \setminus D}(y), & y \in D.
\end{cases}
$$
With the convention: if $D = \emptyset$, then $\Delta_D \equiv +\infty$; if $D = Y$, then $\Delta_D \equiv -\infty$.

Next,we collect the basic properties of the oriented distance function.
\begin{proposition}  (see \cite{XuLi2016, KTZ2015, Zaffaroni2003})  \label{XSQ}
	Let $D \subseteq Y$ be nonempty and $D \neq Y$. Then:
	\begin{itemize}
		\item[(i)] $\Delta_D$ is real-valued and $\Delta_{Y \setminus D} = -\Delta_D$;
		\item[(ii)] $\Delta_D$ is $1$-Lipschitz;
		\item[(iii)] $\Delta_D(y) < 0 \Leftrightarrow y \in \operatorname{int} D$;
		\item[(iv)] $\Delta_D(y) = 0 \Leftrightarrow y \in \partial D$;
		\item[(v)] $\Delta_D(y) > 0 \Leftrightarrow y \in \operatorname{int} D^c$;
		\item[(vi)] If $D$ is closed, then $D = \{y \in Y : \Delta_D(y) \leq 0\}$;
		\item[(vii)] If $D$ is a cone, then $\Delta_D$ is positively homogeneous;
		\item[(viii)] If $D$ is convex, then $\Delta_D$ is a convex function;
		\item[(ix)] If $D$ is a convex cone, then $\Delta_D$ is nonincreasing with respect to the order induced by $D$ on $Y$, i.e., for any $y_1, y_2 \in Y$,
		\[
		y_1 - y_2 \in D \;\Longrightarrow\; \Delta_D(y_1) \leq \Delta_D(y_2);
		\]
		if $D$ has nonempty interior, then
		\[
		y_1 - y_2 \in \operatorname{int} D \;\Longrightarrow\; \Delta_D(y_1) < \Delta_D(y_2).
		\]
	\end{itemize}
\end{proposition}

By Lemma 2 (ii) of \cite{LiuNgYang2009}, we have the following lemma.

\begin{lemma}  \label{QWED} Assume that  $D \subseteq Y$  is nonempty and convex and  $y \in Y$. If ${d_D}\left( y \right) > 0$,  then
	$${d_D}\left( y \right) = \mathop {\sup }\limits_{{y^*} \in {S_{{Y^*}}}} \mathop {\inf }\limits_{u \in D} \left\langle {{y^*},y - u} \right\rangle .$$	
\end{lemma}

By Lemma  \ref{QWED} and the proof of Proposition 1  of \cite{LiuNgYang2009}, we   get the following lemma.  
\begin{lemma}  \label{QWAD} Assume that  $D \subseteq Y$  is nonempty, closed and convex, $D \ne Y$
	and  $y \in D$. Then
	$${d_{Y\backslash D}}(y) = \mathop {\inf }\limits_{{y^*} \in {S_{{Y^*}}}} \mathop {\sup }\limits_{u \in D} \left\langle {{y^*},y - u} \right\rangle .$$
\end{lemma}

\begin{proposition} \label{wqaz} Let $D$ be a nonempty and convex subset of $Y$. Then 
	$${\Delta _D}(y) = \mathop {\sup }\limits_{{y^*} \in {S_{{Y^*}}}} \mathop {\inf }\limits_{u \in D} \left\langle {{y^*},y - u} \right\rangle , \quad \forall y \in Y.$$	
\end{proposition} 
\begin{proof}  It follows from      Lemma 2.3 of \cite{Han2022} that ${\Delta _{{\rm{cl}}D}}(y) = {\Delta _D}(y)$.  There are two cases to be considered.

Case 1. $y \notin {\rm{cl}}D$. Then we have ${d_D}\left( y \right) > 0$ and ${\Delta _D}\left( y \right) = {d_D}\left( y \right)$. This together with  Lemma  \ref{QWED}  implies that 
$${\Delta _D}\left( y \right) = {d_D}\left( y \right) = \mathop {\sup }\limits_{{y^*} \in {S_{{Y^*}}}} \mathop {\inf }\limits_{u \in D} \left\langle {{y^*},y - u} \right\rangle .$$	

Case 2. $y \in {\rm{cl}}D$.  This yields    that ${\Delta _{{\rm{cl}}D}}(y) =  - {d_{Y\backslash {\rm{cl}}D}}(y)$. By Lemma  \ref{QWAD}, we have 
$${d_{Y\backslash {\rm{cl}}D}}(y) = \mathop {\inf }\limits_{{y^*} \in {S_{{Y^*}}}} \mathop {\sup }\limits_{u \in {\rm{cl}}D} \left\langle {{y^*},y - u} \right\rangle = \mathop {\inf }\limits_{{y^*} \in {S_{{Y^*}}}} \mathop {\sup }\limits_{u \in D} \left\langle {{y^*},y - u} \right\rangle .$$
Combining this with  $ - {S_{{Y^*}}} = {S_{{Y^*}}}$, we get
\begin{eqnarray*} {\Delta _D}(y) &=& {\Delta _{{\rm{cl}}D}}(y) =  - {d_{Y\backslash {\rm{cl}}D}}(y) = - \mathop {\inf }\limits_{{y^*} \in {S_{{Y^*}}}} \mathop {\sup }\limits_{u \in D} \left\langle {{y^*},y - u} \right\rangle \\
	&=& \mathop {\sup }\limits_{{y^*} \in {S_{{Y^*}}}} \mathop {\inf }\limits_{u \in D} \left\langle { - {y^*},y - u} \right\rangle = \mathop {\sup }\limits_{{y^*} \in  - {S_{{Y^*}}}} \mathop {\inf }\limits_{u \in D} \left\langle {{y^*},y - u} \right\rangle  \\
	&=& \mathop {\sup }\limits_{{y^*} \in {S_{{Y^*}}}} \mathop {\inf }\limits_{u \in D} \left\langle {{y^*},y - u} \right\rangle .		
\end{eqnarray*}
 \end{proof}

\begin{remark}  In view of Proposition \ref{wqaz},       the condition \(\operatorname{ri}(A) \neq \emptyset\) in Proposition 1 of \cite{LiuNgYang2009} is redundant,  the conclusion remains valid even if this condition is removed.
   \end{remark}

By Proposition \ref{wqaz}, the following corollary is readily obtained.
 \begin{corollary}\label{cor:Delta}
 	For a nonempty and  convex cone $P \subseteq Y$,  we have 
 $$
 		\Delta_P(y) = \sup_{y^* \in S(P^+)} \langle -y^*, y \rangle, \quad \forall y \in Y.
$$
 \end{corollary}

\begin{remark}   As a consequence of Corollary \ref{cor:Delta}, the hypothesis $\intr K \neq \emptyset$ in Corollary 2 of \cite{LiuNgYang2009} turns out to be redundant, the conclusion remains valid even without this condition.
  \end{remark}

\subsection{Problem-specific notation}

Returning to the vector optimization problem \eqref{eq:VP}, define the lifted cone
\[
C_X := F^{-1}(C) = \{x \in X : F(x) \in C\}.
\]
One readily verifies that if \( F(X) \cap \operatorname{int} C \neq \emptyset \), then \( \operatorname{int} C_X = F^{-1}(\operatorname{int} C) \neq \emptyset \).  
To avoid the triviality we always assume that $F(X) \cap \intr C \neq \emptyset$.

The adjoint operator $F^*: Y^* \to X^*$ is defined by $\langle F^*(y^*), x \rangle = \langle y^*, F(x) \rangle$, and satisfies the key relation (see \cite{Zalinescu1978}):
$$
	F^*(C^+) = (F^{-1}(C))^+ = C_X^+.
$$

The set
$$
	\widehat{A} := X \setminus (A + \intr C_X)
$$
will play a central role. Note that $E_w  = A \cap \widehat{A}$ (see \cite[Theorem 1]{LiuNgYang2009}).
It is clear that  $E_w $ is closed.

Let $K := \overline{\operatorname{co}}^{w^*}\big(S(C^+)\big)$,   where $S(C^+)=\{y^*\in C^+:\|y^*\|=1\}$  and ${\overline{\operatorname{co}}}^{w^*}$ denotes the closed convex hull in the weak$^*$ topology.
   In this paper, the set $K$ is of central importance, as many of the results are formulated in terms of $K$.

\begin{proposition}   \label{RDCxl}   
	Let $K = \overline{\operatorname{co}}^{w^*}\big(S(C^+)\big)$. Then $K$ is a nonempty weak$^*$-compact convex set  and  $0 \notin K$ 
   \end{proposition}
\begin{proof}  
Clearly, $K \subseteq C^+$ and  $K \subseteq \{y^*\in Y^*:\|y^*\|  \leq 1\}$.
Since $S(C^+)=\{y^*\in C^+:\|y^*\|=1\}$ is nonempty weak$^*$-compact, we obtain that $K$ is a nonempty weak$^*$-compact convex set.  

Fix \(y_0 \in \operatorname{int} C\). Then  there exists \(\varepsilon > 0\) such that
\begin{equation}\label{Edfty1}
	y_0 + \varepsilon B_Y \subseteq C,
\end{equation}
where \(B_Y \) is the closed unit ball in \(Y\).
Take an arbitrary \(y^* \in S(C^+)\), i.e., \(y^* \in C^+\) and \(\|y^*\| = 1\). By the definition of \(C^+\), we have \(\langle y^*, c \rangle \ge 0\) for all \(c \in C\).
In particular, for every \(u \in B_Y\), by~\eqref{Edfty1} we have \(y_0 + \varepsilon u \in C\). Hence
\[
\langle y^*, y_0 + \varepsilon u \rangle \ge 0, \quad \forall u \in B_Y.
\]
This is equivalent to
\[
\langle y^*, y_0 \rangle \ge -\varepsilon \langle y^*, u \rangle \quad \forall u \in B_Y.
\]
Taking the supremum over the right-hand side, we obtain
\begin{equation}\label{Edfty2}
	\langle y^*, y_0 \rangle \ge \sup_{u \in B_Y} \bigl(-\varepsilon \langle y^*, u \rangle\bigr)
	= \varepsilon \sup_{u \in B_Y} \bigl(-\langle y^*, u \rangle\bigr).
\end{equation}
Note that
\[
\sup_{u \in B_Y} \bigl(-\langle y^*, u \rangle\bigr) = \sup_{u \in B_Y} \langle y^*, -u \rangle = \sup_{v \in B_Y} \langle y^*, v \rangle = \|y^*\| = 1,
\]
which together with~\eqref{Edfty2} yields
\[
\langle y^*, y_0 \rangle \ge \varepsilon \cdot 1 = \varepsilon > 0.
\]
Since \(y^* \in S(C^+)\) is arbitrary, we conclude that
\begin{equation}\label{Edfty3}
	\inf_{y^* \in S(C^+)} \langle y^*, y_0 \rangle \ge \varepsilon > 0.
\end{equation}

Now take any \(v \in \operatorname{co}(S(C^+))\). Then there exist finitely many scalars \(\lambda_i \ge 0\) with \(\sum_{i=1}^n \lambda_i = 1\) and \(y_i^* \in S(C^+)\) such that
\[
v = \sum_{i=1}^n \lambda_i y_i^*.
\]
For each \(i\), by~\eqref{Edfty3} we have \(\langle y_i^*, y_0 \rangle \ge \varepsilon\). Therefore,
\begin{equation}\label{Edfty4}
	\langle v, y_0 \rangle = \sum_{i=1}^n \lambda_i \langle y_i^*, y_0 \rangle \ge \sum_{i=1}^n \lambda_i \varepsilon = \varepsilon.
\end{equation}

If \(w = 0 \in K\), there exists a net \(\{v_\alpha\} \subseteq \operatorname{co}(S(C^+))\) such that \(v_\alpha \xrightarrow{w^*} w\). Then \(\langle v_\alpha, y_0 \rangle \to \langle w, y_0 \rangle\).

By~\eqref{Edfty4}, for every \(\alpha\) we have \(\langle v_\alpha, y_0 \rangle \ge \varepsilon\). Hence,
\[
0 = \langle 0, y_0 \rangle = \langle w, y_0 \rangle = \lim_\alpha \langle v_\alpha, y_0 \rangle \ge \varepsilon > 0,
\]
which is a contradiction. Therefore, \(0 \notin K\).
   \end{proof}  

\section{Properties and dual representation of the merit function  }\label{sec:merit}

 The aim of this section is to establish the fundamental analytic properties of 
 $\theta$—its concavity, Lipschitz continuity, and dual representation—which will be indispensable for the subsequent study of directional derivatives and error bounds.

\begin{definition}\label{def:theta}  \cite{LiuNgYang2009}
	The merit function $\theta : X \to \mathbb{R}\cup \left\{ +\infty \right\}$ is defined by
	 	\begin{eqnarray*}  	\theta(x) &=& -\inf_{a \in A} \Delta_C(F(x) - F(a)) = \sup_{a \in A} (-\Delta_C(F(x) - F(a))) \\
		&=&   \sup_{a \in A} \inf_{y^* \in S(C^+)} \langle y^*, F(x) - F(a) \rangle, \quad \forall x \in X.
	\end{eqnarray*}

\end{definition}

\begin{lemma} \label{YL:thetaprop} \cite[Theorem 3]{LiuNgYang2009}
	The function $\theta$ satisfies:
	\begin{enumerate}[leftmargin=*,label=(\roman*)]
		\item For any $x \in A$, $\theta(x) \geq 0$.
		\item $\theta$ is $C_X$-monotone, i.e.,  for any $x_1, x_2 \in X$ with $x_1 \le_{C_X} x_2$, we have 	$	  \theta(x_1)   \leq	\theta(x_2) .		$
		\item $\theta(x) = 0$ if and only if $x \in E_w$.
		\item There exists a constant $r > 0$   such that
		\[
		r \cdot d_{\widehat{A}}(x) \leq \theta(x) \leq \|F\| \cdot d_{\widehat{A}}(x), \quad \forall x \in A.
		\]
	\end{enumerate}
\end{lemma}

\begin{proposition}\label{AQWE}
	If there exists    $y_0^*\in C^+\setminus\{0\}$ such that $\inf_{a\in A}\langle y_0^*,F(a)\rangle > -\infty$, then $\theta(x)<+\infty$ for all $x\in X$.
\end{proposition}
\begin{proof}
	Suppose  that  there exists $y_0^*\in C^+\setminus\{0\}$ such that
	\[
	m:=\inf_{a\in A}\langle y_0^*,F(a)\rangle > -\infty.
	\]
	Set $\tilde y^*=y_0^*/\|y_0^*\|$.  Then $\tilde y^*\in S(C^+)$ and $\inf_{a\in A}\langle\tilde y^*,F(a)\rangle = m/\|y_0^*\| > -\infty$.   For any $x\in X$, we have
	\[
	\begin{aligned}
		\theta(x) &= \sup_{a \in A} \inf_{y \in S(C^+)} \langle y^*, F(x) - F(a) \rangle  
		\leq \sup_{a \in A} \langle \tilde y^*, F(x) - F(a) \rangle \\
		&= \langle \tilde y^*, F(x) \rangle - \inf_{a \in A} \langle \tilde y^*, F(a) \rangle  
		= \langle \tilde y^*, F(x) \rangle - m/\|y_0^*\| < + \infty.
	\end{aligned}
	\]
\end{proof}

\begin{definition}  	\cite{Luc, DF2024} 
	A nonempty set $D \subseteq Y$ is called:
	\begin{enumerate}
		\item[(i)] $C$-bounded, if for every neighborhood $U$ of $0 \in Y$, there exists $t > 0$ such that $D \subseteq tU + C$;
		\item[(ii)] $C$-sequentially compact, if for any sequence $\{y_n\}  \subseteq D$, there exist a sequence $\{c_n\}  \subseteq C$, $y_0 \in D$, and a subsequence $\{y_{n_k} - c_{n_k}\} $ of $\{y_n - c_n\} $ such that $y_{n_k} - c_{n_k} \to y_0$.
	\end{enumerate}
\end{definition}

\begin{corollary} \label{ZRUQ}
	If $F(A)$ is $C$-bounded, then $\theta(x)<+\infty$ for all $x\in X$.
\end{corollary}

\begin{proof}
	Take any $y^*\in S(C^+)$. According to the definition of $C$-boundedness, for    the open unit ball $U = \{y\in Y : \|y\| < 1\}$,  there exists $t_0 > 0$ such that
$	F(A) \subseteq t_0 U + C .$
	This means that for each $a\in A$, there exist $u_a \in U$ and $c_a \in C$ satisfying
$	F(a) = t_0 u_a + c_a .$
	Applying the functional $y^*$ to both sides yields
\begin{equation}\label{DzxcB1}
	\langle y^*, F(a)\rangle = t_0\langle y^*, u_a\rangle + \langle y^*, c_a\rangle .
 \end{equation}
	Since $y^*\in C^+$, we have $\langle y^*, c_a\rangle \ge 0$. Moreover, from $\|u_a\| < 1$ and $\|y^*\| = 1$, we get $|\langle y^*, u_a\rangle| \le \|y^*\|\cdot\|u_a\| < 1$, hence $\langle y^*, u_a\rangle > -1$. This together with  (\ref{DzxcB1})   implies that 
	\[
	\langle y^*, F(a)\rangle > -t_0 \quad \forall a\in A .
	\]
	Taking the infimum over $a$ gives
	\[
	\inf_{a\in A} \langle y^*, F(a)\rangle \ge -t_0 > -\infty .
	\]
	By Proposition \ref{AQWE}, $\theta(x)<+\infty$ for all $x\in X$.
\end{proof}

\begin{proposition}\label{XLJKD}
	If $F(A)$ is $C$-sequentially compact, then for any $x\in X$, there exists $a_0 \in A$ such that
	$$\theta(x) =  - {\Delta _C}\bigl( F(x) - F(a_0)\bigr).$$
\end{proposition}

\begin{proof}
	Fix an arbitrary $x\in X$, and set $f(y) = \Delta_C(F(x) - y)$.  It is easy to see  that $-\theta(x) = \inf_{y \in F(A)} f(y)$.
	By the $1$-Lipschitz continuity of $\Delta_C$, $f: Y \to \mathbb{R}$ is continuous.
	
	Let $y_1, y_2 \in Y$ satisfy $y_2 - y_1 \in C$, then $F(x) - y_1 - (F(x) - y_2) \in C$. By Proposition \ref{XSQ}, we have
	$$f(y_1) = {\Delta _C}(F(x) - y_1) \le {\Delta _C}(F(x) - y_2) = f(y_2).$$
	According to Proposition 3.5 of \cite{Han2025}, there exists $y_0 \in F(A)$ such that
	$${\Delta _C}(F(x) - y_0) = f(y_0) = \inf_{y \in F(A)} f(y) = -\theta(x).$$
	Since $y_0 \in F(A)$, there exists $a_0 \in A$ with $y_0 = F(a_0)$. Hence,
	$$- {\Delta _C}\bigl( F(x) - F(a_0)\bigr) = - {\Delta _C}\bigl( F(x) - y_0\bigr) = \theta(x).$$
\end{proof}

\begin{proposition} \label{QQB} 
	For a nonempty set $Q \subseteq X$, we have:
	\begin{enumerate}
		\item[(i)] if $Q$ is $C_X$-bounded, then $F(Q)$ is $C$-bounded.
		\item[(ii)] if $F$ is surjective  and   $F(Q)$ is $C$-bounded, then  $Q$ is $C_X$-bounded.
	\end{enumerate}
\end{proposition}

\begin{proof}
	(i). Take any neighborhood $V$ of $0$ in $Y$. By continuity of $F$, $U := F^{-1}(V)$ is a neighborhood of $0$ in $X$. By the $C_X$-boundedness of $Q$, there exists $t > 0$ such that $Q \subseteq tU + C_X$. Since $F$ is linear,  we have 
	\[
	F(Q) \subseteq F(tU + C_X) = tF(U) + F(C_X).
	\]
	Observe that $F(U) = F(F^{-1}(V)) \subseteq V$ and $F(C_X) \subseteq C$, hence $F(Q) \subseteq tV + C$.
	By the arbitrariness of $V$, $F(Q)$ is $C$-bounded.
	
	(ii). Since $X, Y$ are Banach spaces and $F$ is a continuous linear surjection, by the Banach--Schauder open mapping theorem, $F$ is an open map. Take any open neighborhood $U$ of $0$ in $X$. Set $V := F(U)$. By the open mapping property, $V$ is an open set in $Y$ containing $0 \in Y$. By the $C$-boundedness of $F(Q)$, there exists $t > 0$ such that $F(Q) \subseteq tV + C$. For any $x \in Q$, since $F(x) \in F(Q) \subseteq tV + C$, there exist $v \in V$ and $c \in C$ such that $F(x) = t v + c$. From $V = F(U)$, there exists $u \in U$ satisfying $v = F(u)$. Thus,
	\[
	F(x) = t F(u) + c \implies F(x - t u) = c \in C.
	\]
	By $C_X = F^{-1}(C)$, we have $x - t u \in C_X$. Hence $x = t u + (x - t u) \in tU + C_X$.
	By the arbitrariness of $x$, $Q \subseteq tU + C_X$. Therefore, $Q$ is $C_X$-bounded.
\end{proof}

\begin{remark}  By   Corollary  \ref{ZRUQ}  and Proposition  \ref{QQB}  (i), we get that  if $A$ is $C_X$-bounded, then $\theta(x)<+\infty$ for all $x\in X$.
   \end{remark}

\begin{remark}   \label{VKYJQ}  Let $y^* \in  C^+$ and  $A$ be $C_X$-bounded.  It follows from  Proposition  \ref{QQB}  (i)  that $F(A)$ is $C$-bounded. Then it is easy to obtain that 
	$$\mathop {\inf }\limits_{a \in A} \left\langle {{F^*}{y^*},a} \right\rangle  = \mathop {\inf }\limits_{a \in A} \left\langle {{y^*},F\left( a \right)} \right\rangle  >  - \infty $$
	
  \end{remark}

\begin{lemma}\label{XDRT}
	If $F$ is surjective, then there exists a constant $\lambda>0$ such that for every $x\in X$,
	\[
	\operatorname{dist}(x, C_X) \le \lambda \operatorname{dist}(F(x), C).
	\]
\end{lemma}

\begin{proof}
	Let $\delta = \operatorname{dist}(F(x), C)$. 
	If \(\delta=0\),  it follows from   the closedness of  \(C\)   that 
	 \(F(x)\in C\), and hence \(x\in C_X\). Therefore, the result is trivial. In the following, we consider the case \(\delta>0\).

	For any $\varepsilon > 0$,   there exists $c_\varepsilon \in C$ such that
	\begin{equation}\label{eq:1}
		\|F(x) - c_\varepsilon\| < \delta + \varepsilon .
	\end{equation}
	Since $F$ is a continuous linear surjection, the Banach--Schauder open mapping theorem asserts that $F$ is open. Hence there exists $r>0$ such that
	\begin{equation}\label{eq:2}
		r B_Y \subseteq F(B_X).
	\end{equation}

	Let $y_\varepsilon := F(x) - c_\varepsilon \in Y$.	We conclude from  \(\delta>0\)  that    $y_\varepsilon \neq 0$,  and define
$	y_r := r \frac{y_\varepsilon}{\|y_\varepsilon\|}.$
	Then $y_r \in r B_Y$. By \eqref{eq:2}, there exists $v_r \in B_X$ such that $F(v_r) = y_r$. Let
$	u_r := \frac{\|y_\varepsilon\|}{r} v_r.$
	By linearity of $F$,
	\begin{equation}\label{eq:++}
	F(u_r) = \frac{\|y_\varepsilon\|}{r} F(v_r) = \frac{\|y_\varepsilon\|}{r} \cdot r \frac{y_\varepsilon}{\|y_\varepsilon\|} = y_\varepsilon.
	\end{equation}
	Moreover,
	\begin{equation}\label{eq:3}
		\|u_r\| = \frac{\|y_\varepsilon\|}{r} \|v_r\| \leq \frac{\|y_\varepsilon\|}{r} .
	\end{equation}
It follows from  (\ref{eq:++})  that  $	F(x - u_r) = F(x) - y_\varepsilon = c_\varepsilon \in C,$
and	so $x - u_r \in F^{-1}(C) = C_X$. Combining \eqref{eq:3} and $y_\varepsilon = F(x) - c_\varepsilon$, we obtain
	\begin{equation}\label{eq:4}
		\operatorname{dist}(x, C_X) \le \|x - (x - u_r)\| = \|u_r\| \leq \frac{\|y_\varepsilon\|}{r} = \frac{1}{r} \|F(x) - c_\varepsilon\|.
	\end{equation}
	Together with \eqref{eq:1}, we have
$	\operatorname{dist}(x, C_X) \le \frac{1}{r} (\delta + \varepsilon).$
	Since $\varepsilon > 0$ is arbitrary,  we get
	\[
	\operatorname{dist}(x, C_X) \le \frac{1}{r} \delta = \frac{1}{r} \operatorname{dist}(F(x), C).
	\]
\end{proof}

\begin{proposition}
	For a nonempty set $Q \subseteq X$, we have:
	\begin{enumerate}
		\item[(i)] if $Q$ is $C_X$-sequentially compact, then $F(Q)$ is $C$-sequentially compact.
		\item[(ii)] if $F$ is surjective     and    $F(Q)$ is $C$-sequentially compact, then   $Q$ is $C_X$-sequentially compact.
	\end{enumerate}
\end{proposition}

\begin{proof}
	(i). This follows from Lemma 2.3 of   \cite{Han2026}.
	
	(ii). Let $\{x_n\} \subseteq Q$ be an arbitrary sequence. Set $y_n:=F(x_n)\in F(Q)$. By $C$-sequential compactness of $F(Q)$, there exist a sequence $\{c_n\}  \subseteq C$, $y_0\in F(Q)$, and a subsequence $\{y_{n_k}\} $ of  $\{y_n\}$ such that
$	y_{n_k}-c_{n_k}\to y_0  $.
	Since $y_0\in F(Q)$, there exists $x_0\in Q$ with $F(x_0)=y_0$. Consider $u_k:=x_{n_k}-x_0\in X$. Then
	\[
	F(u_k)=F(x_{n_k})-F(x_0)=y_{n_k}-y_0.
	\]
	From $y_{n_k}-c_{n_k}\to y_0$, we have $y_{n_k}-y_0-c_{n_k}\to0$, i.e.,
$	F(u_k)-c_{n_k}\to 0 .$
	Because $c_{n_k}\in C$, we obtain $\dist(F(u_k),C)\le \|F(u_k)-c_{n_k}\|\to0$.
	Applying Lemma \ref{XDRT} to each $u_k$, we  have 
	\[
	\dist(u_k,\,C_X)\le \lambda \,\dist(F(u_k),C) \to 0 \qquad(k\to\infty).
	\]
Then for each $k$, there exists $w_k\in C_X$ such that $\|u_k-w_k\|<\dist(u_k,C_X)+\frac1k$. Hence $\|u_k-w_k\|\to0$.
	This together with  $u_k=x_{n_k}-x_0$  implies that 
	$	x_{n_k}-w_k = x_0 + u_k - w_k \to x_0 .$
	Set $d_{n_k}:=w_k\in C_X$ (for other indices one may define arbitrarily, e.g., $d_n=0$ if $n\notin\{n_k\}$). Then the sequence $\{d_n\}\subseteq C_X$ satisfies $x_{n_k}-d_{n_k}\to x_0$. Therefore,  $Q$ is $C_X$-sequentially compact.
\end{proof}

For the convenience of discussion, from now on we always assume that $\theta(x)<+\infty$ for all $x\in X$. The following two properties of $\theta$ are important because they play a key role in the derivation of several results in this paper.

\begin{theorem}\label{YL:thetaprop2}
	The  merit function $\theta$ satisfies:
	\begin{enumerate}
		\item[(i)] $\theta$ is concave on $X$;
		\item[(ii)] $\theta$ is Lipschitz continuous on $X$ with constant $\|F\|$.
	\end{enumerate}
\end{theorem}

\begin{proof}
	(i). Since $C$ is convex, $\Delta_C$ is convex, and therefore $-\Delta_C$ is concave on $Y$. Define $g: X\times A \to \mathbb{R}$ by
	\[
	g(x,a) = -\Delta_C\big(F(x)-F(a)\big),  \quad  \forall  (x,a)  \in  X\times A.
	\]
	For any $(x_1,a_1),(x_2,a_2) \in X\times A$ and $\lambda\in[0,1]$, we have
	\begin{eqnarray} \label{XSRT1}
		g\bigl( \lambda x_1 + (1-\lambda)x_2, \lambda a_1 + (1-\lambda)a_2 \bigr)
		&=& -\Delta_C\big( \lambda(F(x_1)-F(a_1)) + (1-\lambda)(F(x_2)-F(a_2)) \big) \nonumber\\ 
		&\ge& \lambda\big[-\Delta_C(F(x_1)-F(a_1))\big] + (1-\lambda)\big[-\Delta_C(F(x_2)-F(a_2))\big] \nonumber\\ 
		&=& \lambda g(x_1,a_1) + (1-\lambda)g(x_2,a_2).  
\end{eqnarray}
It is clear that
	\begin{equation} \label{XSRT2}
		\theta(x) = -\inf_{a\in A} \Delta_C\big(F(x)-F(a)\big)
		= \sup_{a\in A} \Big[-\Delta_C\big(F(x)-F(a)\big)\Big]
		= \sup_{a\in A} g(x,a), \quad \forall x \in X.  
	\end{equation}
Let $x_1, x_2 \in X$ and $\lambda\in[0,1]$.
	For any $\varepsilon>0$, by (\ref{XSRT2}), there exist $a_1,a_2\in A$ such that
\begin{equation} \label{XSRT+2}
	g(x_1,a_1) > \theta(x_1) - \varepsilon, \qquad g(x_2,a_2) > \theta(x_2) - \varepsilon.
\end{equation}
	Set $a_\lambda = \lambda a_1 + (1-\lambda)a_2$. Since $A$ is convex, $a_\lambda \in A$. Then  it follows from  (\ref{XSRT2})  that
\begin{equation} \label{XSRT3}
	\theta(\lambda x_1 + (1-\lambda)x_2) \ge g\big(\lambda x_1 + (1-\lambda)x_2,\; a_\lambda\big).
	\end{equation}
    We conclude from  (\ref{XSRT1})  and  (\ref{XSRT+2})  that
	\[
	\begin{aligned}
		g\big(\lambda x_1 + (1-\lambda)x_2,\; a_\lambda\big)
		&\ge \lambda g(x_1,a_1) + (1-\lambda)g(x_2,a_2) \\
		&> \lambda(\theta(x_1)-\varepsilon) + (1-\lambda)(\theta(x_2)-\varepsilon) \\
		&= \lambda\theta(x_1) + (1-\lambda)\theta(x_2) - \varepsilon.
	\end{aligned}
	\]
	  This together with   (\ref{XSRT3})  implies that 
	\[
	\theta(\lambda x_1 + (1-\lambda)x_2) \ge \lambda\theta(x_1) + (1-\lambda)\theta(x_2) - \varepsilon.
	\]
	By the arbitrariness of $\varepsilon>0$, we obtain
$
	\theta(\lambda x_1 + (1-\lambda)x_2) \ge \lambda\theta(x_1) + (1-\lambda)\theta(x_2).
$
	
	(ii). Fix arbitrary $x_1,x_2\in X$. For any given $\epsilon>0$, there exists $a_\epsilon\in A$ such that
	\begin{equation}   \label{XSRQT1}
		\theta(x_1)-\epsilon < \inf_{y^*\in S(C^+)}\langle y^*,\,F(x_1)-F(a_\epsilon)\rangle .  
	\end{equation}
	For this $a_\epsilon\in A$, we get 
$$
		\theta(x_2)=\sup_{a\in A}\inf_{y^*\in S(C^+)}\langle y^*,\,F(x_2)-F(a)\rangle
		\ge \inf_{y^*\in S(C^+)}\langle y^*,\,F(x_2)-F(a_\epsilon)\rangle , 
$$
 and so
	\begin{equation}\label{XSRQT2}
		-\theta(x_2)\le -\inf_{y^*\in S(C^+)}\langle y^*,\,F(x_2)-F(a_\epsilon)\rangle .  
	\end{equation}
	Adding  (\ref{XSRQT1})  and  (\ref{XSRQT2}), we obtain
	\begin{align*}
		\theta(x_1)-\theta(x_2)-\epsilon 
		&< \inf_{y^*\in S(C^+)}\langle y^*,\,F(x_1)-F(a_\epsilon)\rangle
		-\inf_{y^*\in S(C^+)}\langle y^*,\,F(x_2)-F(a_\epsilon)\rangle \\
		&\le \sup_{y^*\in S(C^+)}\Big[
		\langle y^*,\,F(x_1)-F(a_\epsilon)\rangle
		-\langle y^*,\,F(x_2)-F(a_\epsilon)\rangle
		\Big] \\
		&= \sup_{y^*\in S(C^+)} \langle y^*,\,F(x_1-x_2)\rangle  
		\le \sup_{y^*\in S(C^+)} \|y^*\|\,\|F(x_1-x_2)\| \\
		&= \|F(x_1-x_2)\| \le \|F\|\,\|x_1-x_2\|.
	\end{align*}
	Since $\epsilon>0$ is arbitrary, we have
$	\theta(x_1)-\theta(x_2) \le \|F\|\,\|x_1-x_2\|. $
	By symmetry,  we get
$	\theta(x_2)-\theta(x_1) \le \|F\|\,\|x_1-x_2\|.$
	Hence,
	$	|\theta(x_1)-\theta(x_2)| \le \|F\|\,\|x_1-x_2\|.$
\end{proof}

\begin{theorem}  \label{TxEQ} 
	Let $x\in X$. The following assertions hold:
	\begin{enumerate}
		\item[(i)] $x\in A + \operatorname{int}C_X$ if and only if $\theta(x)>0$;
		\item[(ii)] If $x\in A + C_X$, then $\theta(x)\ge 0$;
		\item[(iii)] If $\theta(x)\ge 0$ and $F(A)$ is $C$-sequentially compact, then $x\in A + C_X$.
	\end{enumerate}
\end{theorem}

\begin{proof}
	For any $x\in X$, by definition of $\theta$,
	\begin{equation}\label{eq:theta-dist}
		\theta(x)=\sup_{a\in A}\bigl[-\Delta_C(F(x)-F(a))\bigr].
	\end{equation}
	
	\textbf{(i)} Suppose $x\in A+\operatorname{int}C_X$. Then there exist $a_0\in A$ and $d\in\operatorname{int}C_X$ such that $x=a_0+d$. By   $\operatorname{int}C_X=F^{-1}(\operatorname{int}C)$, we get $F(d)\in\operatorname{int}C$. It follows from   Proposition \ref{XSQ} that
	$${\Delta _C}\bigl( F(x) - F(a_0) \bigr) = {\Delta _C}\bigl( F(x - a_0) \bigr) = {\Delta _C}\bigl( F(d) \bigr) < 0.$$
	By \eqref{eq:theta-dist}, we have  $\theta(x) \ge -{\Delta _C}\bigl( F(x) - F(a_0) \bigr) > 0$.
	
	Conversely, suppose $\theta(x)>0$. By \eqref{eq:theta-dist}, there exists $a' \in A$ such that
	$-\Delta_C(F(x) - F(a')) > 0$.
	By Proposition \ref{XSQ},
	$F(x) - F(a') \in \operatorname{int} C$.
	Hence
	$x - a' \in F^{-1}(\operatorname{int} C) = \operatorname{int} C_X$.
	Therefore
	$$x \in a' + \operatorname{int} C_X \subseteq A + \operatorname{int} C_X.$$
	
	\textbf{(ii)} Suppose $x\in A + C_X$. Then there exist $a_0\in A$ and $d\in C_X$ such that $x=a_0+d$. Hence
	\[
	F(x)-F(a_0)=F(d)\in F(C_X)\subseteq C.
	\]
	Thus ${\Delta _C}\bigl( F(x) - F(a_0) \bigr) \le0$.     It follows from  (\ref{eq:theta-dist})  that
	$\theta(x) \ge -{\Delta _C}\bigl( F(x) - F(a_0) \bigr) \ge 0$.
	
	\textbf{(iii)} Since $F(A)$ is $C$-sequentially compact and $\theta(x)\ge 0$, by Proposition \ref{XLJKD}, there exists $a_0 \in A$ such that
	$$\theta(x) = -{\Delta _C}\bigl( F(x) - F(a_0) \bigr) \ge 0.$$
	By Proposition \ref{XSQ} and the closedness of $C$, we have $F(x)-F(a_0) \in C$, hence $x - a_0 \in C_X$, and therefore
	$x \in a_0 + C_X \subseteq A + C_X.$
\end{proof}

\begin{lemma}[Sion's Minimax Theorem]\label{SJDJX}
	Let $H$ be a nonempty compact convex set in a Hausdorff topological vector space, and $Q$ a nonempty convex set in a topological vector space. Suppose the function $f : H \times Q \to \mathbb{R}$ satisfies:
	\begin{enumerate}
		\item For each $y \in Q$, $x \mapsto f(x,y)$ is lower semicontinuous and convex on $H$;
		\item For each $x \in H$, $y \mapsto f(x,y)$ is concave on $Q$.
	\end{enumerate}
	Then
	\[
	\min_{x \in H} \sup_{y \in Q} f(x,y) = \sup_{y \in Q} \min_{x \in H} f(x,y).
	\]
\end{lemma}

\begin{theorem}[Dual representation of the merit function]\label{YDOBS}
	The merit function $\theta$ can be represented as:
	\[
	\theta(x)=\min_{v\in K}\Bigl(\langle v,F(x)\rangle -\inf_{a\in A}\langle v,F(a)\rangle\Bigr)
	=\min_{v\in K}\Bigl(\langle F^*v, x\rangle+\sigma_A(-F^*v)\Bigr), \quad \forall x\in X.
	\]
	where $K=\overline{\operatorname{co}}^{w^*}\bigl(S(C^+)\bigr)$ and $\sigma_A(u)=\sup_{a\in A}\langle u,a\rangle$ is the support function of $A$.
\end{theorem}

\begin{proof}	For any fixed $z\in Y$, the linear functional $y^*\mapsto\langle y^*,z\rangle$ is weak$^*$-continuous on $Y^*$, hence attains its minimum on the weak$^*$-compact set $K$, and this minimum equals the infimum over $S(C^+)$, 
\begin{equation}\label{DmlB1}
	\inf_{y^*\in S(C^+)}\langle y^*,z\rangle = \min_{v\in K}\langle v,z\rangle .  
 \end{equation}
	By definition of $\theta$ and  (\ref{DmlB1}),
\begin{equation}\label{DmlB2}
	\theta(x) = \sup_{a\in A} \inf_{y^*\in S(C^+)} \langle y^*, F(x)-F(a)\rangle
	= \sup_{a\in A}\, \min_{v\in K} \langle v, F(x)-F(a)\rangle .  
 \end{equation}
	Define the bifunction $\varphi : A\times K  \to  \mathbb{R} $ by
	\[
	\varphi(a,v) = \langle v, F(x)-F(a)\rangle , \quad   \forall   (a,v)\in A\times K .
	\]
 	By Lemma \ref{SJDJX} (Sion's minimax theorem), we have
\begin{equation}\label{DmlB3}
	\sup_{a\in A}\, \min_{v\in K} \varphi(a,v) = \min_{v\in K}\, \sup_{a\in A} \varphi(a,v).  
\end{equation}
	Then,
	\begin{eqnarray} \label{DmlB4}
		\sup_{a\in A}\varphi(a,v) &=& \sup_{a\in A}\langle v, F(x)-F(a)\rangle \nonumber\\ 
		&=& \langle v, F(x)\rangle + \sup_{a\in A}\langle -v, F(a)\rangle  		 = \langle v, F(x)\rangle - \inf_{a\in A}\langle v, F(a)\rangle .
\end{eqnarray}
	Using the adjoint operator $F^*:Y^*\to X^*$, we have
	$$\langle v, F(x)\rangle = \langle F^*v, x\rangle, \quad \forall x\in X,$$
	and
	\[
	-\inf_{a\in A}\langle v, F(a)\rangle = \sup_{a\in A}\langle -F^*v, a\rangle = \sigma_A(-F^*v).
	\]
	Combining (\ref{DmlB2}), (\ref{DmlB3})  and  (\ref{DmlB4}), we obtain
	\[
	\theta(x) = \min_{v\in K} \Bigl( \langle v, F(x)\rangle - \inf_{a\in A}\langle v, F(a)\rangle \Bigr)
	= \min_{v\in K} \Bigl( \langle F^*v, x\rangle + \sigma_A(-F^*v) \Bigr).
	\]
\end{proof}

From  Theorem 2.4.18 (p.~97) of \cite{Zalinescu2002}, we obtain the following lemma.

\begin{lemma}\label{ITDL}
	Let   $X$ be a locally convex Hausdorff topological vector space,  $Q$ be a nonempty compact Hausdorff space,  $g: Q \to (X^*, w^*)$ be continuous with respect to the weak$^*$-topology on $X^*$,  and $s: Q \to \mathbb{R}$ be upper semicontinuous.
 	For each $v \in Q$, define the affine function
	\[
	f_v(x) = \langle g(v), x \rangle + s(v), \qquad x \in X.
	\]
	Set
	\[
	\phi(x) = \max_{v \in Q} f_v(x), \qquad x \in X.
	\]
	Then $\phi: X \to \mathbb{R}$ is a convex function.  If $\phi$ is continuous at some point $\bar x \in X$, then its subdifferential $\partial \phi(\bar x)$ is nonempty and satisfies
	\[
	\partial \phi(\bar x) = \overline{\operatorname{co}}^{w^*} \{ g(v) : v \in Q(\bar x) \},
	\]
	where $Q(\bar x) = \{ v \in Q : f_v(\bar x) = \phi(\bar x) \}$ is the active index set.
\end{lemma}

Let $\xi = -\theta$. Clearly $\xi$ is a convex   function on $X$.  Denote by $\partial \xi({x})$ the subdifferential of the convex function $\xi$ at $x \in X $, i.e.,
\[
\partial \xi(\bar{x}) := \big\{ x^* \in X^* : \xi(y) \ge \xi(\bar{x}) + \langle x^*, y - x \rangle,\ \forall y \in X \big\}.
\]

For the concave function \(\theta\), the superdifferential at a point \(x\in X\) is defined by
\[
\partial^+ \theta(x):=\{x^*\in X^*:\theta(y)\le\theta(x)+\langle x^*,y-x\rangle, \; \forall y\in X\}.
\]

\begin{theorem}  \label{BSXEQ}  Let $A  \subseteq X$ be $C_X$-bounded.
 For any $\bar{x} \in E_w$, the following representation holds:
	\[
	-\partial \xi(\bar{x}) = \overline{\operatorname{co}}^{w^*} \bigl\{ F^*v : v \in K,\; F^*v \in -N_A(\bar{x}) \bigr\} = \bigl( -N_A(\bar{x}) \bigr) \cap \overline{\operatorname{co}}^{w^*} \bigl( F^*(S(C^+)) \bigr) = \bigl( -N_A(\bar{x}) \bigr) \cap F^*(K).
	\]
\end{theorem}

\begin{proof}
	By Theorem \ref{YDOBS}, we have
$$
	\theta(x) = \min_{v \in K} \bigl( \langle F^*v, x \rangle + \sigma_A(-F^*v) \bigr),  
$$
 and so 
\begin{equation}\label{XCB1}
	\xi(x) = -\theta(x) = \max_{v \in K} \bigl( -\langle F^*v, x \rangle - \sigma_A(-F^*v) \bigr).  
 \end{equation}
	For each $v \in K$, define
	\[
	h_v(x) := \langle -F^*v, x \rangle - \sigma_A(-F^*v), \quad x \in X.
	\]
	Then,  it follows from  (\ref{XCB1})  that
	\[
	\xi(x) = \max_{v \in K} h_v(x).
	\]
	
	For $\bar{x} \in E_w$, we have $\theta(\bar{x}) = 0$, hence $\xi(\bar{x}) = 0$. Define the active index set
	\[
	K(\bar{x}) := \{ v \in K : h_v(\bar{x}) = \xi(\bar{x}) = 0 \}.
	\]
		For any $v \in K$, set $g(v) = -F^*v$ and $s(v) = - \sigma_A(-F^*v)$.   Since $A$ is $C_X$-bounded, it follows from  $K \subseteq C^+$ and   Remark   \ref{VKYJQ}    that $s(v) \in \mathbb{R}$ for any $v \in K$.
		  It is readily seen that   $g$ is continuous   and   $s$ is upper semicontinuous with respect to the weak$^*$-topology.
		By Lemma \ref{ITDL}, we have
$	\partial \xi(\bar{x}) = \overline{\operatorname{co}}^{w^*} \bigl\{ -F^*v : v \in K(\bar{x}) \bigr\}, $
and so 
\begin{equation}\label{XCB2}
	- \partial \xi(\bar{x}) = \overline{\operatorname{co}}^{w^*} \bigl\{ F^*v : v \in K(\bar{x}) \bigr\}.  
 \end{equation}
	
	By definition of $K(\bar{x})$, $v \in K(\bar{x})$ iff $v \in K$ and
\begin{equation}\label{XCB3}
	0 = \langle F^*v, \bar{x} \rangle + \sigma_A(-F^*v).  
 \end{equation}
	The support function gives
	\[
	\sigma_A(-F^*v) = \sup_{a \in A} \langle -F^*v, a \rangle = - \inf_{a \in A} \langle F^*v, a \rangle.
	\]
	Substituting into  (\ref{XCB3}) yields
\begin{equation}\label{XCB4}
	0 = \langle F^*v, \bar{x} \rangle - \inf_{a \in A} \langle F^*v, a \rangle
	\quad \Longleftrightarrow \quad
	\langle F^*v, \bar{x} \rangle = \inf_{a \in A} \langle F^*v, a \rangle.  
 \end{equation}
		Equation (\ref{XCB4}) means that $\bar{x}$ is a global minimizer of the continuous linear functional $a \mapsto \langle F^*v, a \rangle$ on the closed convex set $A$. This is equivalent to
$	-F^*v \in N_A(\bar{x}).$
	Combining with (\ref{XCB2}), we have
\begin{equation}\label{XCB++}
	-\partial \xi(\bar{x}) = \overline{\operatorname{co}}^{w^*} \bigl\{ F^*v : v \in K,\; F^*v \in -N_A(\bar{x}) \bigr\}.
 \end{equation}
	
	Next, we prove that
\begin{equation}\label{XCB5}
	F^*(K) = \overline{\operatorname{co}}^{w^*}\bigl( F^*(S(C^{+})) \bigr).  
 \end{equation}
	First we show $F^*(K) \subseteq \overline{\operatorname{co}}^{w^*}\bigl( F^*(S(C^{+})) \bigr)$.
In fact,	since $F^*$ is linear, for any finite convex combination, we have
	\[
	F^*\left( \sum_{i=1}^n \lambda_i y_i^* \right) = \sum_{i=1}^n \lambda_i F^*(y_i^*), \quad \lambda_i \ge 0,\ \sum_{i=1}^n \lambda_i = 1, \;  y_i^* \in S(C^+).
	\]
	This means
\begin{equation}\label{XCB6}
	F^*(\operatorname{co}(S(C^+))) = \operatorname{co}( F^*(S(C^+)) ).  
 \end{equation}
	Take any $z \in K = \overline{\operatorname{co}}^{w^*}(S(C^+))$. By definition of the weak$^*$-closure, there exists a net $\{z_\alpha\}_{\alpha \in I} \subseteq \operatorname{co}(S(C^+))$ such that $z_\alpha \xrightarrow{w^*} z$.
	For each $\alpha$, by (\ref{XCB6}) we have $F^*(z_\alpha) \in \operatorname{co}(F^*(S(C^+)))$. Since $F^*$ is weak$^*$-to-weak$^*$ continuous, we get 
	\[
	F^*(z_\alpha) \xrightarrow{w^*} F^*(z) \quad \text{in } X^*.
	\]
 This means that
$	F^*(z) \in \overline{\operatorname{co}}^{w^*}\bigl( F^*(S(C^+)) \bigr).$
	Since $z \in K$ is arbitrary, we obtain 
	\begin{equation}\label{XCB7}
	F^*(K) \subseteq \overline{\operatorname{co}}^{w^*}\bigl( F^*(S(C^+)) \bigr).
	 \end{equation}
	 
	Since $K = \overline{\operatorname{co}}^{w^*}(S(C^+))$ is   a   weak$^*$-compact convex set, and $F^*$ is   linear and weak$^*$-to-weak$^*$ continuous, the set $F^*(K)$ is weak$^*$-compact and convex. This together with $F^*(S(C^+)) \subseteq F^*(K)$  implies that
$	\overline{\operatorname{co}}^{w^*} \bigl( F^*(S(C^+)) \bigr) \subseteq F^*(K).$   Combining this with (\ref{XCB7}), we get that (\ref{XCB5})  holds.

	Since $K = \overline{\operatorname{co}}^{w^*}(S(C^{+}))$, and $K(\bar{x}) \subseteq K$, using (\ref{XCB5}),
	\[
	\{ F^*v : v \in K(\bar{x}) \} \subseteq F^*(K) = \overline{\operatorname{co}}^{w^*}\bigl( F^*(S(C^{+})) \bigr).
	\]
	From (\ref{XCB2}) we obtain
		\begin{equation}\label{XCB8}
	- \partial \xi(\bar{x}) \subseteq \overline{\operatorname{co}}^{w^*} \bigl( F^*(S(C^{+})) \bigr).  
	 \end{equation}
	
Noting that   $-N_A(\bar{x})$ is convex and weak$^*$-closed, 	it is easy to see that $ \overline{\operatorname{co}}^{w^*} \bigl\{ F^*v : v \in K,\; F^*v \in -N_A(\bar{x}) \bigr\}    \subseteq -N_A(\bar{x})$. 	This together with  (\ref{XCB++})  yields that   $-\partial \xi(\bar{x}) \subseteq -N_A(\bar{x})$.
		Combining with (\ref{XCB8}), we have
	\[
	- \partial \xi(\bar{x}) \subseteq \bigl( -N_A(\bar{x}) \bigr) \cap \overline{\operatorname{co}}^{w^*} \bigl( F^*(S(C^{+})) \bigr).
	\]

	Take any $x^* \in \bigl( -N_A(\bar{x}) \bigr) \cap \overline{\operatorname{co}}^{w^*} \bigl( F^*(S(C^+)) \bigr)$.
	By (\ref{XCB5}), we get  $x^* \in F^*(K)$, so there exists $v \in K$ with $x^* = F^*v$.
	Since $F^*v  =    x^* \in -N_A(\bar{x})$,     we have $x^* \in \bigl\{ F^*v : v \in K,\; F^*v \in -N_A(\bar{x}) \bigr\}$.   This together with  (\ref{XCB++})  implies that  $x^* \in   -\partial \xi(\bar{x}).$
	 	This proves
	\[
	\bigl( -N_A(\bar{x}) \bigr) \cap \overline{\operatorname{co}}^{w^*} \bigl( F^*(S(C^+)) \bigr) \subseteq -\partial \xi(\bar{x}).
	\]
\end{proof}

\begin{remark}  \label{RXExcQ}  
	It follows from   $\xi = -\theta$  that  $	\partial^+ \theta(x) =  	- \partial \xi({x}) $.   Let $A  \subseteq X$ be $C_X$-bounded.    Then by Theorem \ref{BSXEQ}, we obtain that if $\bar{x} \in E_w$, then 
		\[
		\partial^+ \theta(x) = \overline{\operatorname{co}}^{w^*} \bigl\{ F^*v : v \in K,\; F^*v \in -N_A(\bar{x}) \bigr\} = \bigl( -N_A(\bar{x}) \bigr) \cap \overline{\operatorname{co}}^{w^*} \bigl( F^*(S(C^+)) \bigr) = \bigl( -N_A(\bar{x}) \bigr) \cap F^*(K).
	\]

  \end{remark}

\section{Error bounds   of the merit   function  }\label{sec:errorbounds}

 In this section we undertake a comprehensive study of error bounds for \(\theta\). We establish the equivalence of thirteen different characterizations, ranging from linear regularity of the pair \(\{A,\widehat A\}\), through  the  global slope condition, to perturbation stability and sublevel‐set Hausdorff continuity. These equivalences unify and extend previous results in the literature and set the stage for the directional‐derivative characterizations developed in later sections.  

\begin{lemma}[Ekeland's variational principle]\label{ELdaz}
	Let $(M, d)$ be a complete metric space and $f: M \to \mathbb{R} \cup \{+\infty\}$ a proper lower semicontinuous function bounded below. If $x_0 \in \operatorname{dom} f$ and $\varepsilon > 0,\ \lambda > 0$ satisfy
	\[
	f(x_0) \le \inf_{x \in M} f(x) + \varepsilon,
	\]
	then there exists $\bar{x} \in M$ such that:
	\begin{enumerate}
		\item $f(\bar{x}) \le f(x_0)$;
		\item $d(x_0, \bar{x}) \le \lambda$;
		\item For all $x \in M \setminus \{\bar{x}\}$,
		\[
		f(x) > f(\bar{x}) - \frac{\varepsilon}{\lambda} \, d(x, \bar{x}) .
		\]
	\end{enumerate}
\end{lemma}

The set-valued map $\Phi:[0,+\infty)\rightrightarrows A$ is defined by
$$\Phi(t)=\{x\in A:\theta(x)\le t\}, \quad  \forall t \in [0,+\infty).$$

For any \(x\in A\) with \(\theta(x)>0\),  the  global slope is defined by
\[
m(x):=\sup_{\substack{y\in A\\ \theta(y)<\theta(x)}}\frac{\theta(x)-\theta(y)}{\|x-y\|}\in[0,+\infty].
\]

\begin{theorem} \label{XLWCJK}  Consider the following statements:
	\begin{enumerate}
		\item[(i)] $\theta$ has an error bound on $A$: there exists $\tau>0$ such that
		\begin{equation}\label{EBAQJ} 
			d(x,E_w)\le\tau \theta(x), \quad  \forall x\in A.
		\end{equation}		 
		\item[(ii)] $\{A,\widehat{A}\}$ is linearly regular: there exists $\kappa>0$ such that
				 $$d_{A\cap\widehat{A}}(x)\le\kappa\max\{d_A(x),d_{\widehat{A}}(x)\},   \quad  \forall     x\in X. $$  
		 
		\item[(iii)] $d_{\widehat{A}}$ has an error bound on $A$: there exists $\tau>0$ such that $d(x,E_w)\le\tau d_{\widehat{A}}(x)$ for all $x\in A$.
		\item[(iv)]  There exist  \(\alpha > 0\) and \(\beta \in (0,1)\) such that for every \(x \in A\) one can find \(y \in A\) satisfying
		\begin{equation}\label{EYZXJ} 
		\|x - y\| \le \alpha\,\theta(x) \quad\text{and}\quad \theta(y) \le \beta\,\theta(x).
		\end{equation}	
		\item[(v)]  There exist \(\kappa > 0\) such that  
		\begin{equation}\label{EBkkQJ} 
		d(x+e, E_w) \le \kappa (\theta(x) + \|e\|),  \quad  \forall x \in A, \;    \forall e \in X.
			\end{equation}	
		
		\item[(vi)]  There exist $\gamma>0$ such that  
		\begin{equation}\label{EBkXJ} 
		d(x,E_w)\le \gamma\bigl(\max\{\theta(x),0\}+d_A(x)\bigr),\quad \forall x\in X .
				\end{equation}	
		
		\item[(vii)]  There exists   \(\kappa > 0\) such that  
			$$			d(x, E_w) \le \kappa \, \zeta(x), \quad \forall x \in X, 		$$
			where  $\zeta(x) := \inf_{a \in A} \bigl( \theta(a) + \|x - a\| \bigr)$.

			\item[(viii)] There exists   \(c>0\) such that for each \(x\in A\setminus E_w\), one can find \(y\in A\) satisfying
			\begin{equation}\label{EevL}
				y\neq x \quad\text{and}\quad
				\theta(y)\le \theta(x)-c\|x-y\| .
	 \end{equation}
			
					\item[(ix)]   There exists   \(\gamma >0\) such that
				\[
				m(x)\ge \gamma,\quad\forall x\in A\setminus E_w.
				\]
			
			 	\item[(x)]  There exists $c>0$ such that    
 	$$ d \bigl(F(x),\operatorname{WMin}(F(A),C)\bigr)     \le c \; \theta(x),   \quad  \forall x\in A.$$ 
		
		\item[(xi)]  Asymptotic   condition:
		\[
		\liminf_{\substack{x\xrightarrow{A} E_w\\ x\notin E_w}}\frac{d_{\widehat{A}}(x)}{d(x,E_w)}>0.
		\]
		where $x\xrightarrow{A} E_w$ means $x\in A$ and $d(x,E_w)\to 0$.
		
		\item[(xii)] There exist \(t_0 > 0\) and \(\beta > 0\) such that 
		\[
		d_H(\Phi(t), E_w) \le \beta t,  \quad  \forall t \in [0, t_0].
		\]

	\item[(xiii)] There exist $\kappa>0$ and an open set $U\supseteq E_w$ such that \begin{equation}\label{EBkmJ} 
			d(x,\Phi(0))\le\kappa\,d(0,\Phi^{-1}(x)),       \quad  \forall       x\in U\cap A. 	
			\end{equation}	 
	\end{enumerate}
	 	Then the following hold:
		\begin{enumerate}
		\item The statements (i)--(ix) are equivalent, (i) $\Rightarrow$ (x), (iii) $\Rightarrow$ (xi),  (i) $\Rightarrow$ (xii) and (i) $\Rightarrow$ (xiii). 
			\item  If there exist $\mu>0$ such that $\|F(z)\|\ge\mu\|z\|$ for any $z\in A-A$,
	then (x) $\Rightarrow$ (i).
			\item  If $A$ is bounded, then (xii) $\Rightarrow$ (i).
	 \item  	If $A$ is compact, then (xi) $\Rightarrow$ (iii) and (xiii) $\Rightarrow$ (i).  
	\end{enumerate}
\end{theorem}

\begin{proof}
It follows from Lemma \ref{YL:thetaprop}     that  (i) $\Leftrightarrow$ (iii).
		By \cite[Theorem 4]{LiuNgYang2009},  we get    (ii) $\Leftrightarrow$ (iii).

	\noindent\textbf{(i) $\Rightarrow$ (iv)}.
  Assume  that   there exists $\tau>0$ such that  (\ref{EBAQJ}) holds.   
  Fix an arbitrary \(x \in A\). 
  
  If \(\theta(x) = 0\), then \(x \in E_w\) and taking \(y = x\) fulfills the requirement.
  
  If \(\theta(x) > 0\),  by   (\ref{EBAQJ}),     there exists \(x' \in E_w\) such that
  \[
  \|x - x'\| < d(x,E_w) + \theta(x) \le \tau\theta(x) + \theta(x) = (\tau+1)\theta(x).
  \]
  Set \(y = x'\),   then \(y \in A\) and \(\theta(y) = 0 \le \beta\theta(x)\) for any \(\beta \in (0,1)\). Thus (iv) holds with constant \(\alpha = \tau+1\) and an arbitrary \(\beta \in (0,1)\).

  	\noindent\textbf{(iv) $\Rightarrow$ (i)}.
  Assume there exist \(\alpha>0\) and \(\beta \in (0,1)\) satisfying  (\ref{EYZXJ}). Set \(\tau = \frac{\alpha}{1-\beta}\).
    Fix an arbitrary \(x \in A\). 
    
      If \(\theta(x)=0\), then \(x \in E_w\) and the conclusion is trivial. 
      
      Suppose \(\theta(x) > 0\). Starting from \(x_0 = x\), we recursively construct a sequence \(\{x_n\} \subseteq A\). 
  Assuming \(x_n\) is already obtained, if \(\theta(x_n) = 0\) we stop,  otherwise by condition (\ref{EYZXJ}) there exists \(x_{n+1} \in A\) such that
  \[
  \|x_n - x_{n+1}\| \le \alpha \theta(x_n), \qquad \theta(x_{n+1}) \le \beta \theta(x_n).
  \]
  Iterating these inequalities yields, for all steps,
  \[
  \theta(x_n) \le \beta^n \theta(x), \qquad \|x_n - x_{n+1}\| \le \alpha \beta^n \theta(x).
  \]
  Since \(0 < \beta < 1\), the series \(\sum_{n=0}^\infty \beta^n\) converges, hence \(\{x_n\}\) is a Cauchy sequence. As a closed subset of a Banach space, \(A\) is complete, so there exists a limit point \(\bar{x} \in A\) with \(x_n \to \bar{x}\). 
  By the Lipschitz continuity of \(\theta\),
  \[
  \theta(\bar{x}) = \lim_{n\to\infty} \theta(x_n) \le \lim_{n\to\infty} \beta^n \theta(x) = 0.
  \]
  Since \(\theta\) is nonnegative on \(A\), we obtain \(\theta(\bar{x}) = 0\), i.e., \(\bar{x} \in E_w\).   Then, we have
  \[ d(x,E_w) \le
  \|x - \bar{x}\| \le \sum_{n=0}^{\infty} \|x_n - x_{n+1}\|
  \le \sum_{n=0}^{\infty} \alpha \beta^n \theta(x)
  = \frac{\alpha}{1-\beta}\,\theta(x) = \tau \theta(x).
  \]
  This establishes (i).

	\noindent\textbf{(i) $\Rightarrow$ (v)}.
Suppose there exists \(\tau > 0\)  such that   (\ref{EBAQJ}) holds. Take arbitrary \(x \in A\) and \(e \in X\).

It follows from  \(x \in A\)  that \(d(x+e, A) \le \|x+e - x\| = \|e\|\).
  For any \(\varepsilon > 0\), choose \(a_\varepsilon \in A\) such that
\begin{equation}\label{eTHVDD1}
\|x+e - a_\varepsilon\| \le d(x+e, A) + \varepsilon \le \|e\| + \varepsilon,
\end{equation}
Since  \(\theta\) is Lipschitz continuous on \(X\) with constant \(L = \|F\|\),  by (\ref{eTHVDD1}), we have 
\[
|\theta(x+e) - \theta(a_\varepsilon)| \le L \|(x+e) - a_\varepsilon\| \le L(\|e\| + \varepsilon),
\]
hence
\begin{equation}\label{eTHVDD2}
	\theta(a_\varepsilon) \le \theta(x+e) + L(\|e\| + \varepsilon).
\end{equation}
 We conclude from  \(a_\varepsilon \in A\),  (\ref{EBAQJ})  and  (\ref{eTHVDD2})  that
\begin{equation}\label{eTHVDD3}
	d(a_\varepsilon, E_w) \le \tau \theta(a_\varepsilon) \le \tau\bigl(\theta(x+e) + L(\|e\| + \varepsilon)\bigr).
\end{equation}
  It follows from  (\ref{eTHVDD1})  and  (\ref{eTHVDD3})    that
\[
\begin{aligned}
	d(x+e, E_w) &\le \|x+e - a_\varepsilon\| + d(a_\varepsilon, E_w) \\
	&\le (\|e\| + \varepsilon) + \tau\bigl(\theta(x+e) + L(\|e\| + \varepsilon)\bigr) \\
	&= \tau \theta(x+e) + (1 + \tau L)(\|e\| + \varepsilon).
\end{aligned}
\]
Letting \(\varepsilon \to 0^+\), we obtain
\begin{equation}\label{eq:d-xe}
	d(x+e, E_w) \le \tau \theta(x+e) + (1 + \tau L)\|e\|.
\end{equation}
Since  \(\theta\) is Lipschitz continuous, \(\theta(x+e) \le \theta(x) + L\|e\|\). Substituting into \eqref{eq:d-xe} yields
\[
\begin{aligned}
	d(x+e, E_w) &\le \tau\bigl(\theta(x) + L\|e\|\bigr) + (1 + \tau L)\|e\| \\
	&= \tau \theta(x) + (2\tau L + 1)\|e\|.
\end{aligned}
\]
 Due to  \(\theta(x) \ge 0\) and \(\|e\| \ge 0\), we have
\[
\tau \theta(x) + (2\tau L + 1)\|e\| \le (\tau + 2\tau L + 1)(\theta(x) + \|e\|).
\]
Setting \(\kappa = \tau + 2\tau L + 1\), we obtain
\[
d(x+e, E_w) \le \kappa (\theta(x) + \|e\|),
\]
  which means that  (\ref{EBkkQJ}) holds.

	\noindent\textbf{(v) $\Rightarrow$ (i)}.
Suppose  that There exist \(\kappa > 0\) such that  (\ref{EBkkQJ}) holds.   
Take any \(x \in A\) and set \(e = 0\). From (\ref{EBkkQJ}) we immediately obtain
\[
d(x, E_w) = d(x+0, E_w) \le \kappa (\theta(x) + 0) = \kappa \theta(x).
\]

	\noindent\textbf{(i) $\Rightarrow$ (vi)}.
Assume  that   there exists $\tau>0$ such that  (\ref{EBAQJ}) holds.
Set $L=\|F\|$, and define
\[
\gamma := \max\{\,\tau,\;1+\tau L\,\}.
\]
Let  $x\in X$      be arbitrary.   For   any    $\varepsilon>0$,  there exists $a_\varepsilon\in A$ such that
 \begin{equation}\label{EvgyB1}	\|x-a_\varepsilon\| < d_A(x)+\varepsilon  .        
 \end{equation}
 It follows from    (\ref{EBAQJ})   and  (\ref{EvgyB1})    that   
 \begin{equation}\label{EvgyB2}
 	d(x,E_w) \le \|x-a_\varepsilon\| + d(a_\varepsilon,E_w) < d_A(x)+\varepsilon + \tau\,\theta(a_\varepsilon) .
 \end{equation}
Since  \(\theta\) is Lipschitz continuous on \(X\) with constant \(L = \|F\|\),  due to (\ref{EvgyB1}),  we have
\begin{equation} \label{EvgyB3}
	\theta(a_\varepsilon) \le \theta(x) + L\|a_\varepsilon-x\| 
	< \theta(x) + L\bigl(d_A(x)+\varepsilon\bigr).  
\end{equation}

Now we distinguish two cases according to the sign of $\theta(x)$.

\noindent\textbf{Case 1: $\theta(x)\ge 0$.}
Then $\max\{\theta(x),0\}=\theta(x)$. Substituting (\ref{EvgyB3}) into (\ref{EvgyB2}),
\[
\begin{aligned}
	d(x,E_w) &< d_A(x)+\varepsilon + \tau\bigl(\theta(x)+L(d_A(x)+\varepsilon)\bigr) \\
	&= \tau\theta(x) + (1+\tau L)d_A(x) + (1+\tau L)\varepsilon.
\end{aligned}
\]
Since $\tau\le\gamma$ and $1+\tau L\le\gamma$, and $\theta(x)\ge0,\;d_A(x)\ge0$, we have
\[
\tau\theta(x) + (1+\tau L)d_A(x) \le \gamma\theta(x) + \gamma d_A(x) = \gamma\bigl(\theta(x)+d_A(x)\bigr).
\]
Thus
\[
d(x,E_w) < \gamma\bigl(\theta(x)+d_A(x)\bigr) + (1+\tau L)\varepsilon.
\]

\noindent\textbf{Case 2: $\theta(x)<0$.}
Then $\max\{\theta(x),0\}=0$. Due to (\ref{EvgyB3})  and   $\theta(x)<0$, we have
$\theta(a_\varepsilon) < L\bigl(d_A(x)+\varepsilon\bigr)$.
Substituting into (\ref{EvgyB2}), we have
\[
\begin{aligned}
	d(x,E_w) &< d_A(x)+\varepsilon + \tau \cdot L\bigl(d_A(x)+\varepsilon\bigr) \\
	&= (1+\tau L)d_A(x) + (1+\tau L)\varepsilon.
\end{aligned}
\]
Since $1+\tau L\le\gamma$ and in this case $\max\{\theta(x),0\}+d_A(x)=d_A(x)$, we obtain
\[
d(x,E_w) < \gamma\bigl(\max\{\theta(x),0\}+d_A(x)\bigr) + (1+\tau L)\varepsilon.
\]

Combining both cases, for any $\varepsilon>0$ we have
\[
d(x,E_w) < \gamma\bigl(\max\{\theta(x),0\}+d_A(x)\bigr) + (1+\tau L)\varepsilon.
\]
Letting $\varepsilon\to0^+$, we get
\[
d(x,E_w) \le \gamma\bigl(\max\{\theta(x),0\}+d_A(x)\bigr).
\]
 This means that (\ref{EBkXJ}) is true.

 	\noindent\textbf{(vi) $\Rightarrow$ (i)}.
 Assume that  there exists $\gamma>0$ such that (\ref{EBkXJ}) is satisfied.
  Take any $x\in A$. Then $d_A(x)=0$  and   $\theta(x)\ge0$, and so $\max\{\theta(x),0\}=\theta(x)$. Substituting into (\ref{EBkXJ}) immediately yields
 \[
 d(x,E_w)\le \gamma\,\theta(x),\quad \forall x\in A.
 \]

  	\noindent\textbf{(i) $\Rightarrow$ (vii)}.
  Suppose there exists \(\tau > 0\) such that   (\ref{EBAQJ}) is true.
 Take an arbitrary \(x \in X\). For any \(a \in A\),   it follows from  (\ref{EBAQJ})  that
 \[
 d(x, E_w) \le \|x - a\| + d(a, E_w) \le \|x - a\| + \tau \, \theta(a).
 \]
 Set \(\kappa := \max\{\tau, 1\}\). Then  
$  \|x - a\| + \tau \, \theta(a) \le \kappa \bigl( \|x - a\| + \theta(a) \bigr),$  and so
   $$d(x, E_w) \le \kappa (\|x - a\| + \theta(a)),  \quad  \forall a \in A.  $$
 Taking the infimum over \(a \in A\) yields  
 \[
 d(x, E_w) \le \kappa \inf_{a \in A} \bigl( \theta(a) + \|x - a\| \bigr) = \kappa \, \zeta(x).
 \]
 Since \(x \in X\) is arbitrary, this proves (vii).

  	\noindent\textbf{(vii) $\Rightarrow$ (i)}.
 Suppose there exists \(\kappa > 0\) such that 
 \begin{equation}\label{ErtB1}
 	d(x, E_w) \le \kappa \, \zeta(x), \quad  \forall x \in X.
 \end{equation}
 Take an arbitrary \(x \in A\). From the definition of \(\zeta\) we have  
 \[
\zeta(x) = \inf_{a \in A} \bigl( \theta(a) + \|x - a\| \bigr) \le \theta(x) + \|x - x\| = \theta(x).
 \]
 Substituting this into    (\ref{ErtB1})   gives  
 \[
 d(x, E_w) \le \kappa \, \theta(x), \quad \forall x \in A.
 \]

	 		\noindent\textbf{(i) $\Rightarrow$ (ix)}.
	 			 		Assume that there exists \(\tau>0\) such that   (\ref{EBAQJ}) is true.
	 	 	 		Take an arbitrary \(x\in A\setminus E_w\). Then \(\theta(x)>0\)  and	$	d:=d(x,E_w)>0 .$
	 		 For any \(\varepsilon>0\), there exists \(y_\varepsilon\in E_w\) satisfying
	 		\begin{equation}\label{Eapprox1}
	 			\|x-y_\varepsilon\|<d+\varepsilon .
	 		\end{equation}
	 		Since \(y_\varepsilon\in E_w\), we have \(\theta(y_\varepsilon)=0<\theta(x)\).  so \(y_\varepsilon\) belongs to the set over which the supremum defining \(m(x)\) is taken.   It follows from  (\ref{Eapprox1})   and    \(\theta(x)>0\)  that
	 		\begin{equation}\label{Eapprox2}
	 		\frac{\theta(x)-\theta(y_\varepsilon)}{\|x-y_\varepsilon\|}=\frac{\theta(x)}{\|x-y_\varepsilon\|}>\frac{\theta(x)}{d+\varepsilon}.
	 		\end{equation}
	 		 From (\ref{EBAQJ}) we obtain \(\theta(x)\ge d/\tau\), and so
	 	\begin{equation}\label{Eapprox3}
	 		\frac{\theta(x)}{d+\varepsilon}\ge\frac{d/\tau}{d+\varepsilon}=\frac{1}{\tau}\cdot\frac{d}{d+\varepsilon}.
	 		\end{equation}
	 		Therefore, for this \(y_\varepsilon\),  we conclude from  (\ref{Eapprox2})  and  (\ref{Eapprox3})  that
	 		\[
	 		m(x)\ge\frac{\theta(x)-\theta(y_\varepsilon)}{\|x-y_\varepsilon\|}>\frac{1}{\tau}\cdot\frac{d}{d+\varepsilon}.
	 		\]
	 		Letting \(\varepsilon\to0^{+}\), the right-hand side tends to \(\gamma:=\frac{1}{\tau}\),  hence \(m(x)\ge \gamma\). Since \(x\in A\setminus E_w\) was arbitrary, this proves (ix).

	 			\noindent\textbf{(ix) $\Rightarrow$ (viii)}.
	 	Suppose   that there exists   \(\gamma >0\) such that
	 	\[
	 	m(x)\ge \gamma,\quad \forall x\in A\setminus E_w.
	 	\]
	 	Set \(c :=\dfrac{\gamma}{2}>0\). For any \(x\in A\setminus E_w\), since \(m(x)\ge \gamma > c\), by the definition of supremum, there exists \(y\in A\) satisfying \(\theta(y)<\theta(x)\) and
	 	\[
	 	\frac{\theta(x)-\theta(y)}{\|x-y\|} > c.
	 	\]
	 	Rearranging yields \(\theta(y) < \theta(x) - c \|x-y\|\).   It follows from \(\theta(y)<\theta(x)\)  that    \(y\neq x\). Therefore (viii) holds (with constant \(c =\dfrac{\gamma}{2}\)).
	 	
	\noindent\textbf{(viii) $\Rightarrow$ (i)}.
	 		Assume that  there exists   \(c>0\) such that     (\ref{EevL})   holds.
	 			We will prove that for all \(x\in A\),
	 		\[
	 		d(x,E_w)\le \frac{1}{c}\,\theta(x).
	 		\]
	 		Suppose, to the contrary, that there exists \(x_0\in A\) such that  
	 		\begin{equation}\label{Econtra1}
	 			d(x_0,E_w)>\frac{1}{c}\,\theta(x_0).
	 		\end{equation}
	 	 This yields  that	 \(\theta(x_0)>0\). 
	 		Consider the restriction \(f:=\theta|_A\colon A\to[0,\infty)\) on the complete metric space \(A\) (a closed subset of a Banach space). It is clear that \(f\) is continuous  and \(\inf_{x \in A} f(x)=0\).  
	 		Set
	 		\[
	 		\varepsilon:=f(x_0)=\theta(x_0)>0,\qquad \lambda:=\frac{\varepsilon}{c}=\frac{\theta(x_0)}{c}>0 .
	 		\]
	 		Then \(f(x_0)\le \inf_{x \in A} f(x) +\varepsilon\).    Lemma  \ref{ELdaz} (Ekeland's variational principle)   yields that  there exists \(\bar x\in A\) satisfying:
	 		\begin{itemize}
	 			\item[(a)] \(f(\bar x)\le f(x_0)\);
	 			\item[(b)] \(\|\bar x-x_0\|\le \lambda\);
	 			\item[(c)] for all \(x\in A\setminus\{\bar x\}\), \(f(x) > f(\bar x)-\dfrac{\varepsilon}{\lambda}\,\|x-\bar x\|\).
	 		\end{itemize}
	 		If \(\bar x \in  E_w\), then  it follows from  (b) that
	 		\[
	 		d(x_0,E_w)\le \|x_0-\bar x\|\le \lambda = \frac{1}{c}\theta(x_0),
	 		\]
	 	which contradicts (\ref{Econtra1}). Hence   \(\bar x \notin  E_w\),  and so   \(\theta(\bar x)=f(\bar x)>0\).   (c) can be rewritten as
	 		\begin{equation}\label{Econtra2}
	 			\forall x\in A\setminus\{\bar x\},\qquad \theta(x) > \theta(\bar x) - c\,\|x-\bar x\|,
	 		\end{equation}
	 		because \(\varepsilon/\lambda = c\).
	 			 		Now apply  (\ref{EevL})   to \(\bar x\in A\setminus E_w\),  there exists \(y_0 \in A\) with \(y_0 \neq \bar x\) such that
	 		\begin{equation}\label{Econtra3}
	 			\theta(y_0 ) \le \theta(\bar x) - c\,\|\bar x-y_0\|.
	 		\end{equation}
	 It follows from  (\ref{Econtra2})  that		$\theta(y_0) > \theta(\bar x) - c\,\|y_0-\bar x\|,$
	 		which contradicts (\ref{Econtra3}). 
	 Therefore
	 		\[
	 		d(x,E_w)\le \frac{1}{c}\,\theta(x), \quad \forall x\in A,
	 		\]
	 		which is exactly statement (i) with \(\tau = 1/c\).
	 		
	\noindent\textbf{(i) $\Rightarrow$ (x)}.
Assume that  there exists    \(\tau > 0\) such that   (\ref{EBAQJ})  is satisfied.    Take any $x\in A$. 
For any $y\in E_w$,  we have
\[
\|F(x)-F(y)\|=\|F(x-y)\|\le\|F\|\,\|x-y\|.
\]
Taking the infimum over all $y\in E_w$ on both sides yields
\[
\inf_{y\in E_w}\|F(x)-F(y)\|\le\|F\|\inf_{y\in E_w}\|x-y\|.
\]
This means that
\begin{equation}\label{Efty1}
	d \bigl(F(x),F(E_w)\bigr)\le\|F\|\,d(x,E_w).  
\end{equation}
It is easy to see that $F(E_w)=\operatorname{WMin}(F(A),C)$.
Combining this with   (\ref{EBAQJ})  and  (\ref{Efty1}), we get
\[
d \bigl(F(x),\operatorname{WMin}(F(A),C)\bigr)\le\|F\|\,d(x,E_w)   \le\|F\|\,\tau\,\theta(x).
\]
Set $c:=\|F\|\,\tau >0$. Since $x\in A$ was arbitrary, (x) follows.

	 	\noindent\textbf{(iii) $\Rightarrow$ (xi)}.
	 If (iii) holds, then for all $x\in A\setminus E_w$,
	 \[
	 \frac{d_{\widehat{A}}(x)}{d(x,E_w)}\ge\frac{1}{\tau}.
	 \]
	 As $x\xrightarrow{A} E_w$, the limit inferior is at least $1/\tau>0$, so (xi) holds.		
	 		
		 	\noindent\textbf{(i) $\Rightarrow$ (xii)}.
Suppose that  there exists \(\tau > 0\) such that   (\ref{EBAQJ})  is satisfied.

Take arbitrary \(t \ge 0\) and \(x \in \Phi(t)\). Since \(x \in A\) and \(\theta(x) \le t\),   (\ref{EBAQJ})  yields
\[
d(x, E_w) \le \tau \theta(x) \le \tau t.
\]
Because \(E_w \subseteq \Phi(t)\), clearly \(\sup_{y \in E_w} d(y, \Phi(t)) = 0\). Hence the Hausdorff distance satisfies
\[
d_H(\Phi(t), E_w) = \sup_{x \in \Phi(t)} d(x, E_w) \le \tau t, \quad \forall t \ge 0.
\]
Taking \(t_0 = 1\) and \(\beta = \tau\), we have for all \(t \in [0, t_0]\) that \(d_H(\Phi(t), E_w) \le \beta t\).

	\noindent\textbf{(i) $\Rightarrow$ (xiii)}.
	By definition of $\Phi$, clearly $\Phi(0)=\theta^{-1}(0)\cap A=E_w$.
	For a fixed $x\in A$,
	\[
	\Phi^{-1}(x)=\{t\ge0:\theta(x)\le t\}=[\theta(x),+\infty),
	\]
and	so $d(0,\Phi^{-1}(x))=\inf\{|t-0|:t\ge\theta(x)\}=\theta(x)$.
	Therefore,      (\ref{EBkmJ})    is equivalent to
\begin{equation}\label{EViiiD}
	d(x,E_w)\le\kappa\,\theta(x),\quad\forall x\in U\cap A.
 \end{equation}
	If (i) holds, taking $U=X$ shows that (xiii) holds.

	\noindent\textbf{(x) $\Rightarrow$ (i)}.  Suppose that   there exist $\mu>0$ such that $\|F(x)\|\ge\mu\|x\|$ for any $x\in A-A$.    Fix an arbitrary $x\in A$.
					 By linearity of $F$, for any $y\in E_w    \subseteq   A$,  we have
			  \[
			 \|F(x)-F(y)\|=\|F(x-y)\|\ge\mu\|x-y\|,
			 \]
 and so 
		\[
		\|x-y\|\le\frac{1}{\mu}\,\|F(x)-F(y)\|.
		\]
		Taking the infimum over all $y\in E_w$, we get
		\begin{equation}\label{EfghB1}
		d(x,E_w)=\inf_{y\in E_w}\|x-y\|
		\le\inf_{y\in E_w}\frac{1}{\mu}\,\|F(x)-F(y)\|
		=\frac{1}{\mu}\,\inf_{y\in E_w}\|F(x)-F(y)\|.
	 \end{equation}
The expression $\inf_{y\in E_w}\|F(x)-F(y)\|$ is exactly $d \bigl(F(x),F(E_w)\bigr)$.  By  $F(E_w)=\operatorname{WMin}(F(A),C)$  and    (\ref{EfghB1}),  we get  
		\begin{equation}\label{EfghB2}
		d(x,E_w)\le\frac{1}{\mu}\,d \bigl(F(x),\operatorname{WMin}(F(A),C)\bigr).  
		 \end{equation}
		  (x) provides a constant $c>0$   such that    
  $$ d \bigl(F(x),\operatorname{WMin}(F(A),C)\bigr)     \le c \; \theta(x),   \quad  \forall x\in A.$$ 
	This together with   (\ref{EfghB2})    implies that  
	  we get
		\[
		d(x,E_w)\le\frac{1}{\mu}\cdot  c  \,\theta(x)=\frac{c}{\mu}\,\theta(x).
		\]
		Set $\tau:=c/\mu>0$,  then
		\[
		d(x,E_w)\le\tau\,\theta(x).
		\]
		This holds for all $x\in A$, which is exactly statement (i).

		 	\noindent\textbf{(xii) $\Rightarrow$ (i)}.
		Suppose there exist \(t_0 > 0\) and \(\beta > 0\) such that 
		\begin{equation}\label{EthhB1}
			d_H(\Phi(t), E_w) \le \beta t,  \quad  \forall t \in [0, t_0].
		\end{equation}

	Take any \(x \in A\) and set \(t = \theta(x)\). We distinguish two cases:
	
  If \(t \le t_0\), then \(x \in \Phi(t)\). By   (\ref{EthhB1}),  we have 
		\[
		d(x, E_w) \le \sup_{z \in \Phi(t)} d(z, E_w) = d_H(\Phi(t), E_w) \le \beta t = \beta \theta(x).
		\]
		
		 If \(t > t_0\), since \(A\) is bounded, let \(\eta = \operatorname{diam}(A) = \sup_{a,b \in A} \|a - b\| < +\infty\). Because \(x \in A\) and \(E_w \subseteq A\), we have \(d(x, E_w) \le \eta\). Combining this with \(t > t_0\) yields
		\[
		d(x, E_w) \le \eta < \frac{\eta}{t_0} t = \frac{\eta}{t_0} \theta(x).
		\]
 
	Setting \(\tau = \max\{\beta, \eta/t_0\}\), we obtain \(d(x, E_w) \le \tau \theta(x)\) for all \(x \in A\). 
	
	\noindent\textbf{(xi) $\Rightarrow$ (iii)}.
	From (xi), set
	\[
	\gamma := \liminf_{\substack{x\xrightarrow{A}E_w\\ x\notin E_w}} \frac{d_{\widehat{A}}(x)}{d(x,E_w)} > 0.
	\]
	Hence there exists $\eta>0$ such that for all $x\in A$ with $0<d(x,E_w)<\eta$,
	\[
	\frac{d_{\widehat{A}}(x)}{d(x,E_w)} \ge \frac{\gamma}{2} .
	\]
	Equivalently,
\begin{equation}\label{HYTz1}
	d(x,E_w) \le \frac{2}{\gamma}\, d_{\widehat{A}}(x), \qquad \forall\,x\in A,\; 0<d(x,E_w)<\eta.  
 \end{equation}
		Let $Q := \{x\in A : d(x,E_w)\ge \eta\}$. Since $A$ is compact and the distance function is continuous, $Q$ is compact.
	
	We now show that $\inf_{x\in Q} d_{\widehat{A}}(x) > 0$.
	Suppose that  $\inf_{x\in Q} d_{\widehat{A}}(x) = 0$.   By the compactness of $Q$ and continuity of $d_{\widehat{A}}$,  there exists  $\bar{x}\in Q \subseteq A$ such that
	$$d_{\widehat{A}}(\bar{x})=  \inf_{x\in Q} d_{\widehat{A}}(x) = 0 .$$
	This together with  the  closedness of  $\widehat{A}$  implies that $\bar{x}\in \widehat{A}$.
	 Moreover $\bar{x}\in A$, so $\bar{x}\in A\cap\widehat{A}=E_w$. This contradicts $d(\bar{x},E_w)\ge \eta>0$. Hence
\begin{equation}\label{HYTz2}
	\alpha := \inf_{x\in Q} d_{\widehat{A}}(x) >0.  
 \end{equation}
	
	Take any $x\in A$, and consider two cases:
 
 Case 1:  $d(x,E_w)<\eta$.  It follows from  (\ref{HYTz1})  that
		\[
		d(x,E_w)\le \frac{2}{\gamma} d_{\widehat{A}}(x).
		\]
		
	  Case 2:   $d(x,E_w)\ge \eta$, i.e., $x\in Q$. Using  (\ref{HYTz2})    and the diameter $\beta :=\operatorname{diam}(A)=\sup_{u,v\in A}\|u-v\|<\infty$,  we have 
		\[
		d(x,E_w) \le \beta \le \frac{\beta}{\alpha}\, d_{\widehat{A}}(x).
		\]
 	Taking $\tau = \max\left\{\dfrac{2}{\gamma},\; \dfrac{\beta}{\alpha}\right\}$, we obtain for all $x\in A$,
	\[
	d(x,E_w) \le \tau\, d_{\widehat{A}}(x).
	\]
	This is the global error bound (iii).
	
	\noindent\textbf{(xiii) $\Rightarrow$ (i)}.
	Assume $A$ is compact   and there exist $\kappa>0$ and an open set $U\supseteq E_w$ such that
  (\ref{EBkmJ})  is satisfied.   
	Set $H:=A\setminus U$.   Since   $A$ is compact, $U$ is open and $E_w \subseteq U$,  we get that  $H$ is compact   and $H\cap E_w= \emptyset$.
	The function $\theta$ is continuous on the compact set $H$, hence attains its minimum $m:=\min_{x\in H}\theta(x)$.
	If $m=0$, then there exists $\bar x\in H$ with $\theta(\bar x)=0$, implying $\bar x\in E_w$, contradicting $H\cap E_w= \emptyset$. Therefore $m>0$.
	
	For any $x\in H$, since $A$ is bounded, $\beta =\operatorname{diam}(A)<\infty$, and so $d(x,E_w)\le \beta$. Together with $\theta(x)\ge m$, we have 
	\begin{equation} \label{EHYZ2}
		d(x,E_w)\le \frac{\beta}{m}\,\theta(x).  
	\end{equation}
	
	Take any $x\in A$.  If $x\in U$, noting that  (\ref{EBkmJ})  is  equivalent to (\ref{EViiiD}), we get  
	   $d(x,E_w)\le\kappa\,\theta(x)$.
  
 If $x\in H$,  it follows from  (\ref{EHYZ2}) that $d(x,E_w)\le (\beta/m)\,\theta(x)$.
 	Taking $\tau:=\max\{\kappa,\; \beta/m\}$, the global error bound holds:
	\[
	d(x,E_w)\le\tau\,\theta(x),\qquad \forall x\in A.
	\]
	This completes the proof. 
	\end{proof}

 \begin{remark}   		Theorem~\ref{XLWCJK} establishes the equivalence of thirteen characterizations of the global error bound property for the concave merit function \(\theta(x)=\sup_{a\in A}(-\Delta_C(F(x)-F(a)))\) in vector optimization. Theorem~\ref{XLWCJK} significantly extends and unifies several streams of error bound theory developed in the literature. Below we compare Theorem~\ref{XLWCJK} with the results in   \cite{LiuNgYang2009, NgZheng, KL1999, NgYang2002, AC2004, Kruger, KLT2018}.
 		
 Theorem~\ref{XLWCJK} directly builds upon and substantially extends the foundational work of Liu, Ng and Yang   \cite{LiuNgYang2009}, who first introduced the merit function \(\theta\) and proved the equivalence (i) \(\Leftrightarrow\) (iii) (error bound for \(\theta\) on \(A\) is equivalent to error bound for \(d_{\widehat A}\) on \(A\)) as well as (i) \(\Leftrightarrow\) (ii) (linear regularity of \(\{A,\widehat A\}\)). Our theorem adds ten new equivalent characterizations---including uniform descent condition (iv), global slope condition (ix), asymptotic condition (xi), and sublevel set Hausdorff stability (xii)---thereby providing a comprehensive equivalence framework not present in the original work.

 Ng and Zheng \cite{NgZheng} studied error bounds for lower semicontinuous functions using directional derivatives. Their Theorem 2.5 provides a sufficient condition \(\underline{d}^+f(x)(h_x)\le -\delta\) for general functions. However, their precise characterizations—Theorems 3.1 and 3.3—are established for convex functions on reflexive Banach spaces. In particular, Theorem 3.1 shows that for convex functions, a global error bound is equivalent to the existence, for each \(x\notin S\), of a unit direction \(h_x\) with \(d^+f(x)(h_x)\le -1/\tau\) (condition (v) of Theorem 3.1). Our condition (iv) in Theorem~\ref{XLWCJK} adopts a different perspective: instead of requiring a directional derivative bound, it imposes a uniform descent condition—namely, from every \(x\in A\) one can move to some \(y\in A\) with \(\|x-y\|\le \alpha\theta(x)\) and \(\theta(y)\le \beta\theta(x)\).    Both conditions share the same underlying idea, as both require the existence of a feasible descent direction from every non-solution point.         Importantly, the proof of  (i) $\Leftrightarrow$ (iv) does not rely on the concavity   of \(\theta\),  it uses only the nonnegativity of \(\theta\) on \(A\) and standard metric arguments. Thus condition (iv) is applicable to general functions, in contrast to the directional derivative criterion in \cite{NgZheng}, which is specific to convex functions.  Moreover,  condition (iv) provides an equivalent characterization of the global error bound, without requiring differentiability, convexity, or explicit computation of directional derivatives.

 Klatte and Li  	\cite{KL1999}  studied asymptotic constraint qualifications for convex inequalities and showed equivalence among three conditions: bounded excess, Slater condition together with the asymptotic constraint qualification, and positivity of normal directional derivatives. Their work concerns general convex constraint systems, not vector optimization merit functions. Condition (xii) (sublevel set Hausdorff stability) in our theorem is conceptually related to their bounded excess condition, while our characterization via directional derivatives in Theorem  \ref{TXEvQ} ---\(\sup_{x\in A\setminus E_w}\varphi(x)<0\)---parallels their positivity of normal directional derivatives.
 		
 Ng and Yang   	\cite{NgYang2002}   studied error bounds for abstract linear inequality systems \(Ax\le b\) in Banach spaces ordered by closed convex cones. They proved that if \(C\) is polyhedral, the system always has an error bound, and characterized error bounds via angles between subspaces in Hilbert spaces. Their results are restricted to linear inequality systems, whereas our theorem applies to the nonlinear (concave) merit function arising from vector optimization, providing a much broader framework.

 Az\'e and Corvellec   \cite{AC2004}   characterized global and local error bounds for lower semicontinuous functions on complete metric spaces using the strong slope \(|\nabla f|\). In their Proposition~3.1, they proved that when \(f\) is convex, the strong slope admits the explicit representation
 		\[
 		|\nabla f|(x) = \sup_{f(z) < f(x)} \frac{f(x) - f(z)}{\|x - z\|},
 		\]
 		which is exactly the definition of the slope \(m(x)\) appearing in our condition (ix). Condition (ix) in our theorem---\(m(x)\ge \gamma\) with \(m(x)=\sup_{\theta(y)<\theta(x)}(\theta(x)-\theta(y))/\|x-y\|\)---is therefore a direct analogue of their strong slope condition, adapted to the concave merit function \(\theta\) on the convex feasible set \(A\). It is important to note, however, that while the slope formula of Az\'e and Corvellec relies on convexity, our proof of the equivalences (i) \(\Leftrightarrow\) (viii) \(\Leftrightarrow\) (ix) uses neither convexity nor concavity of the function. In fact, these implications follow from purely metric and variational arguments (Ekeland's principle) and hold for any lower semicontinuous function on a complete metric space. Thus, our result   provides a general principle that subsumes both convex and concave cases.

 In his foundational paper, Kruger   \cite{Kruger}   develops a systematic theory of error bounds and metric subregularity for set-valued mappings between general metric or Banach spaces. The criteria are expressed in terms of various primal and subdifferential slopes, and they are primarily \emph{local} in nature. In contrast, Theorem~\ref{XLWCJK} provides a comprehensive \emph{global} characterization specifically tailored to the vector optimization setting, where the merit function is concave and its zero set coincides with the weakly efficient solution set \(E_w\). 
 	Notably, conditions (viii) and (ix)---the descent and slope conditions---are direct analogues of the primal slope criteria in    \cite{Kruger}, which are defined on arbitrary metric spaces. Since the proofs of these two conditions do not depend on concavity, they effectively transplant the general global error bound slope criteria to the feasible set \(A\) with \(E_w\) as the zero-level set. Thus, the theorem successfully embeds the general theory into vector optimization without being limited to concave functions. Moreover, condition (v) (perturbation stability) and condition (xiii) (inverse sublevel-set characterization) are new global stability features in the vector optimization context, enriching the framework of    \cite{Kruger} which mainly focused on local slopes.
 		
 		The perturbation paper \cite{KLT2018} extends the developments in Kruger, Ngai and Th\'era      \cite{KNT2010} and characterizes the stability of local and global error bounds under data perturbations in the Banach space setting. It introduces new concepts of arbitrary, convex and linear perturbations of the function defining the constraint system, and interprets the characterizations as estimates of the ``radius of error bounds''. 	Theorem~\ref{XLWCJK} incorporates perturbation-type characterizations in conditions (v) and (xiii): (v) asserts stability of the distance to \(E_w\) under arbitrary displacements \(e\in X\), while (xiii) expresses the error bound in terms of the inverse image of the sublevel mapping \(\Phi\). These conditions are novel in the context of vector optimization merit functions and directly mirror the perturbation philosophy of  \cite{KLT2018}. However, whereas the perturbation paper focuses on general convex functions and their subdifferentials, Theorem~\ref{XLWCJK} establishes these perturbation characterizations for the specific concave merit function \(\theta\), leveraging its Lipschitz continuity and concavity. Furthermore, condition (xii)---the Hausdorff distance between sublevel sets and the solution set---provides a quantitative stability estimate that is analogous to the ``radius of error bounds'' interpretation in the perturbation paper \cite{KLT2018}.

 	Theorem~\ref{XLWCJK} introduces new characterizations---including global slope (ix), perturbation stability (v), extended domain error bound (vi) and (vii), asymptotic   condition (xi), and sublevel set stability (xii)---that are novel in the vector optimization context. Moreover, Theorem~\ref{XLWCJK} establishes partial implications, including (i)\(\Rightarrow\)(x), (iii)\(\Rightarrow\)(xi), (i)\(\Rightarrow\)(xii), (i)\(\Rightarrow\)(xiii), and conditional converses under boundedness of $A$, compactness of $A$, or injectivity assumptions on \(F\). No single previous work contains all these equivalences. Theorem~\ref{XLWCJK} thus represents a significant consolidation and extension of the error bound theory for vector optimization.
   \end{remark}

\begin{definition}
	$\theta$ is said to have a  local error bound at $x_0 \in E_w$, if there exist $\tau > 0$  and $\delta > 0$ such that
	\[
	d(x, E_w) \leq \tau \theta (x), \qquad \forall x \in A \cap B(x_0,\delta).
	\]
	
\end{definition}
	In keeping with the distinction from local error bounds, the term error bound is occasionally called a global error bound

 We now turn to the relationship between local and global error bounds.

\begin{theorem} 
	\label{thm:localglobal}
	Let $A\subset X$ be a nonempty  compact  set. 
	Then the following three statements are equivalent:
	\begin{enumerate}
		\item[(i)]  $\theta$ has a global error bound  on $A$: there exists  $\tau>0$ such that
		\[
		d(x,E_w)\le \tau\,\theta(x),\quad \forall x\in A .
		\]
		\item[(ii)]  $\theta$ has a uniform local error bound: there exist  $\tau>0$ and $\delta>0$ such that for  every $\bar x\in E_w$,
	$$		d(x,E_w)\le \tau\,\theta(x),\quad \forall x\in A\cap B(\bar x,\delta).	$$
			\item[(iii)] $\theta$ has a local error bound at each $\bar x\in E_w$.
	\end{enumerate}
\end{theorem}

\begin{proof}
 The implications (i) $\Rightarrow$ (ii) and (ii) $\Rightarrow$ (iii) follow directly from the definitions.
	
(iii) $\Rightarrow$ (i).
	For any $\bar x\in E_w$, since $\theta$ has a local error bound at $\bar x$, there exist  $\delta_{\bar x} >0$ and $\kappa_{\bar x}>0$ such that
\[
d(x,E_w)\le \kappa_{\bar x}\,\theta(x), \quad \forall x\in A\cap B(\bar x,\delta_{\bar x}).
\]
Because $E_w\subseteq A$ and $A$ is compact,  we get that  $E_w$ is compact. There exist finitely many points $\bar x_1,\dots,\bar x_m\in E_w$ such that
\[
E_w\subseteq\bigcup_{i=1}^{m}B(\bar x_i,\delta_i/2),\qquad \delta_i:=\delta_{\bar x_i}.
\]
Set $\delta:=\min_i\delta_i/2>0$ and $\kappa:=\max_i\kappa_{\bar x_i}$. Now take any $x\in A$ satisfying $d(x,E_w)<\delta$. By the definition of distance, there exists $y\in E_w$ with $\|x-y\|<\delta$. Since $y$ belongs to some $B(\bar x_i,\delta_i/2)$, we have
\[
\|x-\bar x_i\|\le\|x-y\|+\|y-\bar x_i\|<\delta+\frac{\delta_i}{2}\le\delta_i,
\]
hence $x\in B(\bar x_i,\delta_i)$, which yields
\[
d(x,E_w)\le \kappa_{\bar x_i}\,\theta(x)\le\kappa\,\theta(x).
\]
Thus,
\begin{equation} \label{EglZS1}
	d(x,E_w)\le \kappa\,\theta(x),\qquad \forall x\in A,\ d(x,E_w)<\delta.  
\end{equation}

	Now take any $x\in A$. We distinguish two cases.
	
	\noindent\textbf{Case 1:} $d(x,E_w)<\delta$.  From (\ref{EglZS1}) we immediately obtain
$	d(x,E_w)\le \kappa\,\theta(x).$
	
	\noindent\textbf{Case 2:} $d(x,E_w)\ge \delta$. Set
$	H:=\bigl\{x\in A : d(x,E_w)\ge \delta\bigr\}.$
Since	$A$ is compact and $d(\cdot,E_w)$ is continuous, we get that $H$ is compact. Clearly $H\cap E_w=\varnothing$ and $\theta(x)>0$ for all $x\in H$. 
This together with the continuity of $\theta$  implies that
 	\[
	m:=\min_{x\in H}\theta(x)>0.
	\]
	On the other hand, $A$ is bounded, and its diameter $\beta:=\operatorname{diam}(A)=\sup_{u,v\in A}\|u-v\|<+\infty$. For $x\in H$, we have $d(x,E_w)\le \beta$, which together with $\theta(x)\ge m$ yields
	\[
	d(x,E_w)\le \beta \le \frac{\beta}{m}\,\theta(x).
	\]
	
	Combining the two cases, set
	\[
	\tau := \max\Bigl\{\kappa,\;\frac{\beta}{m}\Bigr\},
	\]
	then for every $x\in A$,
$	d(x,E_w)\le \tau \,\theta(x).$
\end{proof}

\section{Directional derivatives of the merit   function  }

By Theorem~\ref{YL:thetaprop2}, \(\theta\) is Lipschitz continuous and concave on \(X\). Consequently, for every \(x\in X\) and every direction \(d\in X\), the   directional derivative  
 \[
  \theta'(x;d):=\lim_{t\downarrow 0}\frac{\theta(x+td)-\theta(x)}{t}
  \]  
  is well‐defined and finite.

The goal of this section is to provide a thorough analysis of this directional derivative, including its explicit representation, its superdifferential properties, and its connections to the tangent cones of the feasible set and the weakly efficient solution set.

\begin{theorem}  \label{JBXZDS} 
Let $x \in X$.	The following statements hold:
	\begin{itemize}
		\item[(i)]  For   any direction $d \in X$, the directional derivative  $	\theta'(x; d)$  exists. 
		\item[(ii)] The map $d \mapsto \theta'( x; d)$ is a Lipschitz continuous   function on $X$ with the   constant $\|F\|$.
			\item[(iii)] For all $d_1, d_2 \in X$ and $\alpha \ge 0$,
			\[
			\theta'(x; d_1+d_2) \ge \theta'( x; d_1) + \theta'( x; d_2), \quad
			\theta'( x; \alpha d_1) = \alpha \,\theta'( x; d_1).
			\]
	In		particular,   the map $d \mapsto \theta'( x; d)$ is a   concave function on $X$.
	\end{itemize}
\end{theorem}

\begin{proof}
By Theorem \ref{YL:thetaprop2},	we know that $\theta:X\to\mathbb{R}$ is a Lipschitz continuous concave function; i.e., there exists a constant $L=\|F\|>0$   such that
	\begin{equation}\label{eq:lip}
		|\theta(x)-\theta(y)|\le L\|x-y\|,\quad \forall x,y\in X,
	\end{equation}
	and for any $x,y\in X$ and $\lambda\in[0,1]$,
	\begin{equation}\label{eq:concave}
		\theta\bigl((1-\lambda)x+\lambda y\bigr)\ge (1-\lambda)\theta(x)+\lambda\theta(y).
	\end{equation}

(i).	Fix an arbitrary point $x\in X$ and a direction $d\in X$.
		For $t>0$,  define the difference quotient
	\[
	Q(t):=\frac{\theta(x+td)-\theta(x)}{t}.
	\]
	
	We first show that $ Q(\cdot)$ is nonincreasing on $(0,\infty)$. For $0<t_1<t_2$, set $\lambda:=\frac{t_1}{t_2}\in(0,1)$, then
	\[
	x+t_1d = (1-\lambda)x+\lambda(x+t_2d).
	\]
   It follows from  (\ref{eq:concave})  that
	\[
	\theta(x+t_1d) \ge (1-\lambda)\theta(x)+\lambda\theta(x+t_2d),
	\]
	and so
	\[
	\theta(x+t_1d)-\theta(x) \ge \lambda\bigl(\theta(x+t_2d)-\theta(x)\bigr)
	= \frac{t_1}{t_2}\bigl(\theta(x+t_2d)-\theta(x)\bigr).
	\]
	Dividing both sides by $t_1$ gives
	\[
	Q(t_1) = \frac{\theta(x+t_1d)-\theta(x)}{t_1}
	\ge \frac{\theta(x+t_2d)-\theta(x)}{t_2} = Q(t_2).
	\]
 
	On the other hand, by the Lipschitz condition \eqref{eq:lip}, $Q(t)\le  L\|d\|$ for all $t>0$,
	so the directional derivative 
	\[
	\theta'(x;d) =\sup_{t > 0}Q(t) =\lim_{t\downarrow0}Q(t) =\lim_{t\downarrow0}\frac{\theta(x+td)-\theta(x)}{t}
	\]
	exists and is finite.

(ii).	For any $d_1,d_2\in X$ and $t>0$, by \eqref{eq:lip},
	\[
	\begin{aligned}
		\bigl|Q_{d_1}(t)-Q_{d_2}(t)\bigr|
		&= \frac{\bigl|\theta(x+td_1)-\theta(x+td_2)\bigr|}{t}\\
		&\le \frac{L\|t d_1-t d_2\|}{t} = L\|d_1-d_2\|.
	\end{aligned}
	\]
	Letting $t\downarrow0$, we obtain
	\[
	\bigl|\theta'(x;d_1)-\theta'(x;d_2)\bigr| \le L\|d_1-d_2\|,\qquad \forall d_1,d_2\in X.
	\]

(iii).	Let $\alpha\ge0$. If $\alpha=0$, then $\theta'(x;0)=\lim_{t\downarrow0}0=0$. If $\alpha>0$, make the change of variable $s=\alpha t$; as $t\downarrow0$, $s\downarrow0$, and
	\[
	\theta'(x;\alpha d) = \lim_{t\downarrow0}\frac{\theta(x+t\alpha d)-\theta(x)}{t}
	= \alpha \lim_{t\downarrow0}\frac{\theta(x+(\alpha t)d)-\theta(x)}{\alpha t}
	= \alpha \lim_{s\downarrow0}\frac{\theta(x+s d)-\theta(x)}{s}
	= \alpha\,\theta'(x;d).
	\]

	For any $d_1,d_2\in X$,   it follows from  (\ref{eq:concave})  that
	\[
	\begin{aligned}
		\theta\bigl(x+t(d_1+d_2)\bigr)
		&= \theta\Bigl(\frac{1}{2}(x+2td_1)+\frac{1}{2}(x+2td_2)\Bigr)\\
		&\ge \frac{1}{2}\theta(x+2td_1) + \frac{1}{2}\theta(x+2td_2).
	\end{aligned}
	\]
	Subtract $\theta(x)$ and divide by $t>0$:
	\[
	\frac{\theta(x+t(d_1+d_2))-\theta(x)}{t}
	\ge \frac{\theta(x+2td_1)-\theta(x)}{2t} + \frac{\theta(x+2td_2)-\theta(x)}{2t}.
	\]
	Letting $t\downarrow0$, the two terms on the right tend to $\theta'(x;d_1)$ and $\theta'(x;d_2)$, respectively,
	while the left-hand side tends to $\theta'(x;d_1+d_2)$. This yields the superadditivity
	$$		\theta'(x;d_1+d_2) \ge \theta'(x;d_1) + \theta'(x;d_2),\qquad \forall d_1,d_2\in X.	$$
	Together with positive homogeneity, for any $\lambda\in[0,1]$,
	\[
	\theta'(x;\lambda d_1+(1-\lambda)d_2)
	\ge \theta'(x;\lambda d_1) + \theta'(x;(1-\lambda)d_2)
	= \lambda\theta'(x;d_1) + (1-\lambda)\theta'(x;d_2).
	\]
	Hence the map $d\mapsto\theta'(x;d)$ is concave.
	\end{proof}

\begin{remark}    \label{RsxfQ} For \(t > 0\), define the difference quotient function
	\[
	q_t(x, d) = \frac{\theta(x + t d) - \theta(x)}{t},   \quad  \forall (x, d) \in X \times X.
	\]
	Since \(\theta\) is continuous, each \(q_t\) is continuous on \(X \times X\).     In view of  the proof of Theorem  \ref{JBXZDS} (i), for any fixed \((x, d)\), the function \(t \mapsto q_t(x,d)\) is nonincreasing on \((0,\infty)\). Hence the directional derivative can be expressed as a supremum:
	\[
	\theta'(x; d) = \lim_{t \downarrow 0} q_t(x,d) = \sup_{t > 0} q_t(x,d).
	\]
	As a supremum of a family of continuous functions, the mapping \((x,d) \mapsto \theta'(x;d)\) is lower semicontinuous.
 \end{remark}

Fix $\bar{x} \in E_w$ and define the function $\Phi : X \to \mathbb{R}$ by $\Phi(d) = \theta'(\bar{x}; d)$ for every $d \in X$.  
By Theorem~\ref{JBXZDS}(iii), $\Phi$ is concave on $X$.  
The superdifferential of the concave function $\Phi$ at a point $d \in X$ is defined as
\[
\partial^+ \Phi(d) = \{ x^* \in X^* : \Phi(y) \le \Phi(d) + \langle x^*, y-d \rangle,\ \forall y \in X \}.
\]

\begin{theorem} \label{AXDSS}
	Let $A \subseteq X$ be a nonempty compact convex set. For any $(v,x) \in K \times X$, define  
	$f_v(x) = \langle F^*v, x \rangle + \sigma_A(-F^*v)$.  
	Then, for every $x \in X$ and every direction $d \in X$, the directional derivative $\theta'(x;d)$ can be expressed as
	\begin{equation}\label{eq:deriV1}
		\theta'(x;d) = \min_{v \in K(x)} \langle F^*v, d \rangle,
	\end{equation}
	where the active index set  
	$K(x) := \{ v \in K : \theta(x) = f_v(x) \}$  
	is nonempty and weak$^*$-compact.
	
	Moreover, for the point $\bar{x} \in E_w$ fixed above, the superdifferential of $\Phi$ admits the following explicit descriptions:
	\begin{itemize}
		\item[(i)] at $0 \in X$,
		\[
		\partial^+ \Phi(0) = \operatorname{co}\bigl(F^*(K(\bar{x}))\bigr) = F^*(K(\bar{x}));
		\]
		\item[(ii)] at an arbitrary $d \in X$,
		\[
		\partial^+ \Phi(d) = \{ x^* \in F^*(K(\bar{x})) : \langle x^*, d \rangle = \Phi(d) \}.
		\]
	\end{itemize}
\end{theorem}

\begin{proof}   Let    $x\in X$ and direction $d\in X$.    To derive \eqref{eq:deriV1},   by  Theorem \ref{YDOBS},  for any $t \ge   0$,  set
	\[
\psi(t):=\theta(x+td)=\min_{v\in K}\bigl(f_v(x)+t\langle F^*v,d\rangle\bigr).
	\]
	Fix $t>0$. For each $v\in K(x)$, we have $\psi(0)=\theta(x)=f_v(x)$, and hence
	\[
	\frac{\psi(t)-\psi(0)}{t}\le\frac{f_v(x)+t\langle F^*v,d\rangle-f_v(x)}{t}=\langle F^*v,d\rangle .
	\]
	Thus
	\begin{equation}\label{eq:deriV2}
		\limsup_{t\downarrow0}\frac{\psi(t)-\psi(0)}{t}\le\min_{v\in K(x)}\langle F^*v,d\rangle .
	\end{equation}
	On the other hand, for each $t$ there exists $v_t\in K$ such that $\psi(t)= f_{v_t}(x)+t\langle F^*v_t,d\rangle$. This together with  $f_{v_t}(x)\ge  \psi(0)$  implies that
	\begin{equation}\label{eq:deriV3}
	\frac{\psi(t)-\psi(0)}{t}
	= \frac{f_{v_t}(x)+t\langle F^*v_t,d\rangle- \psi(0)}{t}
	= \langle F^*v_t,d\rangle +\frac{f_{v_t}(x)- \psi(0)}{t}  \ge  \langle F^*v_t,d\rangle.
		\end{equation}
 	Take a sequence $t_n\downarrow0$ and set $v_n:=v_{t_n}$. By weak$^*$-compactness of $K$, there exists a convergent subnet; assume without loss of generality that $v_n\xrightarrow{w^*} \bar v\in K$. Because $\psi$ is continuous and $t_n\downarrow0$,  we get that   $\psi (t_n)\to \psi (0)$, and
	$f_{v_n}(x)= \psi (t_n)-t_n\langle F^*v_n,d\rangle\to \psi(0)$.
	By weak$^*$-continuity of $v\mapsto f_v(x)$, we have $f_{v_n}(x) \to f_{\bar v}(x)$. Hence
	$f_{\bar v}(x)= \psi(0)$, which means $\bar v\in K(x)$. Moreover, $\langle F^*v_n,d\rangle\to\langle F^*\bar v,d\rangle$. Taking the limit inferior in (\ref{eq:deriV3}) gives
	\begin{equation}\label{eq:deriV4}
		\liminf_{t\downarrow0}\frac{\psi(t)- \psi(0)}{t}\ge\langle F^*\bar v,d\rangle\ge\min_{v\in K(x)}\langle F^*v,d\rangle .
	\end{equation}
	Combining (\ref{eq:deriV2}) and (\ref{eq:deriV4}), we obtain
	\[
	\theta'(x;d)= \lim_{t\downarrow0} \frac{\psi (t)- \psi(0)}{t} = \min_{v\in K(x)}\langle F^*v,d\rangle .
	\]
	It is clear that 
\begin{equation}\label{EderiV1}
	\partial^+ \Phi(0) = \{ x^* \in X^* : \Phi(d) \le \langle x^*, d \rangle,\ \forall d \in X \}.
	\end{equation}
	
	Take any $v \in K(\bar{x})$ and set $x^* = F^*v$. By (\ref{eq:deriV1}), $\Phi(d) \le \langle F^*v, d \rangle$ for all $d \in X$, so  it follows from  (\ref{EderiV1})  that     $x^* \in \partial^+ \Phi(0)$. Thus $F^*(K(\bar{x})) \subseteq \partial^+ \Phi(0)$. Since $K(\bar{x})$ is a weak$^*$-compact convex set and $F^*: Y^* \to X^*$ is weak$^*$-weak$^*$ continuous and linear, its image $F^*(K(\bar{x}))$ is a weak$^*$-compact convex set in $X^*$.   This implies   that  $ \operatorname{co}\bigl(F^*(K(\bar{x}))\bigr) = F^*(K(\bar{x}))$.
	Since	  $\partial^+ \Phi(0)$ is   convex and weak$^*$-closed,  we have
\begin{equation}\label{EderiV2}
	\operatorname{co}(F^*(K(\bar{x}))) \subseteq \partial^+ \Phi(0).
	\end{equation}
	
	Let $x^* \in \partial^+ \Phi(0)$. We need to show $x^* \in F^*(K(\bar{x}))$. Suppose, to the contrary, that $x^* \notin F^*(K(\bar{x}))$. Since $F^*(K(\bar{x}))$ is a weak$^*$-compact convex set,  by the Hahn-Banach separation theorem,  there exist a nonzero $d_0 \in X$ and a real number $\alpha$ such that
	\[
	\langle x^*, d_0 \rangle < \alpha \le \langle y^*, d_0 \rangle, \quad \forall y^* \in F^*(K(\bar{x})).
	\]
	In particular,
	\[
	\langle x^*, d_0 \rangle < \inf_{v \in K(\bar{x})} \langle F^*v, d_0 \rangle = \Phi(d_0).
	\]
	But this contradicts (\ref{EderiV1}). Hence $x^* \in F^*(K(\bar{x}))$, and so  $\partial^+ \Phi(0) \subseteq	 F^*(K(\bar{x})) .$
	 Combining this with (\ref{EderiV2}), we get
	\begin{equation}\label{eq:deriV5}
		\partial^+ \Phi(0) = F^*(W(\bar{x})) = \operatorname{co}\bigl(F^*(W(\bar{x}))\bigr).
	\end{equation}
	
	For   any   $d \in X$, we prove that    
	\begin{equation}\label{eq:deriV6}
		\partial^+ \Phi(d) = \{ x^* \in \partial^+ \Phi(0) : \langle x^*, d \rangle = \Phi(d) \}.
	\end{equation}
	Indeed, let $x^* \in \partial^+ \Phi(d)$. Then for all $y \in X$, $\Phi(y) \le \Phi(d) + \langle x^*, y-d \rangle$. Taking $y=0$ and using $\Phi(0)=0$ gives $0 \le \Phi(d) - \langle x^*, d \rangle$, i.e., $\langle x^*, d \rangle \le \Phi(d)$. Taking $y=2d$ and using positive homogeneity $\Phi(2d)=2\Phi(d)$ gives $2\Phi(d) \le \Phi(d) + \langle x^*, d \rangle$, i.e., $\Phi(d) \le \langle x^*, d \rangle$. Hence $\langle x^*, d \rangle = \Phi(d)$. Moreover, for any $y \in X$, $\Phi(y) \le \Phi(d) + \langle x^*, y-d \rangle = \langle x^*, d \rangle + \langle x^*, y-d \rangle = \langle x^*, y \rangle$, so $x^* \in \partial^+ \Phi(0)$.
	
	Conversely, if $x^* \in \partial^+ \Phi(0)$ and $\langle x^*, d \rangle = \Phi(d)$, then for any $y$, $\Phi(y) \le \langle x^*, y \rangle = \Phi(d) + \langle x^*, y-d \rangle$,   and so $x^* \in \partial^+ \Phi(d)$.
	Thus (\ref{eq:deriV6}) holds.
		Substituting (\ref{eq:deriV5}) into (\ref{eq:deriV6}) yields
	\[
	\partial^+ \Phi(d) = \left\{ x^* \in F^*(W(\bar{x})) : \langle x^*, d \rangle = \Phi(d) \right\}.
	\]
\end{proof}

\begin{theorem}  \label{TTXEQ} 
	Let $A \subseteq X$ be a nonempty convex compact set  and  $\bar x \in E_w$. Set
	\[
	W(\bar{x}) = \bigl\{ y^* \in K : F^* y^* \in -N_A(\bar{x}) \bigr\},
	\]
 	Let $W(\bar{x})^+ = \{z \in Y : \langle y^*, z \rangle \ge 0,\ \forall y^* \in W(\bar{x})\}$ be the dual cone of $W(\bar{x})$.
	Then  the following statements hold:
		\begin{itemize}
			\item[(i)]   For any direction $d \in X$,
		\begin{equation}\label{DZQW+}
		\theta'(\bar{x}; d) = \min_{y^* \in W(\bar{x})} \langle y^*, F(d) \rangle.  
	\end{equation}
		
		\item[(ii)]  $W(\bar{x})$ is a nonempty weak$^*$-compact convex set, and satisfies the ordering cone inclusion
		\[
		C \subseteq W(\bar{x})^+ \subseteq Y.  
		\]
		
\item[(iii)]  For any direction $d \in X$,
		\[
		\theta'(\bar{x}; d) < 0 \quad \Longleftrightarrow \quad F(d) \notin W(\bar{x})^+.   
		\]
	\item[(iv)]  The zero-directional-derivative cone $D(\bar{x}):=\{d\in X:\theta'(\bar{x};d)=0\}$ satisfies
		\[
		D(\bar{x})=\Big\{d\in X:\min_{y^{*}\in W(\bar{x})}\langle y^{*},F(d)\rangle=0\Big\}
		=F^{-1}\Big(\operatorname{bd}\big(W(\bar{x})^{+}\big)\Big),
		\]
		where $\operatorname{bd}$ denotes the topological boundary in $Y$.
\end{itemize}
\end{theorem}

\begin{proof}
(i).	Since $\bar{x} \in E_w$, we have $\theta(\bar{x}) = 0$.
	By Theorem \ref{AXDSS},
	\begin{equation}\label{DZQW1}
		\theta'(\bar{x} ;d)=\min_{v\in K(\bar{x})}\langle F^*v,d\rangle ,
	\end{equation}
	where $K(\bar{x}) = \{ v \in K : 0 = \theta(\bar{x}) = f_v(\bar{x}) \}$ and $f_v(\bar{x})= \langle F^*v,\bar{x} \rangle+\sigma_A(-F^*v)$.
	
	Clearly,
	\begin{equation}\label{DZQW2}
	f_v( \bar{x} ) = 0 \iff \langle F^* v, \bar{x} \rangle = -\sigma_A(-F^*v).
	\end{equation}
	By definition of support function, $-\sigma_A(-F^* v) = \inf_{a \in A} \langle F^* v, a \rangle$. Then (\ref{DZQW2}) is equivalent to
	\begin{equation}\label{DZQW3}
	\langle F^* v, \bar{x} \rangle = \inf_{a \in A} \langle F^* v, a \rangle.
	\end{equation}
	  It is easy to see that (\ref{DZQW3}) is equivalent to $-F^* v \in N_A(\bar{x})$, i.e., $F^* v \in -N_A(\bar{x})$. Therefore,
	\[
	K(\bar{x}) = \{ v \in K : F^* v \in -N_A(\bar{x}) \} =: W(\bar{x}).
	\]
	Substituting into (\ref{DZQW1}) gives
	\[
	\theta'(\bar{x}; d) = \min_{v \in W(\bar{x})} \langle v, F(d) \rangle.
	\]

	(ii).    It follows from    Theorem  \ref{YDOBS} and $\bar{x} \in E_w$  that $\theta(\bar{x}) = \max_{v \in K} f_v(\bar{x}) = 0$.  This means that  there exists $v_0 \in K$ with $f_{v_0}(\bar{x}) = 0$, i.e., $v_0 \in K(\bar{x}) = W(\bar{x})$,  which implies that $W(\bar{x})$ is nonempty.
	Since $K$ is a weak$^*$-compact convex set, and $F^*$ is   linear and continuous, it is easy to see that 
	$W(\bar{x}) = \bigl\{ y^* \in K : F^* y^* \in -N_A(\bar{x}) \bigr\}$ is a weak$^*$-compact convex set.

  Let $v \in W(\bar{x}) \subseteq K = \operatorname{co}(S(C^+)) \subseteq C^+$. For any $c \in C$, from $v \in C^+$ we have $\langle v, c \rangle \ge 0$. Hence $c \in W(\bar{x})^+$, and so $C \subseteq W(\bar{x})^+$.

		(iii).  For any direction $d \in X$, we conclude from    (\ref{DZQW+})  that
	\[
	\theta'(\bar{x}; d) < 0
	\;\Longleftrightarrow\;
	\exists\, v \in W(\bar{x}) \text{ such that } \langle v, F(d) \rangle < 0
	\;\Longleftrightarrow\;
	F(d) \notin W(\bar{x})^+.
	\]
 
(iv).		From (\ref{DZQW+}), $D(\bar{x})=\{d\in X : \min_{v\in W(\bar{x})}\langle v,F(d)\rangle =0\}$. Define $\psi(y):=\min_{v\in W(\bar{x})}\langle v,y\rangle$ for $y\in Y$.
		Since $W(\bar{x})\subseteq K\subseteq B_{Y^{*}}$, $\psi$ is a $1$-Lipschitz continuous function, and positively homogeneous. The positive polar cone $W(\bar{x})^{+}=\{y:\psi(y)\ge0\}$ is a closed convex cone.
	
Next, we prove that
	\begin{equation}\label{DZQW4}
 \operatorname{int}(W(\bar{x})^{+}) = \{y\in Y:\psi(y)>0\}.
	 \end{equation}
 In fact,    if $\psi(y)>0$, by continuity of $\psi$ there exists an open ball $B(y,\varepsilon)$ such that $\psi (z) \ge 0$ for any $z \in B(y,\varepsilon)$, and so
  $B(y,\varepsilon)\subseteq W(\bar{x})^{+}$.     This means that   $y \in  \operatorname{int}(W(\bar{x})^{+})$.  
	
	Conversely, let $y\in\operatorname{int}(W(\bar{x})^{+})$. Then there exists $\varepsilon>0$ such that $B(y,\varepsilon)\subseteq W(\bar{x})^{+}$. Suppose, to the contrary, that there exists $v_0\in W(\bar{x})$ with $\langle v_0,y\rangle=0$. 
		By    Proposition   \ref{RDCxl}, we get \(0 \notin K\). This together with  $W(\bar{x})\subseteq K$  implies that    $v_0\neq0$.
	There exists $y_0 \in Y$ such that $\langle v_0, y_0 \rangle \ne 0$. Without loss of generality, assume $\langle v_0, y_0 \rangle < 0$.
	Set $z:=y+ \delta y_0$ with $0<\delta<\varepsilon/\|y_0\|$. Then $\|z-y\|=\delta\|y_0\|<\varepsilon$, so $z\in B(y,\varepsilon)\subseteq W(\bar{x})^{+}$. But
	\[
	\langle v_0,z\rangle = \langle v_0,y\rangle + \delta \langle v_0,y_0\rangle = \delta \langle v_0,y_0\rangle < 0,
	\]
	contradicting $z\in W(\bar{x})^{+}$. Hence $\langle v,y\rangle>0$ for all $v\in W(\bar{x})$. By weak$^*$-compactness of $W(\bar{x})$ and weak$^*$-continuity of $v\mapsto\langle v,y\rangle$, the minimum $\psi(y)>0$. Therefore, (\ref{DZQW4}) is true.
	
In view of  (\ref{DZQW4}), the boundary is
	\[
	\operatorname{bd}(W(\bar{x})^{+}) = W(\bar{x})^{+}\setminus\operatorname{int}(W(\bar{x})^{+})
	= \{y:\psi(y)\ge0,\ \psi(y)\not>0\} = \{y:\psi(y)=0\}.
	\]
	Hence the condition $\min_{v\in W(\bar{x})}\langle v,F(d)\rangle=0$ is equivalent to $\psi(F(d))=0$, i.e., $F(d)\in\operatorname{bd}(W(\bar{x})^{+})$. Therefore,
$	D(\bar{x}) = F^{-1}\Big(\operatorname{bd}\big(W(\bar{x})^{+}\big)\Big). $
\end{proof}

\begin{theorem}\label{ZSQE}
	Let $A \subseteq X$ be a nonempty convex compact set.
	Then for any $x\in A$,  
	we have
	\[
	x\in E_w \quad\Longleftrightarrow\quad \theta'(x;d)\ge 0,\ \forall d\in T_A(x).
	\]
\end{theorem}

\begin{proof}
	  $ \Rightarrow $.
		Let $x\in E_w$. Then $\theta(x)=0$ and $\theta(y)\ge 0$ for all $y\in A$. Take any $d\in T_A(x)$. Then there exist sequences $t_n\downarrow0$ and $d_n\to d$ such that $x_n:=x+t_n d_n\in A$. By Lipschitz continuity of $\theta$,
	\[
	0\le\frac{\theta(x_n)}{t_n}= \frac{\theta(x+t_n d_n)-\theta(x)}{t_n}
	\le \frac{\theta(x+t_n d)-\theta(x)}{t_n}+\|F\|\,\|d_n-d\|.
	\]
	Letting $n\to\infty$, the first term on the right tends to $\theta'(x;d)$, and the second term tends to $0$. Hence $\theta'(x;d)\ge 0$.  
	
	$ \Leftarrow $. 
 Assume that $x\in A$  and
	\begin{equation}\label{DAE1}
		\theta'(x;d)\ge 0,\ \; \forall d\in T_A(x).
	\end{equation}
	Suppose $x \notin E_w$, i.e., $x\in A\setminus E_w$. Since $A \subseteq X$ is nonempty compact, by Proposition 2.4.10 of \cite{Luc}, $F(A)$ satisfies the weak domination property (WDP).  By   Remark \ref{RngXEQ},     there exists $y\in E_w$ such that
	\[
	F(x)-F(y)\in\operatorname{int}C .
	\]
	Set $d_0:=y-x$. Since $y\in A$ and $A$ is convex, $d_0 \in T_A(x)$. By Theorem \ref{AXDSS},
	\begin{equation}\label{DAE+1}
	\theta'(x;d_0)=\min_{v\in K(x)}\langle F^*v,d_0 \rangle .
	\end{equation}
	where $K(x):=\{v\in K:\theta(x)=f_v(x)\}$  and   $f_v(x)= \langle F^*v,x\rangle+\sigma_A(-F^*v)$.
	Take any $v\in K(x)\subset K\subset C^+\setminus\{0\}$,  then
\begin{equation}\label{DAE2}
	\langle F^*v,d_0 \rangle=\langle v,F(y)-F(x)\rangle=-\langle v,F(x)-F(y)\rangle .
	\end{equation}
	Since $F(x)-F(y)\in\operatorname{int}C$ and $v\in C^+\setminus\{0\}$, we have $\langle v,F(x)-F(y)\rangle>0$. It follows from  (\ref{DAE2})  that $\langle F^*v,d_0 \rangle<0$. This together with   (\ref{DAE+1}) implies that  $\theta'(x;d_0)<0$, contradicting (\ref{DAE1}). Thus $x\in E_w$.  
\end{proof}

The following lemma can be derived from Danskin's theorem. For the reader's convenience, we present its proof.

\begin{lemma}  \label{LXcEQ} 
	Let $Q \subseteq Y^*$ be a weak$^*$-compact set and let $c, g : Q \to \mathbb{R}$ be weak$^*$-continuous functions. For every $t \ge 0$, define
	\[
	\lambda(t) = \min_{y^* \in Q} \bigl[ c(y^*) + t g(y^*) \bigr].
	\]
	Note that $\lambda(0) = \min_{y^*\in Q} c(y^*)$, and set
	$M = \{ y^* \in Q : c(y^*) = \lambda(0) \}$.
	The weak$^*$-continuity of $c$ and the weak$^*$-compactness of $Q$ guarantee that $M$ is a nonempty weak$^*$-compact set.
	Then
	\[
	\lim_{t \downarrow 0} \frac{\lambda(t) - \lambda(0)}{t} = \min_{y^* \in M} g(y^*).
	\]
\end{lemma}

\begin{proof}
 
	Since $Q$ is weak$^*$-compact and $c, g$ are weak$^*$-continuous, for each $t \ge 0$ the function $y^* \mapsto c(y^*) + t g(y^*)$ is weak$^*$-continuous, so the minimum is attained and $\lambda(t)$ is well-defined and finite.
	Let $m = \min_{y^* \in M} g(y^*)$. As $M$ is weak$^*$-compact and $g$ is weak$^*$-continuous, $m$ is finite and there exists $\bar y^* \in M$ with $g(\bar y^*) = m$.

	For any $t > 0$,  it follows from $\bar y^* \in M$ that
$	\lambda(t) \le c(\bar y^*) + t g(\bar y^*) = \lambda(0) + t m,$ 
and so
	\[
	\frac{\lambda(t) - \lambda (0)}{t} \le m \quad \forall t>0.
	\]
	Taking the limit superior as $t \downarrow 0$ gives
	\begin{equation}  \label{EtyuK1}
		\limsup_{t \downarrow 0} \frac{\lambda(t) - \lambda(0)}{t} \le m.  
	\end{equation}

In view of     (\ref{EtyuK1}),   it suffices to prove
\begin{equation}  \label{EtyuK++}
	\liminf_{t \downarrow 0} \frac{\lambda(t) - \lambda(0)}{t} \ge m.
	\end{equation}
	For any $\varepsilon > 0$, we must find $\delta > 0$ such that for all $0 < t < \delta$,
	\[
	\frac{\lambda(t) - \lambda(0)}{t} > m - \varepsilon.
	\]
	Suppose not, then there exists a sequence $t_k \downarrow 0$ such that
	\begin{equation}  \label{EtyuK2}
		\frac{\lambda(t_k) - \lambda(0)}{t_k} \le m - \varepsilon, \quad k=1,2,\dots.  
	\end{equation}
	For each $k$,    there exists $y_k^* \in Q$ such that
	\begin{equation}  \label{EtyuK3}
	\lambda (t_k) = c(y_k^*) + t_k g(y_k^*).  
	\end{equation}
	  We conclude from    $c(y_k^*) \ge \lambda(0)$,    (\ref{EtyuK2})  and  (\ref{EtyuK3})  that
 		\begin{equation}  \label{EtyuK4}
 				t_k g(y_k^*) \le c(y_k^*) + t_k g(y_k^*) - \lambda(0) \le t_k (m - \varepsilon).	
 				\end{equation}
	Cancelling $t_k > 0$ gives
	\begin{equation}  \label{EtyuK5}
		g(y_k^*) \le m - \varepsilon, \quad \forall k   \in \mathbb{N}. 
		\end{equation}
	On the other hand, from (\ref{EtyuK4})  and boundedness of $g$ on $Q$ (say $|g| \le \eta$),
	\[
	0 \le c(y_k^*) - \lambda(0) \le t_k (m - \varepsilon - g(y_k^*)) \le t_k (|m| + \eta).
	\]
	Letting $k \to \infty$, we get
	\begin{equation}  \label{EtyuK6}
		c(y_k^*) \to \lambda(0).  
	\end{equation}
	Since $Q$ is weak$^*$-compact, $\{y_k^*\}$ has a convergent subnet; without loss of generality, $y_k^* \to \hat y^* \in Q$. By weak$^*$-continuity of $c$ and  (\ref{EtyuK6}),
$	c(\hat y^*) = \lambda(0),$
and	so $\hat y^* \in M$. Using weak$^*$-continuity of $g$, taking the limit in (\ref{EtyuK5}) yields
$	g(\hat y^*) \le m - \varepsilon.$
	But $\hat y^* \in M$ implies $g(\hat y^*) \ge m$, a contradiction. Hence  (\ref{EtyuK++}) holds.
\end{proof}

\begin{theorem}  \label{YXDS}
	Let   $A\subseteq X$ be nonempty compact convex  and  $\bar x\in E_w$.  Then for any direction $d\in X$, 
	\[
	\theta'(\bar x;d)=\sup_{a\in A(\bar x)}\min_{y^*\in S^*(a)}\langle y^*,F(d)\rangle,
	\]
	where
	\[
	A(\bar x)=\bigl\{a\in A:\inf_{y^*\in S(C^+)}\langle y^*,F(\bar x)-F(a)\rangle=0\bigr\},
	\]
	\[
	S^*(a)=\bigl\{y^*\in S(C^+):\langle y^*,F(\bar x)-F(a)\rangle=0\bigr\}.
	\]
\end{theorem}

\begin{proof}
	For each $z\in Y$, define
	\[
	\psi(z):=\min_{y^*\in S(C^+)}\langle y^*,z\rangle.
	\]
	The map $y^*\mapsto\langle y^*,z\rangle$ is weak$^*$-continuous, and $S(C^+)$ is weak$^*$-compact, so the minimum is attained.  Since $\psi$ is the pointwise minimum of a family of linear functionals, $\psi$  is  concave and $1$-Lipschitz continuous: for any $z_1,z_2\in Y$,
	\[
	|\psi(z_1)-\psi(z_2)|
	\le\max_{y^*\in S(C^+)}|\langle y^*,z_1-z_2\rangle|
	\le\|z_1-z_2\|.
	\]

	For each $a\in A$, define
	\[
	h_a(x):=\psi(F(x)-F(a))=\min_{y^*\in S(C^+)}\langle y^*,F(x)-F(a)\rangle.
	\]
	From the concavity and Lipschitz properties of $\psi$, $h_a\left( \cdot \right)$ is also concave and $\|F\|$-Lipschitz continuous. By definition,
	\[
	\theta(x)  = \sup_{a \in A} \inf_{y^* \in S(C^+)} \langle y^*, F(x) - F(a) \rangle =\sup_{a\in A}h_a(x).
	\]
	Since $A$ is compact and for fixed $x$ the map $a\mapsto h_a(x)$ is continuous ($\psi$ and $F$ are continuous), the supremum is attained:
\begin{equation}\label{Ebvd1}
	\theta(x)=\max_{a\in A}h_a(x),\qquad\forall x\in X.  
 \end{equation}
	Due to $\bar x\in E_w$, we obtain  $\theta(\bar x)=0$.  It follows from  (\ref{Ebvd1})  that
\begin{equation}\label{Ebvd+1}
h_a(\bar x)\le  0,  \quad  \forall  a\in A . 
  \end{equation}

 It is clear that $h_{\bar x}(\bar x)=\psi(0)=0$. Define
	\[
	A(\bar x):=\{a\in A:h_a(\bar x)=0\}.
	\]
	By continuity of $a\mapsto h_a(\bar x)$ and compactness of $A$, $A(\bar x)$ is a nonempty compact set ($\bar x\in A(\bar x)$).
	
 	Fix $a\in A$ and a direction $d\in X$, let
	\[
\psi_a(t):=h_a(\bar x+td)=\min_{y^*\in S(C^+)}\bigl[\langle y^*,F(\bar x)-F(a)\rangle+t\langle y^*,F(d)\rangle\bigr],\quad t\ge0.
	\]
	Set
	\[
	c_a(y^*):=\langle y^*,F(\bar x)-F(a)\rangle,\qquad g(y^*):=\langle y^*,F(d)\rangle.
	\]
	Then $\psi _a(t)=\min_{y^*\in S(C^+)}[c_a(y^*)+t\,g(y^*)]$.    By  Lemma   \ref{LXcEQ}, we have 
\begin{equation}\label{Ebvd2}
	h'_a(\bar x;d):=\lim_{t\downarrow0}\frac{h_a(\bar x+td)-h_a(\bar x)}{t}   =\lim_{t\downarrow0}\frac{\psi _a(t)-\psi_a(0)}{t} =\min_{y^*\in M_a}g(y^*),
 \end{equation}
	where
	\[
	M_a:=\{y^*\in S(C^+):c_a(y^*)=\psi_a(0)=h_a(\bar x)\}.
	\]
	
If $a\in A(\bar x)$, then $h_a(\bar x)=0$, and by definition $c_a(y^*)\ge0$ for all $y^*\in S(C^+)$. Hence
\[
M_a=\{y^*\in S(C^+):c_a(y^*)=0\}=S^*(a).
\]
This together with  (\ref{Ebvd2})  implies that 
\begin{equation}\label{Ebvd3}
h'_a(\bar x;d)=\min_{y^*\in S^*(a)}\langle y^*,F(d)\rangle. 
 \end{equation}
	
Since $\theta(x)\ge h_a(x)$ for all $x$, and $\theta(\bar x)=h_a(\bar x)=0$ for $a\in A(\bar x)$, for any $t>0$,
	\[
	\frac{\theta(\bar x+td)-\theta(\bar x)}{t}
	\ge\frac{h_a(\bar x+td)-h_a(\bar x)}{t}.
	\]
	Letting $t\downarrow0$, the right-hand side tends to $h'_a(\bar x;d)$, so
	\[
	\liminf_{t\downarrow0}\frac{\theta(\bar x+td)}{t}\ge h'_a(\bar x;d),\quad\forall a\in A(\bar x).
	\]
	Taking the supremum over $a\in A(\bar x)$ gives
\begin{equation}\label{Ebvd4}
	\liminf_{t\downarrow0}\frac{\theta(\bar x+td)}{t}\ge\sup_{a\in A(\bar x)}h'_a(\bar x;d) .  
\end{equation}

		Take any sequence $t_k\downarrow0$. By (\ref{Ebvd1}), for each $k$ there exists $a_k\in A$ such that
\begin{equation}\label{Ebvd5}	
	\theta(\bar x+t_k d)=h_{a_k}(\bar x+t_k d).
\end{equation} 
	Since $A$ is compact, we can extract a subsequence (still denoted $a_k$) converging to some $a^*\in A$. By joint continuity of $h_a(x)$ in $(a,x)$,
	\[
	\theta(\bar x+t_k d)=h_{a_k}(\bar x+t_k d)\to h_{a^*}(\bar x).
	\]
	But by continuity of $\theta$, $\theta(\bar x+t_k d)\to\theta(\bar x)=0$, so $h_{a^*}(\bar x)=0$, i.e., $a^*\in A(\bar x)$.
	
	Since  $h_a\left( \cdot \right)$  is a  concave function, we get
\begin{equation}\label{Ebvd6}	
	 h_a(\bar x+td)\le h_a(\bar x)+t\,h'_a(\bar x;d), \quad  \forall  t>0. 
\end{equation} 
	 Thanks to (\ref{Ebvd+1}), (\ref{Ebvd5}) and (\ref{Ebvd6}), we have
	\[
	\theta(\bar x+t_k d)=h_{a_k}(\bar x+t_k d)
	\le h_{a_k}(\bar x)+t_k\,h'_{a_k}(\bar x;d)
	\le t_k\,h'_{a_k}(\bar x;d),
	\]
and so
\begin{equation}\label{Ebvd7}	
	\frac{\theta(\bar x+t_k d)}{t_k}\le h'_{a_k}(\bar x;d).
\end{equation} 
For any $a\in A(\bar x)$,  it follows from  (\ref{Ebvd3})  that
	\[
\zeta (a):=h'_a(\bar x;d)=\min_{y^*\in S^*(a)}\langle y^*,F(d)\rangle. 
\]
It is easy to see that  $a\mapsto \zeta(a)$ is upper semicontinuous. By upper semicontinuity of $\zeta$ and $a_k\to a^*$,
	\[
	\limsup_{k\to\infty} \zeta(a_k)\le \zeta(a^*)=h'_{a^*}(\bar x;d).
	\]
Combining this with (\ref{Ebvd7}), we get
	\[
	\limsup_{k\to\infty}\frac{\theta(\bar x+t_k d)}{t_k}  \le   	\limsup_{k\to\infty}   h'_{a_k}(\bar x;d)
	\le h'_{a^*}(\bar x;d)
	\le\sup_{a\in A(\bar x)}h'_a(\bar x;d).
	\]
	Since the sequence $\{t_k\}$ was arbitrary,
\begin{equation}\label{Ebvd8}	
	\limsup_{t\downarrow0}\frac{\theta(\bar x+td)}{t}
	\le\sup_{a\in A(\bar x)}h'_a(\bar x;d).  
\end{equation} 
		 We conclude from (\ref{Ebvd3}),  (\ref{Ebvd4})  and  (\ref{Ebvd8})  that
	 	\[
	\theta'(\bar x;d)=\lim_{t\downarrow0}\frac{\theta(\bar x+td)}{t}
	=\sup_{a\in A(\bar x)}h'_a(\bar x;d) =\sup_{a\in A(\bar x)}\min_{y^*\in S^*(a)}\langle y^*,F(d)\rangle.
	\]
\end{proof}

\begin{theorem} 
	Assume that $A \subseteq X$ is a nonempty compact convex set and $\bar x \in E_w$.
 	Then the following statements hold:
	\begin{itemize}
		\item[(i)]  If a direction $d \in T_A(\bar x)$ satisfies $F(d) \in -C$, then $\theta'(\bar x; d) = 0$.
			\item[(ii)]  If a direction $d \in X$ satisfies $\theta'(\bar x; d) > 0$, then there exists $a \in A(\bar x)$ such that   $F(a) - F(\bar x + t d) \notin \operatorname{int} C$  for any $t > 0$.
	\end{itemize}
\end{theorem}

\begin{proof}  For any direction $d\in X$,   it follows from  Theorem   \ref{YXDS}  that
\begin{equation}\label{EkptB1}
	\theta'(\bar x;d)=\sup_{a\in A(\bar x)}\min_{y^*\in S^*(a)}\langle y^*,F(d)\rangle,
 \end{equation}
	where
	\[
	A(\bar x)=\bigl\{a\in A:\inf_{y^*\in S(C^+)}\langle y^*,F(\bar x)-F(a)\rangle=0\bigr\},
	\]
	\[
	S^*(a)=\bigl\{y^*\in S(C^+):\langle y^*,F(\bar x)-F(a)\rangle=0\bigr\}.
	\]
	 (i).	Let $d \in T_A(\bar x)$ and suppose $F(d) \in -C$. For any $y^* \in C^+$,  we have $\langle y^*, -F(d) \rangle \ge 0$, hence $\langle y^*, F(d) \rangle \le 0$.
	Therefore, for each fixed $a \in A(\bar x)$, due to $S^*(a) \subseteq C^+$,
	\[
	\min_{y^* \in S^*(a)} \langle y^*, F(d) \rangle \le 0.
	\]
	Taking the supremum over $a \in A(\bar x)$ and using  (\ref{EkptB1}) gives
\begin{equation}\label{EkptB2}
	\theta'(\bar x; d) = \sup_{a \in A(\bar x)} \min_{y^* \in S^*(a)} \langle y^*, F(d) \rangle \le 0.
 \end{equation}
	On the other hand, by Theorem~\ref{ZSQE} and $\bar x \in E_w$, we have $\theta'(\bar x; d) \ge 0$. Combining with    (\ref{EkptB2})    yields $\theta'(\bar x; d) = 0$.

(ii).	Fix $d \in X$ with $\theta'(\bar x; d) > 0$. Consider the function
	\[
	\psi(a) := \min_{y^* \in S^*(a)} \langle y^*, F(d) \rangle, \quad a \in A(\bar x).
	\]
 It follows from  (\ref{EkptB1})  that
	\[
	\theta'(\bar x; d) = \sup_{a \in A(\bar x)} \min_{y^* \in S^*(a)} \langle y^*, F(d) \rangle > 0,
	\]
and so	there exists $a_0 \in A(\bar x)$ such that
\begin{equation}\label{EkptB3}
	\psi(a_0) = \min_{y^* \in S^*(a_0)} \langle y^*, F(d) \rangle > 0.
 \end{equation}
	
	Suppose that there exists $t_0 > 0$ such that
	\[
	F(a_0) - F(\bar x + t_0 d) \in \operatorname{int} C.
	\]
	Set $z_0 := F(a_0) - F(\bar x) - t_0 F(d) \in \operatorname{int} C$.
	Using $a_0 \in A(\bar x)$ and picking   $y_0^* \in S^*(a_0)$, we have
$	\langle y_0^*, F(\bar x) - F(a_0) \rangle = 0  $.
	Then
	\[
	\langle y_0^*, z_0 \rangle = \langle y_0^*, F(a_0) - F(\bar x) \rangle - t_0 \langle y_0^*, F(d) \rangle = -t_0 \langle y_0^*, F(d) \rangle.
	\]
	Because $z_0 \in \operatorname{int} C$ and $y_0^* \in C^+ \setminus \{0\}$, we   have $\langle y_0^*, z_0 \rangle > 0$. Hence
	\[
	-t_0 \langle y_0^*, F(d) \rangle > 0 \quad \Longrightarrow \quad \langle y_0^*, F(d) \rangle < 0.
	\]
	Consequently,
	\[
	\psi(a_0) = \min_{y^* \in S^*(a_0)} \langle y^*, F(d) \rangle \le \langle y_0^*, F(d) \rangle < 0,
	\]
	contradicting (\ref{EkptB3}). Therefore, for every $t > 0$, $F(a_0) - F(\bar x + t d) \notin \operatorname{int} C$.
\end{proof}

\begin{theorem} \label{ZAQW}  Let $\bar{x}\in E_w$. 	Then the following hold:
	\begin{enumerate}
		\item[(i)]  	$
		T_{E_w}(\bar{x}) \subseteq T_A(\bar{x})\cap T_{\widehat{A}}(\bar{x}) \subseteq \{d\in T_A(\bar{x}):\theta'(\bar{x};d)=0\}.
	$
		
				\item[(ii)]  If $\theta$ has a local error bound at   $\bar x$, then 
					\[
				T_{E_w}(\bar{x})=T_A(\bar{x})\cap T_{\widehat{A}}(\bar{x})=\{d\in T_A(\bar{x}):\theta'(\bar{x};d)=0\},
				\]
				
			\end{enumerate}	 
\end{theorem}

 \begin{proof} 
Denote
\[
S_1=T_{E_w}(\bar{x}),\quad S_2=T_A(\bar{x})\cap T_{\widehat{A}}(\bar{x}),\quad S_3=\{d\in T_A(\bar{x}):\theta'(\bar{x};d)=0\}.
\]

(i).  Take any $d\in S_1=T_{E_w}(\bar{x})$. Then there exist sequences $\{t_k\}\subset\mathbb{R}_+$ with $t_k\to0^+$, and $\{d_k\}\subset X$ with $d_k\to d$, such that      $\bar{x}+t_k d_k\in E_w$       for all $k\in \mathbb{N}$.
 Since $E_w=A\cap\widehat{A}\subseteq A$, we have $\bar{x}+t_k d_k\in A$, so by definition of the tangent cone, $d\in T_A(\bar{x})$.
Similarly, $E_w\subseteq\widehat{A}$ gives $\bar{x}+t_k d_k\in\widehat{A}$, and so $d\in T_{\widehat{A}}(\bar{x})$.
Thus $d\in T_A(\bar{x})\cap T_{\widehat{A}}(\bar{x})=S_2$.  This yields     that $S_1\subseteq S_2$.

Take any $d\in S_2$, i.e., $d\in T_A(\bar{x})\cap T_{\widehat{A}}(\bar{x})$. 
By $d\in T_A(\bar{x})$, $\bar{x}\in E_w$ and Theorem \ref{ZSQE} (necessity does not require $A$ to be compact), we have
\begin{equation}\label{Euiv1}
\theta'(\bar{x};d)\ge 0.  
\end{equation}

Since $d\in T_{\widehat{A}}(\bar{x})$, there exist sequences $t_k\downarrow0$ and $d_k\to d$ such that
\[
y_k:=\bar{x}+t_k d_k\in\widehat{A}=X\setminus(A+\operatorname{int}C_X), \quad  \forall k \in \mathbb{N}.
\]
    We conclude from  Theorem  \ref{TxEQ} (i)  that
\begin{equation}\label{Euiv2}
\theta(y_k)\le 0,  \quad  \forall k \in \mathbb{N}.
\end{equation}
Using the   Lipschitz continuity of $\theta$ (constant $L=\|F\|$, Theorem \ref{YL:thetaprop2}), for each $k \in \mathbb{N}$,
\[
|\theta(\bar{x}+t_k d_k)-\theta(\bar{x}+t_k d)|\le L t_k\|d_k-d\|,
\]
and so
\[
\theta(\bar{x}+t_k d_k)=\theta(\bar{x}+t_k d)+r_k,\quad |r_k|\le L t_k\|d_k-d\|.
\]
Thus
\[
\frac{\theta(\bar{x}+t_k d_k)}{t_k}=\frac{\theta(\bar{x}+t_k d)}{t_k}+\frac{r_k}{t_k},
\]
with $|r_k|/t_k\le L\|d_k-d\|\to0$.   Due to  $\theta(\bar{x})=0$,   the definition of directional derivative gives $\lim_{k\to\infty}\frac{\theta(\bar{x}+t_k d)}{t_k}=\theta'(\bar{x};d)$, therefore
\[
\lim_{k\to\infty}\frac{\theta(y_k)}{t_k}=\lim_{k\to\infty}\frac{\theta(\bar{x}+t_k d_k)}{t_k}=\theta'(\bar{x};d).
\]
From (\ref{Euiv2}), each term $\theta(y_k)/t_k\le0$, so the limit $\theta'(\bar{x};d)\le0$. Together with (\ref{Euiv1}) we get $\theta'(\bar{x};d)=0$, i.e., $d\in S_3$.  This means that $S_2\subseteq S_3$.

(ii).         It suffices to prove  that   $S_3\subseteq S_1$.      Since $\theta$ has a local error bound at   $\bar x$,   there exist $\tau$ and $\delta>0$ such that
\begin{equation}\label{Euiv3}
d(x,E_w)\le \tau\,\theta(x),\quad\forall x\in A\cap B(\bar{x},\delta).
\end{equation}

Take any $d\in S_3$. Then $d\in T_A(\bar{x})$ and $\theta'(\bar{x};d)=0$.
By $d\in T_A(\bar{x})$, there exist $t_k\downarrow0$, $d_k\to d$ such that   $x_k:=\bar{x}+t_k d_k\in A$ for any $k \in \mathbb{N}$.
Using Lipschitz continuity of $\theta$, we obtain
\[
\theta(x_k)=\theta(\bar{x}+t_k d_k)\le\theta(\bar{x}+t_k d)+L t_k\|d_k-d\|.
\]
Since $\theta'(\bar{x};d)=0$ and $\theta(\bar{x})=0$, we have $\theta(\bar{x}+t_k d)=t_k\varepsilon_k$ with $\varepsilon_k\to0$. Hence
\[
0\le\theta(x_k)\le t_k\varepsilon_k+L t_k\|d_k-d\|=t_k\delta_k,
\]
where $\delta_k:=\varepsilon_k+L\|d_k-d\|\to0$.  
For sufficiently large $k$, $x_k\in A\cap B(\bar{x},\delta)$.  It follows from  (\ref{Euiv3})  that
$ d(x_k,E_w)\le\tau\theta(x_k)\le\tau t_k\delta_k.$
Then we can choose $z_k\in E_w$ such that
\[
\|x_k-z_k\|<d(x_k,E_w)+t_k^2\le\tau t_k\delta_k+t_k^2.
\]
Consider the vector
\[
\frac{z_k-\bar{x}}{t_k}= \frac{z_k-x_k}{t_k}+\frac{x_k-\bar{x}}{t_k}= \frac{z_k-x_k}{t_k}+d_k.
\]
Since $\|z_k-x_k\|/t_k<\tau\delta_k+t_k\to0$ and $d_k\to d$, we obtain
\[
\lim_{k\to\infty}\frac{z_k-\bar{x}}{t_k}=d.
\]
Because $z_k\in E_w$, $t_k\downarrow0$, and $(\bar{z}_k-\bar{x})/t_k\to d$, the definition of the tangent cone yields $d\in T_{E_w}(\bar{x})=S_1$. By arbitrariness, $S_3\subseteq S_1$.
 \end{proof}   

\begin{corollary} Let   $A \subseteq X$ a  nonempty compact   polyhedron. 
	 Then	for any  $\bar{x} \in E_w$, we have
	\[
	T_{E_w}(\bar{x})=T_A(\bar{x})\cap T_{\widehat{A}}(\bar{x})=\{d\in T_A(\bar{x}):\theta'(\bar{x};d)=0\}.
	\]
\end{corollary}

\begin{proof}
	By Corollary 5  of  \cite{LiuNgYang2009}, $d_{\widehat{A}}$ has an error bound on $A$, which is equivalent to $\theta$ having an error bound on $A$. Clearly, $\theta$ has a local error bound at $\bar{x}$. Hence, by Theorem \ref{ZAQW}, we obtain
	\[
	T_{E_w}(\bar{x})=T_A(\bar{x})\cap T_{\widehat{A}}(\bar{x})=\{d\in T_A(\bar{x}):\theta'(\bar{x};d)=0\}.
	\]
\end{proof}

The following lemma is easily established.
\begin{lemma} \label{CYU}
	Let $X$ be a Hilbert space, $S\subset X$ be a nonempty closed convex set, and $x\in S$. Then for any direction $d\in X$, the directional derivative
	\[
	d_S'(x;d) := \lim_{t\downarrow0}\frac{d_S(x+td)}{t}
	\]
	exists, and
	$	d_S'(x;d) = \operatorname{dist}(d, T_S(x))$.
\end{lemma}

\begin{theorem} \label{TaqXEQ} 
	Let  $X$ be a Hilbert space  and $\bar x\in E_w$.  Assume that
	\begin{itemize}
		\item[(i)]  $\theta$ has a  local error bound at $\bar x$, i.e.,  there exist constants $\tau>0,\delta>0$ such that
	\begin{equation}\label{ECIA++}
		d(x,E_w)\le \tau\,\theta(x)\qquad \forall x\in A\cap B(\bar x,\delta).
		\end{equation}
		\item[(ii)]  $E_w$ is    locally convex  at $\bar x$, i.e., there exists $\eta >0$ such that $E_w\cap B(\bar x,\eta)$ is convex. 
	\end{itemize}
		Then for any   direction $d\in T_A(\bar x)$,   we have
	\[
	\theta'(\bar x;d)\ge \frac{1}{\tau}\,\operatorname{dist}\bigl(d,\;T_{E_w}(\bar x)\bigr).
	\]
\end{theorem}

\begin{proof}
	Since $\theta$ is concave,  for any $x \in X$,  it is easy to get that 
	\[
	\theta(x+td) \le \theta(x) + t\,\theta'(x;d),\quad \forall t>0.
	\]
	In particular, at $\bar x\in E_w$, $\theta(\bar x)=0$, so
	\begin{equation}\label{ECIA1}
		\theta(\bar x+td) \le t\,\theta'(\bar x;d),\quad \forall t>0.  
	\end{equation}

	Set
$	\rho := \min\{\delta,\, \eta \}/2 >0,$  	and define
$	S := E_w \cap \overline{B}(\bar x,\rho), $
	where $\overline{B}(\bar x,\rho)$ is the closed ball.  By the local convexity hypothesis and the closedness of $E_w$,   we obtain that $S$ is  closed and convex. Clearly, $S\subset E_w$ and $\bar x\in S$.

	Set $\rho' := \rho/2$. For any  $x\in B(\bar x,\rho')$, on one hand, $S\subset E_w$ implies $d(x,E_w)\le d(x,S)$.
	
		On the other hand, note $d(x,E_w)\le \|x-\bar x\| < \rho' = \rho/2$. Pick $\varepsilon$ small enough so that $0<\varepsilon < \rho/2 - d(x,E_w)$. By definition of distance, for this $\varepsilon>0$ there exists $y_\varepsilon\in E_w$ with
	$$	\|x-y_\varepsilon\| < d(x,E_w) + \varepsilon.$$
	Then
	\[
	\|y_\varepsilon-\bar x\| \le \|y_\varepsilon-x\| + \|x-\bar x\| < d(x,E_w)+\varepsilon+\rho' < \frac{\rho}{2} + \frac{\rho}{2} = \rho.
	\]
	Hence $y_\varepsilon\in B(\bar x,\rho)\cap E_w = S$. Therefore
	\[
	d(x,S) \le \|x-y_\varepsilon\| < d(x,E_w) + \varepsilon.
	\]
	Letting $\varepsilon\downarrow0$ gives $d(x,S)\le d(x,E_w)$. Together with $d(x,E_w)\le d(x,S)$,  we get
\begin{equation}\label{ECIA2}
		d(x,E_w)=d(x,S),\quad \forall x\in B(\bar x,\rho').  
	\end{equation}

	Take any $d\in T_A(\bar x)$. Then  there exist sequences $\{t_k\}\subset (0,\infty)$ with $t_k\downarrow0$, and $\{d_k\}\subset X$ with $d_k\to d$, such that
	\[
	x_k := \bar x + t_k d_k \in A,\qquad \forall k\in\mathbb N.
	\]
	Since $d_k\to d$, $\{d_k\}$ is bounded, and $t_k\to0$, there exists $k_0 \in\mathbb N$ such that for all $k\ge k_0$, $\|x_k-\bar x\| = t_k\|d_k\| < \rho'$. Thus for large $k$, $x_k\in A\cap B(\bar x,\rho')$.
	At such $k$,  it follows from  (\ref{ECIA++})   and   (\ref{ECIA2})    that
	\begin{equation}\label{ECIA3}
		d(x_k,E_w) = d(x_k,S) \le \tau\,\theta(x_k).  
	\end{equation}
		Using the Lipschitz property of $\theta$ (Theorem \ref{YL:thetaprop2}),
 	$$\theta(x_k) = \theta(\bar x + t_k d_k) \le \theta(\bar x + t_k d) + L\|t_k d_k - t_k d\| = \theta(\bar x + t_k d) + L\,t_k\|d_k-d\|. $$
	This together with  (\ref{ECIA1})   implies that   
 	\begin{equation}\label{ECIA4}
		\theta(x_k) \le t_k\theta'(\bar x;d) + L\,t_k\|d_k-d\|.  
	\end{equation}

	By the triangle inequality for distance,
	\begin{align*}
		d(\bar x+t_k d, E_w) &\le d(x_k, E_w) + \|(\bar x+t_k d)-x_k\| \notag\\
		&= d(x_k, E_w) + t_k\|d_k-d\|.  
	\end{align*}
 Combining this with   (\ref{ECIA3})   and   (\ref{ECIA4}), we get
	\begin{align*}
		d(\bar x+t_k d, E_w)
		&\le \tau\theta(x_k) + t_k\|d_k-d\| \notag\\
		&\le \tau\bigl( t_k\theta'(\bar x;d) + L t_k\|d_k-d\|\bigr) + t_k\|d_k-d\| \notag\\
		&= t_k \tau \theta'(\bar x;d) + t_k(\tau L+1)\|d_k-d\|.
	\end{align*}
	Dividing both sides by $t_k>0$, we obtain the key inequality:
	\begin{equation}\label{ECIA5}
		\frac{d(\bar x+t_k d, E_w)}{t_k} \le \tau\,\theta'(\bar x;d) + (\tau L+1)\|d_k-d\|.  
	\end{equation}

	For sufficiently large $k$, $\bar x+t_k d\in B(\bar x,\rho')$, so by  (\ref{ECIA2}), $d(\bar x+t_k d, E_w) = d(\bar x+t_k d, S)$. Applying Lemma~\ref{CYU} to the closed convex set $S$, we get
	\begin{equation}\label{ECIA6}
	\lim_{k\to\infty}\frac{d(\bar x+t_k d, E_w)}{t_k} = \lim_{k\to\infty}\frac{d(\bar x+t_k d, S)}{t_k} = \operatorname{dist}\bigl(d, T_S(\bar x)\bigr).
	\end{equation}
	Moreover, since $S$ and $E_w$ coincide in a neighborhood of $\bar x$, their tangent cones are the same: $T_S(\bar x) = T_{E_w}(\bar x)$.  This together with  (\ref{ECIA6})  yields that 
	\begin{equation}\label{ECIA7}
		\lim_{k\to\infty}\frac{d(\bar x+t_k d, E_w)}{t_k} =   \operatorname{dist}\bigl(d, T_{E_w}(\bar x)\bigr).
	\end{equation}
	    We conclude from  (\ref{ECIA5})  and  (\ref{ECIA7})  that
$	\operatorname{dist}\bigl(d, T_{E_w}(\bar x)\bigr) \le \tau\,\theta'(\bar x;d),$
and so
	\[
	\theta'(\bar x;d) \ge \frac{1}{\tau}\,\operatorname{dist}\bigl(d, T_{E_w}(\bar x)\bigr).
	\]
	This completes the proof.
\end{proof}

\section{Characterizing error bounds via directional derivatives}
 
 The preceding section has laid the foundation for the directional differentiability of the merit function \(\theta\). We obtained explicit formulas for \(\theta'(\bar{x};d)\) at weakly efficient points, characterized the cone of zero directional derivatives, and established a precise link between the tangent cone of the solution set \(T_{E_w}(\bar{x})\) and the directional behaviour of \(\theta\) when a local error bound is present. These results reveal that directional derivatives carry essential information about both the local geometry of the solution set and the rate at which \(\theta\) increases away from it. In this section we   derive characterizations of error bounds in terms of directional derivative conditions. Our main goal is to show that the global error bound property is   captured by a uniform negativity condition on the infimum of directional derivatives over the unit sphere of the tangent cone, and that the optimal local error bound constant is precisely the reciprocal of the corresponding minimal directional derivative. To this end, we first introduce two key quantities that will play a central role throughout the analysis.

 For any \(x \in A\), the unit-ball minimum directional derivative \(\mu(x)\) and the unit-sphere minimum directional derivative \(\varphi(x)\) are defined, respectively, by
\[
\mu(x) := \inf\bigl\{ \theta'(x; d) : d \in T_A(x),\; \|d\| \le 1 \bigr\},
\]
and
\[
\varphi(x) := \inf\bigl\{ \theta'(x; d) : d \in T_A(x),\; \|d\| = 1 \bigr\}.
\]

We now discuss the relationship between the unit-ball minimum directional derivative \(\mu(x)\) and the unit-sphere minimum directional derivative \(\varphi(x)\).
 
\begin{theorem}\label{thm:mu-c}
	Suppose \(T_A(x)\neq\{0\}\). Then
	\[
	\mu(x)=\min\{\varphi(x),\,0\}.
	\]
 	In particular, \(\mu(x)\le 0\) always holds, and \(\mu(x)<0\) if and only if there exists a feasible descent direction (i.e., there exists \(d_0 \in T_A(x)\)  such that  \(\|d_0 \|=1\) and \(\theta'(x;d_0 )<0\)).
\end{theorem}

\begin{proof}[Proof.]
	Since \(0\in T_A(x)\) and \(\|0\|\le1\), we have \(\mu(x)\le\theta'(x;0)=0\), hence \(\mu(x)\le0\) always holds.
	
	Now we distinguish two cases according to the sign of $\varphi(x)$.
	
	\noindent\textbf{Case 1: \(\varphi(x)\ge 0\).} Then for all \(d\in T_A(x)\) with \(\|d\|=1\), we have \(\theta'(x;d)\ge 0\). Take any \(v\in T_A(x)\) with \(\|v\|\le1\). If \(v=0\), then \(\theta'(x;v)=0\). If \(v\neq0\), set \(d=v/\|v\|\); then \(\|d\|=1\). By positive homogeneity,
	\[
	\theta'(x;v)=\|v\|\,\theta'(x;d)\ge 0,
	\]
	so \(\theta'(x;v)\ge0\) for all \(\|v\|\le1\), which yields \(\mu(x)\ge0\). Together with \(\mu(x)\le0\) we obtain \(\mu(x)=0\).
	
	\noindent\textbf{Case 2: \(\varphi(x)<0\).} For any \(v\in T_A(x)\) with \(\|v\|\le1\), if \(\|v\|=1\) then \(\theta'(x;v)\ge \varphi(x)\).   If \(\|v\|<1\), take \(d=v/\|v\|\) (when \(v\neq0\)). Then it follows from   \(\varphi(x)<0\) and \(\|v\|   < 1\)    that  
	\[
	\theta'(x;v)=\|v\|\,\theta'(x;d)\ge \|v\|\,\varphi(x) > \varphi(x).
	\]
  Hence \(\theta'(x;v)\ge \varphi(x)\) for all \(v\in T_A(x)\) with \(\|v\|\le1\). 
  This means that \(\mu(x)  \geq   \varphi(x)\).   It is clear that \(\mu(x)  \leq   \varphi(x)\). Therefore, we get  \(\mu(x)=\varphi(x)\).  It is easy to see that   \(\mu(x) =  \varphi(x)<0\) if and only if 
   there exists \(d_0\in T_A(x)\) with \(\|d_0\|=1\) such that \(\theta'(x;d_0)<0\). 
\end{proof}

\begin{remark} \label{Rxdt}  Let  $A\subset X$ be a nonempty compact convex set.   It follows from Theorem  \ref{ZSQE} that
	\begin{equation}\label{Eyup1}
x\in E_w\iff \theta'(x;d)\ge0,\ \forall d\in T_A(x). 
\end{equation}
If $x\in E_w$,   it follows from  (\ref{Eyup1})  that    \(\varphi(x)\ge 0\).  This together with  Theorem  \ref{thm:mu-c}  implies that  \(\mu(x)=0\).  Conversely, if \(x \notin E_w\),  by  (\ref{Eyup1}),   there exists  \(d_0\in T_A(x)\)   such that \(\theta'(x;d_0)<0\).  Then it is easy to see that    \(\varphi(x) < 0\),  and so  by Theorem  \ref{thm:mu-c}, we get  \(\mu(x) = \varphi(x) < 0\). Therefore,
\[
x \in E_w \iff \mu(x) = 0 \iff \varphi(x) \ge 0   \quad   {\rm{and }}    \quad   x \notin E_w \iff \mu(x) = \varphi(x) < 0.
\]
 \end{remark}

Next,   we  present some characterizations of the unit-ball minimum directional derivative \(\mu(x)\).

\begin{theorem} 
	Let $X$ be a finite-dimensional space and  $A\subset X$ be a nonempty compact convex set.
	Then the following hold:
	\begin{enumerate}
		\item[(i)]  For any \(x \in A\), 	the unit-ball minimum directional derivative \(\mu(x)\) can be represented as:
		\[
		\mu(x)=-\max_{v\in K(x)}\operatorname{dist}\bigl(F^*v,\,-N_A(x)\bigr)=-\max_{v\in K(x)}\bigl\|P_{T_A(x)}(-F^*v)\bigr\|,
		\]
		  where \(P_{T_A(x)}\) denotes the orthogonal projection onto the tangent cone \(T_A(x)\),  and  $K(x)=\{v\in K:\theta(x)=\langle F^*v,x\rangle+\sigma_A(-F^*v)\}$ is a nonempty weak$^*$ compact set.  
	
		\item[(ii)]
		Let $x\in A\setminus E_w$ and choose $v_0\in K(x)$ such that $\mu(x)= - \bigl\|P_{T_A(x)}(-F^* v_0)\bigr\|$. Then $P_{T_A(x)}(-F^*v_0)\neq 0$. Set
		\[
		d_0:=\frac{P_{T_A(x)}(-F^*v_0)}{\|P_{T_A(x)}(-F^*v_0)\|}.
		\]
		Then $d_0\in T_A(x)$, $\|d_0\|=1$, and
	$		\theta'(x;d_0)=\mu(x).	$
	\end{enumerate}
\end{theorem}

\begin{proof}
	By Theorem~\ref{JBXZDS}, the map $d\mapsto\theta'(x;d)$ is Lipschitz continuous on $X$.
		Define
	\[
	D:=\{d\in T_A(x):\|d\|\le 1\}.
	\]
	Since $T_A(x)$ is closed and convex, and $X$ is finite-dimensional, $D$ is a compact convex set. Hence, there exists $d_0\in D$ such that
	\begin{equation}\label{VRT1}
		\theta'(x;d_0)=\min_{d\in D} \theta'(x;d)=\mu(x).
	\end{equation}
		By Theorem~\ref{AXDSS}, we have
	\begin{equation}\label{VRT2}
		\theta'(x;d)=\min_{v\in K(x)}\langle F^*v,d\rangle , \quad \forall d\in X.  
	\end{equation}
	
	Because $K$ is a weak$^*$ compact convex set in $Y^*$, its weak$^*$ closed subset $K(x)$ is also weak$^*$ compact.
	For each $v\in K(x)$, define
	\[
	p(v) := \inf_{d\in D} \langle F^*v, d\rangle .
	\]
	Since $D$ is compact and the map $d\mapsto \langle F^*v, d\rangle$ is continuous, and $v\mapsto F^*v$ is continuous, we have
	\begin{equation}\label{VRT++}
		p(v) := \min_{d\in D} \langle F^*v, d\rangle ,
	\end{equation}
	and $p$ is defined and continuous on the weak$^*$ compact set $K(x)$. Consequently,
	\begin{equation}\label{VRT3}
		\inf_{v\in K(x)} p(v) = \min_{v\in K(x)} p(v).
	\end{equation}
	
	For any fixed $d\in D$, by~\eqref{VRT2},
	\[
	\theta'(x;d) = \min_{v\in K(x)} \langle F^*v, d\rangle \ge \min_{v\in K(x)} \Bigl( \inf_{d'\in D} \langle F^*v, d'\rangle \Bigr) = \min_{v\in K(x)} p(v).
	\]
	Since this inequality holds for all $d\in D$, taking the infimum over $d$ yields
	\begin{equation}\label{VRT4}
		\mu(x) = \inf_{d\in D} \theta'(x;d) \ge \min_{v\in K(x)} p(v).  
	\end{equation}
	
	Take any $v_0 \in K(x)$. For every $d\in D$, by~\eqref{VRT2}, $\theta'(x;d) = \min_{v\in K(x)} \langle F^*v, d\rangle \le \langle F^*v_0, d\rangle$. Thus
	\[
	\mu(x) = \inf_{d\in D} \theta'(x;d) \le \inf_{d\in D} \langle F^*v_0, d\rangle = p(v_0).
	\]
	Since $v_0\in K(x)$ is arbitrary, taking the infimum over $v_0$ gives
	\begin{equation}\label{VRT5}
		\mu(x) \le \inf_{v\in K(x)} p(v).  
	\end{equation}
	Combining~\eqref{VRT++}, \eqref{VRT3}, \eqref{VRT4} and~\eqref{VRT5}, we obtain
	\begin{equation}\label{VRT6}
		\mu(x) = \min_{v\in K(x)} p(v) = \min_{v\in K(x)} \min_{d\in D} \langle F^*v, d\rangle =\min_{v\in K(x)}\Bigl(\min_{\substack{d\in T_A(x)\\\|d\|\le 1}}\langle F^*v,d\rangle\Bigr).
	\end{equation}
	
	Fix $w = F^*v$. Set $T = T_A(x)$ and define
	\begin{equation}\label{VRT7}
		\alpha(w) := \min_{\|d\|\le 1,\, d\in T} \langle w, d\rangle = -\max_{\|d\|\le 1,\, d\in T} \langle -w, d\rangle.
	\end{equation}
	
	In a Hilbert space, for the closed convex cone $T$, its polar cone is
	\[
	T^\circ = \{ u : \langle u, d\rangle \le 0,\ \forall d\in T \},
	\]
	and by the relation between tangent and normal cones of a convex set, $T^\circ = N_A(x)$. Using the Moreau orthogonal decomposition, any vector $u$ can be uniquely expressed as
	\[
	u = P_T(u) + P_{T^\circ}(u), \qquad \langle P_T(u), P_{T^\circ}(u) \rangle = 0.
	\]
	From this one deduces
	\begin{equation}\label{VRT8}
		\max_{\|d\|\le 1,\, d\in T} \langle u, d\rangle = \|P_T(u)\| = \operatorname{dist}(u, T^\circ).
	\end{equation}
	Substituting $u = -w$ and $T^\circ = N_A(x)$ into~\eqref{VRT8} gives
	\[
	\max_{\|d\|\le 1,\, d\in T} \langle -w, d\rangle = \|P_T(-w)\| = \operatorname{dist}(-w, N_A(x)).
	\]
	Together with~\eqref{VRT7}, this implies
	\begin{equation}\label{VRT9}
		\alpha(w) = -\|P_T(-w)\| = -\operatorname{dist}(-w, N_A(x)) = -\operatorname{dist}(w, -N_A(x)).
	\end{equation}
	Insert $w = F^*v$ and $T = T_A(x)$ into~\eqref{VRT9} and use~\eqref{VRT6} and~\eqref{VRT7} to get
	\begin{eqnarray}\label{VRT10}
		\mu(x) &=& \min_{v\in K(x)} \alpha(F^*v)
		= \min_{v\in K(x)} \bigl[-\operatorname{dist}(F^*v, -N_A(x))\bigr] \nonumber\\ 
		&=& -\max_{v\in K(x)} \operatorname{dist}(F^*v, -N_A(x)) = -\max_{v\in K(x)}\bigl\|P_{T_A(x)}(-F^*v)\bigr\|.
	\end{eqnarray}
 This proves part (i).

	\medskip
	\noindent\textbf{(ii).} Let $x\in A\setminus E_w$.  It follows from  Remark   \ref{Rxdt}    that $\mu(x) < 0$. By part (i), choose $v_0\in K(x)$ such that
	\[
	\mu(x) = -\|P_T(-F^*v_0)\|,
	\]
	and so $\|P_T(-F^*v_0)\| = -\mu(x) > 0$. Define the unit vector
	\[
	d_0 = \frac{P_T(-F^*v_0)}{\|P_T(-F^*v_0)\|} \in T_A(x), \quad \|d_0\| = 1.
	\]
	
	On the one hand, for any $v\in K(x)$, using~\eqref{VRT++} and~\eqref{VRT6},
	\[
	\langle F^*v, d_0\rangle \ge \min_{d\in D} \langle F^*v, d\rangle = p(v) \ge \min_{v'\in K(x)} p(v') = \mu(x).
	\]
	Combining with~\eqref{VRT2} yields
	\begin{equation}\label{VRT11}
		\theta'(x;d_0) = \min_{v\in K(x)} \langle F^*v, d_0\rangle \ge \mu(x).
	\end{equation}
	
	On the other hand, for the specific $v_0$, using the orthogonal decomposition $-F^*v_0 = P_T(-F^*v_0) + P_{T^\circ}(-F^*v_0)$ and $\langle P_T(-F^*v_0), P_{T^\circ}(-F^*v_0)\rangle = 0$,
	\[
	\langle F^*v_0, d_0\rangle = \langle -P_T(-F^*v_0) - P_{T^\circ}(-F^*v_0),\, \frac{P_T(-F^*v_0)}{\|P_T(-F^*v_0)\|} \rangle 
	= -\|P_T(-F^*v_0)\| = \mu(x).
	\]
	Hence
	\[
	\theta'(x;d_0) = \min_{v\in K(x)} \langle F^*v, d_0\rangle \le \langle F^*v_0, d_0\rangle = \mu(x).
	\]
	Together with~\eqref{VRT11}, we conclude $\theta'(x;d_0) = \mu(x)$.
\end{proof}

 \begin{theorem}  \label{TXEvQ} 
 	Let \(A \subseteq X\) be a nonempty  compact convex set. Assume that the unit-sphere minimum directional derivative \(\varphi\) satisfies
 	\[
 	\sup_{x \in A \setminus E_w} \varphi(x) < 0.
 	\]
 	Then $\theta$ has an error bound on $A$.
 \end{theorem}

 \begin{proof}   
 Due to
 	\[
 	\sup_{x \in A \setminus E_w} \varphi(x) < 0,
 	\]
   there exists $c > 0$ such that
 	\[
 	\sup_{x \in A \setminus E_w} \inf_{\substack{d\in T_A(x)\\\|d\|\le 1}}\theta'(x;d)  =  \sup_{x \in A \setminus E_w} \varphi(x) < -c.
 	\]
 	This means that for every $x \in A \setminus E_w$, there exists $d_x \in T_A(x)$ with $\|d_x\| = 1$ such that 
 	\[
 	\lim_{t \downarrow 0} \frac{\theta(x + t d_x) - \theta(x)}{t} = \theta'(x; d_x) < -c, 
 	\]
 	which implies that there exists $\delta_x > 0$ such that for all $t \in (0, \delta_x)$,
 	\begin{equation}\label{EtyhB7}
 		\frac{\theta(x + t d_x) - \theta(x)}{t} < -c.
 	\end{equation}
 Due to $d_x \in T_A(x)$, there exist a sequence $t_k \downarrow 0$ and $d_k \to d_x$ such that
 	\[
 	x_k := x + t_k d_k \in A, \quad \forall k \in \mathbb{N}.
 	\]
 	Using the Lipschitz continuity of $\theta$ (with constant $L = \|F\|$),
 	\[
 	|\theta(x_k) - \theta(x + t_k d_x)| \le L \|x_k - (x + t_k d_x)\| = L t_k \|d_k - d_x\|.
 	\]
 	Hence
 	\[
 	\frac{\theta(x_k) - \theta(x)}{t_k}
 	\le \frac{\theta(x + t_k d_x) - \theta(x)}{t_k} + L \|d_k - d_x\|.
 	\]
 	For sufficiently large $k$, we have $t_k < \delta_x$ and $L \|d_k - d_x\| \le c/2$. Combining with (\ref{EtyhB7}), we obtain
 	\[
 	\frac{\theta(x_k) - \theta(x)}{t_k} \le -c + \frac{c}{2} = -\frac{c}{2},
 	\]
 	and so
 	\begin{equation}\label{EtyhB8}
 		\theta(x_k) \le \theta(x) - \frac{c}{2} t_k.
 	\end{equation}
 	Since $\|d_x\| = 1$ and $d_k \to d_x$, for sufficiently large $k$ we have
 	\[
 	\|x_k - x\| = t_k \|d_k\| \le t_k (\|d_x\| + \|d_k - d_x\|) \le 2 t_k .
 	\]
 	Thus $t_k \ge \frac{1}{2}\|x_k - x\|$. Substituting into (\ref{EtyhB8}) yields
 	\[
 	\theta(x_k) \le \theta(x) - \frac{c}{4} \|x_k - x\|.
 	\]
 	Therefore, for each $x \in A \setminus E_w$, there exists $y = x_k \in A$, $y \neq x$, such that
 	\[
 	\theta(y) \le \theta(x) - \frac{c}{4} \|y - x\|.
 	\]
 	By Theorem \ref{XLWCJK}, this implies that  $\theta$ has an error bound on $A$.
 \end{proof}
 
 \begin{theorem} \label{TRgfd} 
 	Let $X$ be a  finite-dimensional   space and  $\bar{x} \in E_w$. Then  $	\varphi (\bar x) > 0$ if and only if  	$\theta$ has a local sharp minimum at   $\bar x\in E_w$,	i.e., there exist constants $\kappa > 0$ and $\delta > 0$ such that
 	\[
 	\theta(x) \ge \kappa \|x - \bar x\| \qquad \forall x \in A \cap B(\bar x, \delta).
 	\]
  	Moreover, letting $\beta = \sup \left\{ \kappa > 0 : \exists \delta > 0,\ \theta(x) \ge \kappa \|x - \bar x\|,\ \forall x \in A \cap B(\bar x, \delta) \right\}$, we have $\beta = \varphi (\bar x)$.
 \end{theorem}

\begin{proof}    $\Rightarrow$.   Assume $\varphi (\bar x) > 0$. By positive homogeneity, for all $d \in T_A(\bar x)$, $\theta'(\bar x; d) \ge \varphi (\bar x) \|d\|$.
 
 Fix any $\epsilon >0$ with $\epsilon <	\varphi (\bar x)$. We prove that there exists $\delta > 0$ such that
 \begin{equation}\label{DvfB1}
 	\theta(x) \ge \left( 	\varphi (\bar x)-\epsilon \right) \|x - \bar x\| \qquad \forall x \in A \cap B(\bar x, \delta).
 \end{equation}

 If (\ref{DvfB1}) is not true, then for each positive integer $m$ there exists $x_m \in A$ with
 \[
 0 < \|x_m - \bar x\| < \frac{1}{m} \quad\text{and}\quad \theta(x_m) < \left( 	\varphi (\bar x) - \epsilon \right) \|x_m - \bar x\|.
 \]
 Set $t_m := \|x_m - \bar x\|$ and $d_m := (x_m - \bar x)/t_m$; then    $t_m   \to 0  $,     $\|d_m\| = 1$ and
 \begin{equation}\label{DvfB2}
 	\frac{\theta(x_m)}{t_m} < 	\varphi (\bar x) -\epsilon , \quad \forall m    \in \mathbb{N}.  
 \end{equation}
 
 Since $X$ is finite-dimensional, the unit sphere is compact, so there exists a convergent subsequence (still denoted $d_m$) with $d_m \to d$ and $\|d\| = 1$. By definition of the tangent cone (as $x_m = \bar x + t_m d_m \in A$, $t_m \to 0$),     we have      $d \in T_A(\bar x)$.
 
 Using Theorem   \ref{YL:thetaprop2} (ii),  we get
 \[
 \theta(x_m) = \theta(\bar x + t_m d_m) \ge \theta(\bar x + t_m d) - \|F\|  t_m \|d_m - d\|.
 \]
 Dividing by $t_m$,
 \[
 \frac{\theta(x_m)}{t_m} \ge \frac{\theta(\bar x + t_m d)}{t_m} - \|F\| \|d_m - d\|.
 \]
 Let $m \to \infty$, the first term on the right tends to $\theta'(\bar x; d)$, and the second term tends to $0$. Taking the limit inferior yields
 \[
 \liminf_{m \to \infty} \frac{\theta(x_m)}{t_m} \ge \theta'(\bar x; d).
 \]
 This together with  (\ref{DvfB2})  implies that 
 \begin{equation}\label{DvfB3}
 \theta'(\bar x; d) \le 	\varphi (\bar x) -\epsilon. 
  \end{equation}
  By the definition of $\varphi (\bar x)$,  we have $\theta'(\bar x; d) \ge 	\varphi (\bar x)$,  which contradicts (\ref{DvfB3}). 
  Hence (\ref{DvfB1}) holds.
 From (\ref{DvfB1}) we have $\beta \ge 	\varphi (\bar x) -\epsilon$. Since $\epsilon>0$ is arbitrary, $\beta \ge 	\varphi (\bar x)$.

 $ \Leftarrow $. 
 Assume that there exist $\kappa > 0$   and  $\delta > 0$ such that
 \begin{equation}\label{DvfB4}
\theta(x) \ge \kappa \|x - \bar x\|,  \quad  \forall x \in A \cap B(\bar x, \delta).
 \end{equation}
  Take any nonzero direction $d \in T_A(\bar x)$.  
 Then there exist sequences $t_k \downarrow 0$ and $d_k \to d$ such that $x_k := \bar x + t_k d_k \in A$. For sufficiently large $k$, $\|x_k - \bar x\| \le t_k \|d_k\| < \delta$, so  it follows from  (\ref{DvfB4})  that
 \begin{equation}\label{DvfB5}
 \theta(x_k) \ge \kappa \|x_k - \bar x\| = \kappa t_k \|d_k\|.  
\end{equation}
 By Theorem   \ref{YL:thetaprop2} (ii),  we have
 \[
 |\theta(x_k) - \theta(\bar x + t_k d)| \le \|F\| \|t_k d_k - t_k d\| = \|F\| t_k \|d_k - d\|,
 \]
 and so
 \[
 \theta(\bar x + t_k d) \ge \theta(x_k) - \|F\| t_k \|d_k - d\|.
 \]
  This together with   (\ref{DvfB5})  implies that 
 \[
 \frac{\theta(\bar x + t_k d)}{t_k} \ge \kappa \|d_k\| - \|F\| \|d_k - d\|.
 \]
 Let $k \to \infty$; the right-hand side tends to $\kappa \|d\|$ (because $\|d_k\| \to \|d\|$, $\|d_k - d\| \to 0$). The limit of the left-hand side is exactly the directional derivative $\theta'(\bar x; d)$. Hence
 \[
 \theta'(\bar x; d) \ge \kappa \|d\|,   \quad  \forall d \in T_A(\bar x).
 \]
  In particular, on the unit sphere $S = \{d \in T_A(\bar x) : \|d\| = 1\}$, $\theta'(\bar x; d) \ge \kappa$. Taking the infimum gives
 \[
	\varphi (\bar x) = \inf_{d \in S} \theta'(\bar x; d) \ge \kappa > 0.
 \]
 By definition of $\beta$, we also have $	\varphi (\bar x) \ge \beta$. Hence $\beta = 	\varphi (\bar x)$.
  \end{proof}

 Let  $\bar{x} \in E_w$.  The optimal local error bound constant is defined by 
 \[
 \tau^*(\bar{x}) := \inf\{ \tau > 0 : \exists \delta > 0,\; \forall x \in A \cap B(\bar{x}, \delta),\; d(x, E_w) \le \tau \theta(x) \},
 \]
 (with the convention $\inf \varnothing = +\infty$).

 \begin{theorem}  \label{TYXWs}      Assume that $X$ is  a  finite-dimensional   space and  $\bar{x} \in E_w$. 
 	Then the following hold:
 	\begin{enumerate}
 		\item[(i)]   If $	\varphi(\bar{x}) > 0$, then $\theta$ has a  local error bound at $\bar x$ and
 		\[
 		\tau^*(\bar{x}) = \frac{1}{	\varphi(\bar{x})} < +\infty.
 		\]
 				\item[(ii)] 	If $	\varphi(\bar{x}) = 0$ and there exists a direction $d_0\in T_A(\bar{x})\setminus T_{E_w}(\bar{x})$ with $\theta'(\bar{x};d_0)=0$, then $\theta$ does \emph{not} admit a local error bound at $\bar{x}$, i.e., $\tau^*(\bar{x}) = +\infty$.
 		 	\end{enumerate}
 	\end{theorem}
 
  \begin{proof}   (i).  Assume that $	\varphi(\bar{x}) > 0$.    Fix $\eta \in (0, 	\varphi(\bar{x}))$.  It follows from  (\ref{DvfB1})  that    there exists $\delta_0 > 0$ such that
  	 \begin{equation}\label{Eyhub1}
  \theta(x) \ge (\varphi(\bar{x}) - \eta) \|x - \bar{x}\|,  \quad	\forall x \in A \cap B(\bar{x}, \delta_0). 
  \end{equation}
  Due to $\bar{x} \in E_w$ and (\ref{Eyhub1}), we have 
  $$ \theta(x) \ge (\varphi(\bar{x}) - \eta) \|x - \bar{x}\|  \ge  (\varphi(\bar{x}) - \eta) d(x, E_w) ,  \quad	\forall x \in A \cap B(\bar{x}, \delta_0), $$
  and so 
  $$ \frac{1}{	\varphi(\bar{x}) - \eta} \, \theta(x)    \ge   d(x, E_w) ,  \quad	\forall x \in A \cap B(\bar{x}, \delta_0). $$
 This yields  that $\theta$ has a  local error bound at $\bar x$ and 
 \[
 \tau^*(\bar{x}) \le \frac{1}{	\varphi(\bar{x}) - \eta}.
 \]
 As $\eta \in (0, 	\varphi(\bar{x}))$ is arbitrary, letting $\eta \downarrow 0$ gives
   \begin{equation}\label{Eyhub2}
 \tau^*(\bar{x}) \le \frac{1}{	\varphi(\bar{x})}.  
  \end{equation}
  
  We show that 
   \begin{equation}\label{Eyhub3}
  	E_w \cap B(\bar x, \delta_0) = \{\bar x\}.
  \end{equation}
  	Indeed, if there exists $x' \in E_w \cap B(\bar x, \delta_0)$ with $x' \neq \bar x$, then by  (\ref{Eyhub1}),
  \[
  0 = \theta(x') \ge  (	\varphi(\bar{x}) - \eta) \|x' - \bar x\| > 0,
  \]
which is  a contradiction. Therefore,  (\ref{Eyhub3}) holds.
  
   Set $\delta' = \delta_0 /2$. Take any $x \in B(\bar{x}, \delta')$. For any $y \in E_w \setminus \{\bar{x}\}$, we conclude from  (\ref{Eyhub3})  that $\|y - \bar{x}\| \ge \delta_0$,  and  so
  \[
  \|x - y\| \ge \|y - \bar{x}\| - \|x - \bar{x}\| \ge \delta_0 - \delta' = \delta' > \|x - \bar{x}\|.
  \]
  This means that
    \begin{equation}\label{Eyhub4}
  d(x, E_w) = \|x - \bar{x}\|,   \quad	\forall  x \in A \cap B(\bar{x}, \delta').
  \end{equation}
  
  Suppose $\tau^*(\bar{x}) < 1/	\varphi(\bar{x})$; then there exists $\tau'$ with $\tau^*(\bar{x}) \le \tau' < 1/	\varphi(\bar{x})$ and $\delta'' > 0$ such that
   \begin{equation}\label{Eyhub5}
  d(x, E_w) \le \tau' \theta(x),   \quad  \forall x \in A \cap B(\bar{x}, \delta'').
  \end{equation}
  Take $\delta = \min\{\delta', \delta ''\}$.  
  It follows from  (\ref{Eyhub4}) and (\ref{Eyhub5})   that
   \begin{equation}\label{Eyhub6}
  	\theta(x) \ge \frac{1}{\tau'} \|x - \bar{x}\|,   \quad	\forall  x \in A \cap B(\bar{x}, \delta).
  \end{equation}
  Since $X$ is finite-dimensional, we obtain that  $S := \{d \in T_A(\bar{x}) : \|d\| = 1\}$ is compact. By Theorem   \ref{JBXZDS}  (ii),  the map $d \mapsto \theta'( x; d)$ is   continuous.  Then there exists  $d_0 \in S$  such that $\theta'(\bar{x}; d_0) = 	\varphi(\bar{x})$.  Due to $d_0 \in  T_A(\bar{x}) $, 
   there exist sequences $t_k \downarrow 0$, $d_k \to d_0$ such that $x_k = \bar{x} + t_k d_k \in A$. For sufficiently large $k$, $x_k \in B(\bar{x}, \delta)$, and applying (\ref{Eyhub6})  yields $\theta(x_k) \ge \frac{1}{\tau'} \|x_k - \bar{x}\| = \frac{1}{\tau'} t_k \|d_k\|$, and so
  \[
  \frac{\theta(x_k)}{t_k} \ge \frac{\|d_k\|}{\tau'}.
  \]
  Since $\|d_k\| \to 1$, taking the limit inferior gives 
   \begin{equation}\label{Eyhub7}
   \liminf_{k \to \infty} \frac{\theta(x_k)}{t_k} \ge \frac{1}{\tau'}.
   \end{equation}
  On the other hand, by Theorem   \ref{YL:thetaprop2} (ii),  we have
  \[
  \frac{\theta(x_k)}{t_k} \le \frac{\theta(\bar{x} + t_k d_0)}{t_k} + \|F\|   \|d_k - d_0\|,
  \]
  and taking the limit superior, using the definition of directional derivative and $\|d_k - d_0\| \to 0$, gives
  \[
  \limsup_{k \to \infty} \frac{\theta(x_k)}{t_k} \le \theta'(\bar{x}; d_0) = 	\varphi(\bar{x}).
  \]
  Combining this with  (\ref{Eyhub7}), we get
  \[
  \frac{1}{\tau'} \le 	\varphi(\bar{x}) \quad \Longrightarrow \quad \tau' \ge \frac{1}{	\varphi(\bar{x})},
  \]
  contradicting $\tau' < 1/	\varphi(\bar{x})$.  Therefore, we must have $\tau^*(\bar{x}) \ge 1/	\varphi(\bar{x})$. Together with (\ref{Eyhub2}) we obtain $\tau^*(\bar{x}) = 1/	\varphi(\bar{x})$.

  (ii).    Suppose that $	\varphi(\bar{x}) = 0$ and there exists a direction $d_0\in T_A(\bar{x})\setminus T_{E_w}(\bar{x})$ with $\theta'(\bar{x};d_0)=0$. 
  Since $d_0\notin T_{E_w}(\bar{x})$,    there exists  $\alpha>0$ and $t_0>0$ such that
\begin{equation}\label{EmkiX1}
 d(\bar{x}+td_0,\,E_w) \ge \alpha t,    \quad   \forall t\in(0,t_0).
 \end{equation}
It follows from $d_0\in T_A(\bar{x})$ that    there exist sequences $\{t_k\}\downarrow 0$ and $\{d_k\}\to d_0$ such that
  \[
  x_k := \bar{x}+t_k d_k \in A\qquad   \forall k\in\mathbb{N}.  
  \]
   In view of Theorem   \ref{YL:thetaprop2} (ii),  we have
  \[
  \theta(x_k) = \theta(\bar{x}+t_k d_k) \le \theta(\bar{x}+t_k d_0) + L t_k \|d_k - d_0\|,
  \]
and so
  \[
  \frac{\theta(x_k)}{t_k} \le \frac{\theta(\bar{x}+t_k d_0)}{t_k} + L \|d_k - d_0\|.
  \]
  Let $k\to\infty$.  By   $\theta'(\bar{x};d_0)=0$ and $d_k\to d_0$, we get
  \[
  \limsup_{k\to\infty} \frac{\theta(x_k)}{t_k} \le 0 + 0 = 0.
  \]
  This together with $\theta(x_k) \ge 0$   implies that
  \begin{equation}\label{EmkiX2}
  	\lim_{k\to\infty} \frac{\theta(x_k)}{t_k} = 0.
  \end{equation}
    It is clear that
   \begin{equation}\label{EmkiX3}
  d(x_k, E_w) \ge d(\bar{x}+t_k d_0, E_w) - \|x_k - (\bar{x}+t_k d_0)\| = d(\bar{x}+t_k d_0, E_w) - t_k \|d_k - d_0\|.
   \end{equation}
  Take $k$ large enough so that $t_k < t_0$ and $\|d_k - d_0\| \le \alpha/2$.    It follows from  (\ref{EmkiX1})  that  $d(\bar{x}+t_k d_0, E_w) \ge \alpha t_k$.  Combining this with  (\ref{EmkiX3}), we get
   \begin{equation}\label{EmkiX4}
  d(x_k, E_w) \ge \alpha t_k - t_k\frac{\alpha}{2} = \frac{\alpha}{2} t_k > 0.  
   \end{equation}
  
 Suppose that $\theta$ has a  local error bound at $\bar x$.
  Then  there exist $\tau>0$ and $\bar{\delta } >0$ such that
   \begin{equation}\label{EmkiX5}
d(x,E_w)\le \tau\theta(x),   \quad  \forall x\in A\cap B(\bar{x},\bar{\delta }).
 \end{equation}
 Choose $k$ large enough so that $x_k\in B(\bar{x},\bar{\delta })$.   We conclude from  (\ref{EmkiX4})  and  (\ref{EmkiX5})  that
  \[
  \frac{\alpha}{2} t_k \le d(x_k, E_w) \le \tau \theta(x_k) .
  \]
  This together with    (\ref{EmkiX2})    yields   that 
  $$ 0 < {\alpha  \over 2 \tau} \le {{   \theta \left( {{x_k}} \right)} \over {{t_k}}} \to 0,$$
which leads to a contradiction.  Therefore, $\theta$ does not have a local error bound at $\bar{x}$, i.e., $\tau^*(\bar{x}) = +\infty$.
     \end{proof}

Theorem  \ref{TYXWs} (ii) shows: if $	\varphi(\bar{x})=0$ and, additionally, there exists a direction $d_0\in T_A(\bar{x})\setminus T_{E_w}(\bar{x})$ with $\theta'(\bar{x};d_0)=0$, then a local error bound certainly fails.   The following counterexample demonstrates that the additional condition is essential.

 \begin{counterexample}   
 	Let $X=\mathbb{R}^2$, $Y=\mathbb{R}^2$ with the Euclidean norm.  
 	Set $A=[0,1]^2$, $C=\mathbb{R}^2_+=\{y=(y_1,y_2):y_1\ge0,\ y_2\ge0\}$, and let $F:X\to Y$ be the identity map ($F(x)=x$).  
  Weakly efficient solution set is clearly
 	\[
 	E_w = \{(x_1,x_2)\in[0,1]^2 : x_1=0\ \text{or}\ x_2=0\},
 	\]
 	and the merit function $\theta$ is defined by
 	\[
 	\theta(x) = \sup_{a\in A} \inf_{y^*\in S(C^+)} \langle y^*, x - a\rangle .
 	\]
 	Since $C^+ = C = \mathbb{R}^2_+$ and $S(C^+) = \{y^*\in\mathbb{R}^2_+ : \|y^*\|=1\}$, a direct computation gives
 	\[
 	\theta(x) = \min\{x_1,\,x_2\}, \qquad \forall\, x\in C .
 	\]
 	
 	Take $\bar{x}=(0,0)\in E_w$.  Clearly $T_A(\bar{x}) = \mathbb{R}^2_+ = C$.  
 	For any direction $d=(d_1,d_2)\in T_A(\bar{x})$,
 	\[
 	\theta(\bar{x}+td) = \theta(td_1,td_2) = t\min\{d_1,d_2\},
 	\]
 	hence $\theta'(\bar{x};d) = \min\{d_1,d_2\}$.  
 	Consequently, the infimum of the directional derivative over the unit sphere $S = \{d\in C : \|d\|=1\}$ is
 	\[
 	\varphi(\bar{x}) = \inf_{d\in C,\ \|d\|=1} \min\{d_1,d_2\} = 0 .
 	\]
 	
 	For any $x=(x_1,x_2)\in A$, the distance to the solution set is $d(x,E_w) = \min\{x_1,\,x_2\}$, which coincides exactly with $\theta(x)$.  
 	Hence $d(x,E_w) = \theta(x)$ for all $x\in A$, yielding a local (in fact global) error bound at $\bar{x}$ with constant $\tau=1$.  
 	The optimal local error bound constant is $\tau^*(\bar{x}) = 1 < +\infty$.
 	
 	A straightforward computation shows that
 	\begin{equation}\label{Ecyta1}
 		T_{E_w}(\bar{x}) = \{(d_1,d_2): d_1\ge0,\; d_2=0\} \cup \{(d_1,d_2): d_1=0,\; d_2\ge0\}.
 	\end{equation}
 	
 	Finally, we verify that the condition ``there exists a direction $d_0\in T_A(\bar{x})\setminus T_{E_w}(\bar{x})$ with $\theta'(\bar{x};d_0)=0$'' is not satisfied.  
 	Suppose, to the contrary, that such a direction $\bar{d}=(d_1,d_2)\in T_A(\bar{x})\setminus T_{E_w}(\bar{x})$ exists with $\theta'(\bar{x};\bar{d})=0$.  
 	Since $\bar{d}\in T_A(\bar{x})$, we have $d_1\ge0$ and $d_2\ge0$.  From $\theta'(\bar{x};\bar{d}) = \min\{d_1,d_2\}=0$, it follows that $d_1=0$ or $d_2=0$.  By \eqref{Ecyta1}, this forces $\bar{d}\in T_{E_w}(\bar{x})$, contradicting the choice of $\bar{d}$.  
 	Thus no such direction exists.
 	
Thus, the hypothesis of Theorem~\ref{TYXWs}(ii) is not satisfied, yet a local error bound still exists. This demonstrates that the additional condition in Theorem~\ref{TYXWs}(ii) cannot be removed.
 \end{counterexample}

\section{Conclusion}

This paper has developed a systematic theory of directional derivatives and error bounds for the concave merit function arising in linear vector optimization. The central philosophy of our approach is to exploit the dual representation of the oriented distance function and the variational structure of the merit function itself, thereby establishing a rigorous bridge between analytic properties of the merit function and the geometric features of the weakly efficient solution set.

Our analysis demonstrates that the directional derivatives of the merit function encode fundamental geometric information about the solution set. The explicit connection between the zero-directional-derivative cone, the tangent cone of the feasible set, and the tangent cone of the solution set provides a coherent framework for understanding how first-order information of the merit function reflecting the local structure of the weakly efficient solution set. This perspective unifies several classical concepts from nonsmooth analysis, variational geometry, and multiobjective optimization, and reveals the merit function as a natural scalarization device that preserves essential geometric features of the original vector problem.

The error bound theory developed in this work establishes a comprehensive equivalence framework that consolidates and extends many classical results from the literature. The characterizations, ranging from linear regularity and global slope conditions to perturbation stability and sublevel-set Hausdorff continuity, highlight the multifaceted nature of the error bound property. The unified treatment shows that seemingly disparate conditions—geometric, analytic, and stability-based—are in fact equivalent manifestations of the same underlying regularity phenomenon. Moreover, the characterization via the unit-sphere minimal directional derivative provides a conceptually simple and computationally meaningful criterion for global error bounds, emphasizing the role of uniform descent away from the solution set.

From a broader perspective, the results obtained here contribute to the growing body of research that seeks to merge the theory of variational analysis with vector optimization. The interplay between directional derivatives, error bounds, and tangent cone geometry exemplifies how tools from nonsmooth analysis can be effectively adapted to the multiobjective setting. The explicit formulas and quantitative estimates established in this paper provide a theoretical foundation for algorithmic developments, particularly for the convergence analysis of descent methods and proximal algorithms tailored to vector optimization problems in Banach spaces.

Several promising directions for future investigation emerge from this work. 

\begin{itemize}
	\item[(i)] Extending the directional derivative analysis to nonlinear objective maps, including smooth and nonsmooth cases, would broaden the applicability of the theory to a wider class of vector optimization problems. This would require developing appropriate chain rules and generalized derivative concepts for the composition of the oriented distance function with nonlinear mappings.
	
	\item[(ii)] Relaxing the compactness and convexity assumptions on the feasible set \(A\) remains a significant challenge. Non-convex or unbounded feasible sets arise naturally in many applications, and extending the error bound characterizations to such settings would require new techniques from variational analysis.
	
	\item[(iii)] The counterexample provided in Section~6 reveals that the local error bound theory when \(\varphi(\bar{x})=0\) is more subtle than previously recognized. A complete characterization of local error bounds in this critical case remains an open problem and deserves further investigation.
	
	\item[(iv)] It would be of interest to extend the present framework—directional derivatives, error bounds, and tangent cone characterizations—from vector optimization to set optimization problems. This can be achieved by employing generalized versions of the Hiriart-Urruty oriented distance function (see Ha \cite{Ha2018}, Han \cite{Han2022}) to define suitable merit functions for set optimization problems, and then investigating their directional derivatives and error bounds in the spirit of this paper.
	
	\item[(v)] The computational realization of the merit function \(\theta\) and the implementation of gradient-like algorithms that exploit the directional derivative structure deserve further investigation, particularly in infinite-dimensional Banach spaces and in problems with polyhedral or conic constraints. Developing efficient numerical methods based on the error bound characterizations presented here would bridge the gap between theory and practice.
	
\end{itemize}

In conclusion, this paper provides a comprehensive theoretical framework that connects directional derivatives, error bounds, and geometric properties of the solution set in vector optimization. The results not only deepen our understanding of the structure of multiobjective optimization problems but also offer practical tools for algorithmic development and convergence analysis. We hope that this work will stimulate further research at the intersection of variational analysis and vector optimization, and contribute to the ongoing development of both fields.

\bigskip
\noindent
{\bf Funding.} This work was supported by the National Natural Science Foundation of China [11801257] and the Natural Science Foundation of Jiangxi Province [20232BAB211012].

\end{document}